\documentclass[reqno,10pt]{amsart}

\usepackage{amssymb}
\usepackage{comment}
\usepackage{bbm}
\usepackage{tikz}
\usepackage[all,cmtip]{xy}
\usepackage{hyperref}
\usepackage{stmaryrd}

\usepackage{tikz}
\usetikzlibrary{matrix,arrows}

\usepackage{tikz-cd}
\usepackage{enumitem}
\usepackage{bm}

\newcounter{mt}

\newtheorem*{MainTheorem}{Main theorem}

\newtheorem{Proposition}{Proposition}[section]
\newtheorem{Definition}[Proposition]{Definition}
\newtheorem{Lemma}[Proposition]{Lemma}
\newtheorem{Theorem}[Proposition]{Theorem}
\newtheorem{Corollary}[Proposition]{Corollary}
\newtheorem{Remark}[Proposition]{Remark}

\DeclareMathOperator{\Val}{Val}

\DeclareMathOperator{\nc}{nc}

\DeclareMathOperator{\Gr}{Gr}
\DeclareMathOperator{\Geod}{Geod}

\DeclareMathOperator{\Jac}{Jac}
\DeclareMathOperator{\glob}{glob}

\DeclareMathOperator{\id}{Id}
\DeclareMathOperator{\tr}{tr}
\renewcommand{\Re}{\mathrm {Re}}
\newcommand{\p}{\mathbb{P}}
\DeclareMathOperator{\vol}{vol}

\DeclareMathOperator{\Dens}{Dens}
\DeclareMathOperator{\Span}{Span}

\DeclareMathOperator{\Image}{Image}
\DeclareMathOperator{\Supp}{Supp}
\DeclareMathOperator{\Hom}{Hom}

\DeclareMathOperator{\Stab}{Stab}
\DeclareMathOperator{\Sym}{Sym}
\DeclareMathOperator{\Ker}{Ker}

\DeclareMathOperator{\Cr}{Cr}

\DeclareMathOperator{\WF}{WF}
\DeclareMathOperator{\LC}{LC}

\DeclareMathOperator{\GL}{GL}

\DeclareMathOperator{\OO}{O}
\DeclareMathOperator{\SO}{SO}
\DeclareMathOperator{\sign}{sign}

\DeclareMathOperator{\Isom}{Isom}
\newcommand{\bi}{\mathbf i}
\newcommand{\R}{\mathbb{R}}
\newcommand{\C}{\mathbb{C}}

\makeatletter
\def\moverlay{\mathpalette\mov@rlay}
\def\mov@rlay#1#2{\leavevmode\vtop{%
		\baselineskip\z@skip \lineskiplimit-\maxdimen
		\ialign{\hfil$\m@th#1##$\hfil\cr#2\crcr}}}
\newcommand{\charfusion}[3][\mathord]{
	#1{\ifx#1\mathop\vphantom{#2}\fi
		\mathpalette\mov@rlay{#2\cr#3}
	}
	\ifx#1\mathop\expandafter\displaylimits\fi}
\makeatother

\newcommand{\largewedge}{\mbox{\Large $\wedge$}}

\newtheorem{note}{Remark} 
\def\note#1{\ifvmode\leavevmode\fi\vadjust{\vbox to0pt{\vss
			\hbox to 0pt{\hskip\hsize\hskip1em
				\vbox{\hsize3.5cm\small\raggedright\pretolerance10000
					\noindent #1\hfill}\hss}\vbox to8pt{\vfil}\vss}}}

\date{}

\begin{document}
	
	\title{Crofton formulas in pseudo-Riemannian space forms}
	
	\author{Andreas Bernig}
	\author{Dmitry Faifman}
	\author{Gil Solanes}
	
	\email{bernig@math.uni-frankfurt.de}
	\email{faifmand@tauex.tau.ac.il} 
	\email{solanes@mat.uab.cat}
	\address{Institut f\"ur Mathematik, Goethe-Universit\"at Frankfurt,
		Robert-Mayer-Str. 10, 60629 Frankfurt, Germany}
	\address{School of Mathematical Sciences, Tel Aviv University, Tel Aviv 6997801, Israel}
	\address{Departament de Matem\`atiques, Universitat Aut\`onoma de Barcelona, 08193 Bellaterra, Spain, and \newline\indent Centre de Recerca Matem\`atica, Campus de Bellaterra, 08193 Bellaterra, Spain }

	\date{}

	\begin{abstract}
		Crofton formulas on simply-connected Riemannian space forms allow to compute the volumes, or more generally the Lipschitz-Killing curvature integrals of a submanifold with corners, by integrating the Euler characteristic of its intersection with all geodesic submanifolds. We develop a framework of Crofton formulas with distributions replacing measures, which has in its core Alesker's Radon transform on valuations. We then apply this framework, and our recent Hadwiger-type classification, to compute explicit Crofton formulas for all isometry-invariant valuations on all pseudospheres, pseudo-Euclidean and pseudohyperbolic spaces. We find that, in essence, a single measure which depends analytically on the metric, gives rise to all those Crofton formulas through its distributional boundary values at parts of the boundary corresponding to the different indefinite signatures. In particular, the Crofton formulas we obtain are formally independent of signature.
	\end{abstract}
	
	\thanks{  	
		\\ A.B. was supported
		by DFG grant BE 2484/5-2\\
		D.F. was supported by ISF Grant 1750/20.
		\\
		G.S. was supported by FEDER/MICINN grant PGC2018-095998-B-I00 and the Serra H\'unter Programme}
	\maketitle
	
	\tableofcontents
	
	\section{Introduction}

	\subsection{Crofton formulas}
	
	The classical Crofton formula computes the length of a rectifiable curve $\gamma$ in $\R^2$ as 
	\begin{equation} \label{eq_crofton_2dim}
	\mathrm{Length}(\gamma)=\frac{\pi}{2} \int_{\overline{\Gr}_1(\R^2)} \#(\gamma \cap \overline L) d\overline L,
	\end{equation} 
	where $\overline{\Gr}_1(\R^2)$ is the space of lines in $\R^2$ with a rigid motion invariant measure (which is normalized in a suitable way). 
	
	A higher dimensional version states that for $M \subset \R^n$ a compact submanifold with boundary, we have 
	\begin{displaymath}
	\mu_k(M)=c_{n,k} \int_{\overline{\Gr}_{n-k}(\R^n)} \chi(M \cap \overline E)d\overline E,
	\end{displaymath}
	where $\overline{\Gr}_{n-k}(\R^n)$ is the Grassmann manifold of affine $(n-k)$-planes equipped with a rigid motion invariant measure, $\chi$ is the Euler characteristic, and $\mu_k(M)$ is the $k$-th intrinsic volume of $M$, which can be defined via Weyl's tube formula \cite{weyl_tubes}. The same formula also holds with the submanifold $M$ replaced by a compact convex body $K$, in which case the $k$-th intrinsic volume can be defined via Steiner's tube formula \cite{steiner}. We refer to \cite{klain_rota, schneider_book14} for more information about intrinsic volumes of convex bodies.
	
	More generally, we can take an arbitrary translation-invariant measure $m$ on $\overline{\Gr_{n-k}}(\R^n)$ and consider the integral
	\begin{displaymath}
	\mu(K):=\int_{\overline{\Gr_{n-k}}(\R^n)} \chi(K \cap \overline E) dm(\overline E).
	\end{displaymath}
	By the additivity of the Euler characteristic, we have 
	\begin{displaymath}
	\mu(K \cup L)+\mu(K \cap L)=\mu(K)+\mu(L)
	\end{displaymath}
	whenever $K,L, K \cup L$ are compact convex bodies, hence $\mu$ is a valuation. Clearly $\mu$ belongs to the space $\Val$ of translation-invariant valuations which are continuous with respect to the Hausdorff metric. Additionally, $\mu$ is $k$-homogeneous and even, that is invariant under $-\id$. We thus get a map 
	\begin{displaymath}
	\Cr:\mathcal M(\overline{\Gr_{n-k}}(\R^n))^{tr} \to \Val_k^+,
	\end{displaymath}
	where $\mathcal M^{tr}$ denotes the space of translation-invariant measures. 
	
	Alesker \cite{alesker_mcullenconj01} has shown that the image of this map is dense with respect to the natural Banach space topology on $\Val_k^+$. Therefore Crofton formulas are a central tool in the study of valuations and in integral geometry. 
	
	When restricted to smooth measures and valuations (see Section \ref{sec_preliminaries} for the notion of smoothness of valuations), the map $\Cr$ is in fact a surjection 
	\begin{displaymath}
	\Cr:\mathcal M^\infty(\overline{\Gr_{n-k}}(\R^n))^{tr} \twoheadrightarrow \Val_k^{\infty,+},
	\end{displaymath}
	the kernel of which coincides with the kernel of the cosine transform \cite{alesker_bernstein04}.
	
	Among the many applications in integral geometry of such formulas, we mention the construction of a basis of the space of unitarily invariant valuations on $\C^n$ by Alesker \cite{alesker03_un}, the interpretation of Alesker's product of smooth and even valuations in terms of Crofton measures \cite{bernig07}, the Holmes-Thompson intrinsic volumes of projective metrics \cite{alvarez_fernandes}. Applications outside integral geometry include isoperimetric inequalities in Riemannian geometry \cite{croke84}, symplectic geometry \cite{oh90}, systolic geometry \cite{treibergs85}, algebraic geometry \cite{akhiezer_kazarnovskii} and more. Crofton formulas are employed also outside of pure mathematics, in domains such as microscopy and stereology, see \cite{Kiderlen_Vedel_Jensen}.

	Crofton formulas do not only exist on flat spaces, but also on manifolds. In this case, we need a family of sufficiently nice subsets, endowed with a measure. Then the Crofton integral is given by the integral of the Euler characteristic of the intersection with respect to the measure. Under certain conditions which are given in \cite{fu_alesker_product}, it yields a smooth valuation on the manifold in the sense of Alesker \cite{alesker_val_man2}.  
	
	On spheres and hyperbolic spaces, a natural class of subsets are the totally geodesic submanifolds of a fixed dimension, endowed with the invariant measure. On the $2$-dimensional unit sphere, we have a formula similar to \eqref{eq_crofton_2dim}, with affine lines replaced by equators. This formula is the main ingredient in the proof of the F\'ary-Milnor theorem that the total curvature of a knot in $\R^3$ is bigger than $4\pi$ if the knot is non-trivial. 
	
	In higher dimensions, the formula becomes slightly more complicated. On the $n$-dimensional unit sphere we have 
	\begin{equation} \label{eq_spherical_crofton}
	\int_{\Geod_{n-k}(S^{n})} \chi(M \cap E)dE=\sum_j \frac{1}{\pi \omega_{k+2j-1}} \binom{-\frac{k}{2}}{j} \mu_{k+2j}(M).
	\end{equation}
Here $\Geod_{n-k}(S^{n})$ denotes the totally geodesic submanifolds of dimension $(n-k)$, $\mu_j(M)$ is the $j$-th intrinsic volume of $M$ (which can be defined as the restriction of the $j$-th intrinsic volume on $\R^{n+1}$ under the isometric embedding $S^n \hookrightarrow \R^{n+1}$), and $\omega_n$ denotes the volume of the $n$-dimensional unit ball. A similar formula holds on hyperbolic space. See \cite{fu_wannerer, fu_barcelona} for more on the integral geometry of real space forms. 
	
	Moving on to Lorentzian signature, few results are available. The main challenge to overcome is the non-compactness of the isotropy group, which in general renders the Crofton integral divergent. Some special Crofton-type formulas in Lorentzian spaces of constant curvature, applicable under certain rather restrictive geometric conditions, appeared in \cite{birman84, langevin_chaves_bianconi03, solanes_teufel,  ye_ma_wang16}.
	\subsection{Results}
	
	We are going to prove Crofton formulas on flat spaces, spheres and hyperbolic spaces of arbitrary signatures. Let us recall the definition of these manifolds, referring to  \cite{oneill,wolf61} for more information. 
	\begin{Definition}\label{def:space_forms}
		\begin{enumerate}
			\item The \emph{pseudo-Euclidean space of signature $(p,q)$} is $\R^{p,q}=\R^{p+q}$ with the quadratic form $Q=\sum_{i=1}^p dx_i^2-\sum_{i=p+1}^{p+q} dx_i^2$. 
			\item The \emph{pseudosphere of signature $(p,q)$ and radius $r>0$} is 
			\begin{displaymath}
			S_r^{p,q}=\{v \in \R^{p+1,q}:Q(v)=r^2\},
			\end{displaymath}
			equipped with the induced pseudo-Riemannian metric. Its sectional curvature equals $\sigma=\frac{1}{r^2}$.
			The pseudosphere $S_1^{n,1}\subset \R^{n+1,1}$ is called \emph{de Sitter space}. 
			\item The \emph{pseudohyperbolic space of signature $(p,q)$ and radius $r>0$} is 
			\begin{displaymath}
			H_r^{p,q}=\{v \in \R^{p,q+1}:Q(v)=-r^2\}.
			\end{displaymath}
			Its sectional curvature equals $\sigma=-\frac{1}{r^2}$. The pseudohyperbolic space $H_1^{n,1}$ is called the \emph{anti-de Sitter space}.
		\end{enumerate}
	\end{Definition}
	We will colloquially call these spaces \emph{generalized pseudospheres}. The isometry groups of generalized pseudospheres are given by 
	\begin{align*}
	\Isom(\R^{p,q}) & \cong \overline{\OO}(p,q)=\OO(p,q) \ltimes \R^{p,q},\\
	\Isom(S_r^{p,q}) & \cong \OO(p+1,q),\\
	\Isom(H_r^{p,q}) & \cong \OO(p,q+1).
	\end{align*}
	In each case the action is transitive, and the isotropy group is conjugate to $\OO(p,q)$. These spaces are \emph{isotropic} in the sense that the isotropy group acts transitively on the level sets of the metric in the tangent bundle.

	\begin{Definition}
		A complete connected pseudo-Riemannian manifold of constant sectional curvature is called a \emph{space form}.
	\end{Definition}
	Up to taking connected components and universal coverings, any space form is a generalized pseudosphere (cf. \cite[Chapter 8, Corollary 24]{oneill}).
	
	On a generalized pseudosphere $M$, we will formulate Crofton formulas using the space $\Geod_{n-k}(M)$ of totally geodesic subspaces. However, there is no isometry-invariant Radon measure on this space. Therefore we will use an isometry-invariant \emph{generalized} measure (also called distribution). This causes some technical problems, as the function that we want to integrate is not smooth. Nevertheless, in many cases the integral can still be evaluated. The result is not a valuation anymore, but a \emph{generalized} valuation in the sense of Alesker \cite{alesker_val_man4}. The Crofton map is then a map 
	\begin{displaymath}
	\Cr:\mathcal M^{-\infty}(\Geod_{n-k}(M)) \to \mathcal V^{-\infty}(M).
	\end{displaymath}

In the Riemannian case, any isometry invariant valuation admits an invariant Crofton measure. The corresponding statement in other signatures is also true, but much harder to prove. The second named author proved in \cite{faifman_crofton} the statement first for certain signatures by an explicit, but difficult, computation and then used the behavior of Crofton formulas under restrictions and projections to handle the general case.

 Furthermore, with the exception of Riemannian and Lorentzian signatures, the space of isometry-invariant generalized measures is of greater dimension than the space of isometry-invariant valuations. Thus we are forced to choose a distribution, and must take care to avoid the kernel of the Crofton map.
	
 Using results by Muro \cite{muro99} on analytic families of homogeneous generalized functions on the space of symmetric matrices, one can construct such an invariant generalized measure on $\Geod_{n-k}(M)$. We construct a particular generalized measure $m_k$ with the distinguishing property that it behaves well under restrictions of Crofton measures (see subsection \ref{sec:restriction_crofton}), and is independent (in a precise sense) of signature and dimension.
 
There is some freedom in the normalization of a Crofton measure. We choose the normalization in such a way that the first coefficient in the Crofton formulas will always be $1$. 
 
 Our Crofton formula will evaluate 
\begin{displaymath}
	\Cr_k^M:=\Cr_M(\widehat m_k),
	\end{displaymath}
	where $\sigma \neq 0$ is the curvature of $M$ and $\widehat m_k:=\pi\omega_{k-1}\sqrt{\sigma^{-1}}^k m_k$. The flat case $\sigma=0$ appears through a careful limiting procedure. 
	
	The right hand side of the Crofton formula will be expressed in terms of the recently introduced intrinsic volumes on pseudo-Riemannian manifolds \cite{bernig_faifman_solanes}, which are complex-valued generalized valuations on $M$. They satisfy a Hadwiger-type classification \cite{bernig_faifman_solanes_part2}, which allows us to use the template method to compute the coefficients in this formula. However, the resulting computations lead to distributional integrals on the space of symmetric matrices, that can be evaluated directly essentially only for the Lorentzian signature. To conclude the general case, we use techniques of meromorphic continuation and distributional boundary values of analytic functions. 
	
Due to the functorial properties of the constructed Crofton distribution mirroring those of the intrinsic volumes, namely their adherence to the Weyl principle, the resulting Crofton formulas are signature independent. Remarkably, they are also holomorphic; i.e. they involve the intrinsic volumes only and not their complex conjugates.

	\begin{MainTheorem}
		Let $M$ be a generalized pseudosphere of sectional curvature $\sigma$. Then
		\begin{displaymath}
		\Cr_k^M= \sum_{j \geq 0} \frac {\omega_{k-1}} {\omega_{k+2j-1}} {-\frac k 2 \choose j}\sigma^{j} \mu_{k+2j}.
		\end{displaymath}
	\end{MainTheorem}
	
	The Crofton formulas should be understood formally, namely as the correspondence between distributions on the Grassmannian and the intrinsic volumes through an abstractly defined Crofton map. However they can also be interpreted as explicit Crofton-type formulas applicable to sufficiently nice subsets of the generalized pseudospheres. 
	
By a strictly convex subset of a non-flat generalized pseudosphere $M^m$ we mean its intersection with a strictly convex cone in $\R^{m+1}$, with $M\subset\R^{m+1}$ embedded as in Definition \ref{def:space_forms}. For the Riemannian round sphere and hyperbolic space, this coincides with the standard definition of strict convexity.
	
	\begin{Corollary}
	 Let $A \subset M$ be a smooth and strictly convex domain in $M$. Then the generalized measure $\widehat m_k$ can be applied to the function $E \mapsto \chi(A \cap E), E \in \Geod_{n-k}(M)$, and
		\begin{displaymath}
		 \int_{\Geod_{n-k}(M)} \chi(A \cap E) d\widehat m_k(E)=\sum_{j \geq 0} \frac{\omega_{k-1}} {\omega_{k+2j-1}} {-\frac k 2 \choose j}\sigma^{j} \mu_{k+2j}(A).
		\end{displaymath}
	\end{Corollary}

	Note that the spherical Crofton formula \eqref{eq_spherical_crofton} is a special case of our theorem. We also note that we will prove a slightly more general statement in Corollary \ref{cor:can_compute_on_LC} using the notion of LC-regularity from \cite{bernig_faifman_solanes}.

	\subsection{Plan of the paper}
	
	After covering the preliminaries, we turn in section \ref{sec:crofto} to study general Crofton formulas with a distributional Crofton measure, utilizing the Alesker-Radon transform on valuations. In particular, we study under which conditions such formulas can be applied directly to a given subset.
	
	In Section \ref{sec:LC_crofton_WF} we consider LC-regular domains and hypersurfaces of space forms, and deduce that they would be in good position for the evaluation of intrinsic volumes through Crofton integrals, once the corresponding distributions are constructed. The latter construction is carried out in Section \ref{sec:opq_distributions}. Moreover, these distributions are embedded in a meromorphic family of measures on a complex domain as a distributional boundary value, and some delicate - though central to our analysis - convergence questions are investigated and settled. Finally in Section \ref{sec:template}, the Hadwiger-type description of intrinsic volumes combined with the template method are applied to yield the explicit Crofton formulas in all cases.
	
	\subsection*{Acknowledgments}
	We wish to thank Gautier Berck for several inspiring talks and discussions, and the referee for numerous valuable comments which helped improve the exposition.
	
	\section{Preliminaries}
	\label{sec_preliminaries}
	
	\subsection{Notations}
	
	By 
	\begin{displaymath} 
	 \omega_n=\frac{\pi^\frac{n}{2}}{\Gamma\left(\frac{n}{2}+1\right)}
	\end{displaymath}
we denote the volume of the $n$-dimensional unit ball. The space of smooth complex valued $k$-forms on a manifold is denoted by $\Omega^k(M)$. The space of smooth complex valued measures on $M$ is $\mathcal M^\infty(M)$. The space of generalized measures, also called distributions, is denoted by
	\begin{displaymath}
	\mathcal M^{-\infty}(M):=(C^\infty_c(M))^*,
	\end{displaymath}
	where here and in the following the subscript $c$ denotes compactly supported objects. Similarly, for $m=\dim M$, we denote the space of $k$-dimensional currents on $M$ by  $$\Omega_{-\infty}^{m-k}(M):=(\Omega^{k}_c(M))^*.$$ The elements of this space can also be thought of as generalized $(m-k)$-forms.
	
	For an oriented $k$-dimensional submanifold $X \subset M$, we let $\llbracket X\rrbracket$ be the $k$-current which is integration over $X$. 
	
	By $\p_M:=\mathbb P_+(T^*M)$ we denote the cosphere bundle of $M$, which consists of all pairs $(p,[\xi]), p \in M, \xi \in T_p^*M \setminus 0$, where $[\xi]=[\xi']$ if there is some $\lambda>0$ with $\xi=\lambda \xi'$. When no confusion can arise, we use the same notation for subsets of $\mathbb P_M$ and their lifts to $T^*M$. The natural involution on $\p_M$ is the fiberwise antipodal map $s(p,[\xi]):=(p,[-\xi])$.
	
	The wave front set of a generalized form $\omega\in\Omega_{-\infty}(M)$ is a closed subset of $\p_M$, denoted by $\WF(\omega)$, and we refer to \cite{hoermander_pde1} or \cite{duistermaat_book96} for details.
	
	For a generalized pseudo-sphere $M$, we denote by $\Geod_k(M)$ the space of totally geodesic $k$-dimensional submanifolds of $M$. 
	
	\subsection{Smooth valuations}
	
	Let $M$ be a smooth manifold of dimension $m$, which we assume oriented for simplicity. 
	
	Let $\mathcal P(M)$ be the set of compact differentiable polyhedra on $M$. To $A\in\mathcal P(M)$ we associate two subsets of $\mathbb P_M$. The conormal cycle, denoted $\nc(A)$, is the union of all conormal cones to $A$. It is an oriented closed Lipschitz submanifold of dimension $(m-1)$, and naturally stratified by locally closed smooth submanifolds corresponding to the strata of $A$.
	
	The conormal bundle, denoted $N^*A$, is the union of the conormal bundles to all smooth strata of $A$. It holds that $\nc(A) \subset N^*A$. By definition, two stratified spaces intersect transversally if all pairs of smooth strata are transversal. 
	
	A smooth valuation is a functional $\mu:\mathcal P(M)\to \R$ of the form 
	\begin{displaymath}
	\mu(A)=\int_A \phi+\int_{\nc(A)} \omega, \quad \phi \in \Omega^m(M), \omega \in \Omega^{m-1}(\p_M).
	\end{displaymath}
	We will write $\mu=[[\phi,\omega]]$ in this case. 
	
	The space of smooth valuations is denoted by $\mathcal V^\infty(M)$. It admits a natural filtration 
	\begin{displaymath}
	\mathcal V^\infty(M)=\mathcal W_0^\infty(M) \supset \mathcal W_1^\infty(M) \supset \ldots \supset \mathcal W_m^\infty(M) =\mathcal M^\infty(M).
	\end{displaymath}
	It is compatible with the Alesker product of valuations. 
	
	\subsection{Generalized valuations}
	
	The space of generalized valuations is 
	\begin{displaymath}
	\mathcal V^{-\infty}(M):=(\mathcal V^\infty_c(M))^*.
	\end{displaymath}
	By Alesker-Poincar\'e duality we have a natural embedding $\mathcal V^\infty(M) \hookrightarrow \mathcal V^{-\infty}(M)$.
	
	There is a natural filtration 
	\begin{displaymath}
	\mathcal V^{-\infty}(M)=\mathcal W_0^{-\infty}(M) \supset \mathcal W_1^{-\infty}(M) \supset \ldots \supset \mathcal W_m^{-\infty}(M) =\mathcal M^{-\infty}(M).
	\end{displaymath}
	In particular, we may consider a generalized measure as a generalized valuation. 
	
	A compact differentiable polyhedron $A$ defines a generalized valuation $\chi_A$ by 
	\begin{displaymath}
	\langle \chi_A,\mu\rangle=\mu(A), \quad \mu \in \mathcal V_c^\infty(M).
	\end{displaymath}
	
	A generalized valuation $\psi$ can be represented by two generalized forms $\zeta \in C^{-\infty}(M), \tau \in \Omega_{-\infty}^m(\p_M)$ such that 
	\begin{displaymath}
	\langle \psi,[[\phi,\omega]]\rangle=\langle \zeta,\phi\rangle+\langle \tau,\omega\rangle.
	\end{displaymath}
	
	We refer to $\zeta$ and $\tau$ as the \emph{defining currents}. For instance, the defining currents of $\chi_A, A \in \mathcal P(M)$ are $\zeta=\mathbbm 1_A, \tau=\llbracket \nc(A) \rrbracket$. The wave front set of $\psi$ is defined as the pair $\Lambda \subset \p_M, \Gamma \subset \p_{\p_M}$ of the wave front sets of $\zeta$ and $\tau$. The space of all generalized valuations with wave front sets contained in $\Lambda,\Gamma$ is denoted by $\mathcal V^{-\infty}_{\Lambda,\Gamma}(M)$.
	
	Consider $\psi \in \mathcal V^{-\infty}(M)$. We say that $A \in \mathcal P(M)$ is WF-transversal to $\psi$, denoted $A \pitchfork \psi$, if the conditions of \cite[Theorem 8.3]{alesker_bernig} hold for 
	\begin{align*}
	(\Lambda_1, \Gamma_1) & :=\WF(\chi_A),\\ 
	(\Lambda_2,\Gamma_2) & :=\WF(\psi).
	\end{align*}
	
	These conditions imply that Alesker's product of smooth valuations can be extended to a jointly sequentially continuous product 
	\begin{displaymath}
	\mathcal{V}_{\Lambda_1,\Gamma_1}^{-\infty}(M) \times \mathcal{V}_{\Lambda_2,\Gamma_2}^{-\infty}(M) \to \mathcal{V}^{-\infty}(M),
	\end{displaymath}
	and in particular the pairing $\psi(A):=\langle \psi, \chi_A\rangle=\int_M \psi \cdot \chi_A$ is well-defined.

	Let us write a sufficient set of conditions in a particular case. 
		\begin{Proposition} \label{prop:WF_transversal}
		Assume $\WF(\psi) \subset(N^*D, N^*L)$ for some submanifolds with corners $D\subset M$, $L\subset\p_M$, and take $A \in \mathcal P(M)$. Assume further
		\begin{enumerate}
			\item $A\pitchfork D$.
			\item $\nc(A)\pitchfork \pi^{-1}D$, where $\pi:\mathbb P_M\to M$ is the natural projection.
			\item $\pi^{-1}A\pitchfork L$.
			\item $\nc(A) \pitchfork s(L)$, where $s:\mathbb P_M\to\mathbb P_M$ is the antipodal map. 
		\end{enumerate}
		Then $A\pitchfork \psi$.
	\end{Proposition}
	
	\proof
These conditions imply the conditions in \cite[Theorem 8.3]{alesker_bernig}.
	\endproof
	
	\subsection{Intrinsic volumes on pseudo-Riemannian manifolds}
	
	In \cite{bernig_faifman_solanes} we constructed a sequence of complex-valued generalized valuations $\mu_0^M,\ldots,\mu_m^M$ naturally associated to a pseudo-Riemannian manifold $M$ of dimension $m$. They are invariant under isometries and called the \emph{intrinsic volumes of $M$}. The intrinsic volume $\mu_0$ equals the Euler characteristic, while $\mu_m$ is the volume measure of $M$, multiplied by $\mathbf i^q$ where $q$ is the negative index of the signature. For other values of $k$, $\mu_k$ is typically neither real nor purely imaginary.  
	
	The wave front set of $\mu_k$ is contained in $(\emptyset,N^*(\LC^*_M))$, where $\LC^*_M \subset \p_M$ is the dual light cone of the metric, i.e. the set of all pairs $(p,[\xi]) \in \p_M$ such that $g|_p(\xi,\xi)=0$. Here we use the metric to identify $TM$ and $T^*M$. 
	
	A subset $A\in\mathcal P(M)$ is \emph{LC-transversal} if $\nc(A)\pitchfork \LC^*_M$.
	
	\begin{Lemma}
		Assume $D=\emptyset, L=\LC^*_M \subset \p_M$, and $A \in \mathcal P(M)$. Then the conditions of Proposition \ref{prop:WF_transversal} are equivalent to the LC-transversality of $A$. In particular, the intrinsic volume $\mu_k$ may be evaluated at LC-transversal $A$.
	\end{Lemma}
	
\proof
The first two conditions are empty. The third condition is satisfied for arbitrary $A$, since the tangent space to $\pi^{-1}A$ contains all vertical directions, while the tangent space of $\LC^*_M$ contains all horizontal directions. The fourth condition is precisely LC-transversality.
\endproof

We will also need the notion of LC-regularity, which was introduced in \cite{bernig_faifman_solanes}. 
\begin{Definition}
 Let $X$ be a smooth manifold equipped with a smooth field $g$ of quadratic forms over $TX$. We say that $(X,g)$ is LC-regular if $0$ is a regular value of $g \in C^\infty(TX \setminus \underline{0})$. 
\end{Definition}

It was shown in \cite[Proposition 4.9]{bernig_faifman_solanes} that the extrinsic notion of LC-transversality and the intrinsic notion of LC-regularity coincide: a submanifold of a pseudo-Riemannian manifold, equipped with the field of quadratic forms induced from the metric, is LC-regular if and only if it is LC-transversal.

The most important property of the intrinsic volumes is that they satisfy a Weyl principle: for any isometric immersion $M \looparrowright \widetilde M$ of pseudo-Riemannian manifolds we have  
	\begin{displaymath}
	\mu_k^{\widetilde M}|_M=\mu_k^M,
	\end{displaymath}
	in particular the restriction on the left hand side is well-defined. Conversely, we have shown in \cite{bernig_faifman_solanes_part2} that any family of valuations associated to pseudo-Riemannian manifolds that satisfies the Weyl principle must be a linear combination of intrinsic volumes.

	\section{Distributional Crofton formulas}\label{sec:crofto}
	
	Let $M^m$ be a manifold. A Crofton formula for a smooth valuation $\phi\in\mathcal V^\infty(M)$ has the form $\phi(A)=\int_{S}\chi(X(s)\cap A)d\mu(s)$, where $S$ is a smooth manifold parametrizing a smooth family of submanifolds of $M$, and $\mu$ a smooth measure on $S$. Similarly, a distributional Crofton formula has $\phi\in\mathcal V^{-\infty}(M)$, and $\mu$ is a distribution.
	
	In this section we study some general properties of such formulas, when $M\subset V\setminus\{0\}$ is a submanifold without boundary in a $d$-dimensional linear space $V$, and $S=\Gr_{d-k}(V), k <d, X(s)=s \cap M, s \in S$.
	
	We utilize the Radon transform on valuations, introduced in \cite{alesker_intgeo}. Loosely speaking, the Crofton map is but the Radon transform of a measure with respect to the Euler characteristic. However, there are technical difficulties in applying this formalism directly to distributions, and a large part of this section is concerned with resolving those difficulties. The main results to this end are Propositions \ref{prop:smooth_on_grassmannian} and \ref{prop:yomdin}. In the last part, we describe the Crofton wave front of sufficiently nice sets in Proposition \ref{prop:base_of_induction}, which controls the applicability of an explicit Crofton integral to the given set.
	
	\subsection{The general setting}\label{sec:crofton_functorial}
	
	\label{Preliminaries} For a submanifold with corners $X\subset M$, define $Z_X\subset X\times \Gr_{d-k}(V)$ by $Z_X=\{(x,E): x\in X\cap E\}$. Then $Z_X$ is a manifold with corners, more precisely it is the total space of the fiber bundle over $X$ with fiber $\Gr_{d-k-1}(V/\R x)$ at $x\in X$. Write 	
\begin{displaymath}
	 \xymatrix@=2em{X & Z_X \ar[l]_-{\pi_X} \ar[r]^-{\tau_X} & \Gr_{d-k}(V)}
	\end{displaymath}

	for the natural projections. 
 	
	Denote by $W_X\subset\Gr_{d-k}(V)$ the set of subspaces intersecting $X$ transversally in $V$. 
	
	We will need a simple fact from linear algebra, which we state in a rather general form that will be useful for us in several places.
	
	\begin{Lemma}\label{lem:pair_of_spaces}
	Let $V$ be a vector space, $L_0 \in \Gr_l(V), E_0 \in \Gr_k(V)$ with $L_0 \subset E_0$. Denote by $i:L_0\hookrightarrow E_0$ the inclusion, and $\pi: V/L_0\to V/E_0$ the projection.
            \begin{enumerate}
             \item Let $E(t)\in\Gr_k(V)$ be a smooth path with $E(0)=E_0$ and $A:L_0\to V/L_0$ a linear map. Then there is a smooth path $L(t) \in \Gr_l(E(t))$ with $L(0)=L_0$ and $L'(0)=A$ if and only if the following diagram commutes:
		\begin{displaymath}
		\xymatrix{L_0 \ar[r]^-{A} \ar[d]^{i} & V/L_0 \ar[d]^{\pi}\\
			E_0 \ar[r]^-{E'(0)} & V/E_0\\
		}
		\end{displaymath}
		\item Let $L(t) \in \Gr_l(V)$ be a smooth path with $L(0)=L_0$. Let $B:E_0\to V/E_0$ be a linear map. Then there is a smooth path $E(t) \in \Gr_k(V)$ with $L(t) \subset E(t)$, $E(0)=E_0$ and $E'(0)=B$ if and only if the following diagram commutes:
		\begin{displaymath}
		\xymatrix{L_0 \ar[r]^-{L'(0)} \ar[d]^{i} & V/L_0 \ar[d]^{\pi}\\
			E_0 \ar[r]^-{B} & V/E_0\\
		}
		\end{displaymath} 
            \end{enumerate}
	\end{Lemma}

	\begin{Remark}
		The 'only if' statement obviously remains true if instead of $L(t)\subset E(t)$, we have $\measuredangle (L(t), E(t))=o(t)$ with respect to any Euclidean structure. 
	\end{Remark}
	
	\proof
	Consider the partial flag manifold $Z=\{ L\subset E\}\subset \Gr_l(V)\times\Gr_k(V)$. The group $\GL(V)$ acts transitively on $Z$, and any smooth path $F(t)=(L(t) \subset E(t))\in Z$ can be lifted to a smooth curve $g(t)\in\GL(V)$ with $g(0)=\id$ and $F(t)=g(t)F(0)$. Thus $E'(0):E_0\to V/E_0$ and $L'(0):L_0\to V/L_0$ are both projections of $g'(0):V\to V$, and the diagram commutes. 
	
	In the other direction, write $\pi_W:V\to V/W$ for the natural projection. it follows by the above that the set of velocity vectors $L'(0)$ for all curves $L(t)\subset E(t)$ is the affine space $\{\pi_{L_0}\circ T|_{L_0}\in\mathrm{Hom}(L_0, V/L_0): T\in\mathfrak{gl}(V), \pi_{E_0}\circ T|_{E_0}=E'(0)\}$, which is of dimension $\binom{k}{2}-\binom{l}{2}-\binom{ k-l}{2}=l(k-l)$. This is also the dimension of the affine space of all $A$ such that the diagram commutes, which finishes the proof of the first part. The second part follows from the first one by taking orthogonal complements.
	\endproof
	
	We need the following technical statement appearing in \cite[Proposition 5.1.3]{alesker_intgeo}.
	\begin{Lemma}\label{lem:radon_condition}
		The natural projection $\pi:N^*Z_M\setminus 0\to T^*M\setminus 0$ is a submersion.
	\end{Lemma}

\proof
Let $(p_t,\xi_t)$ be a smooth path in $T^*M\setminus 0$. We will lift it to a smooth path $(p_t,E_t,\xi_t,\eta_t)\in T^*(M\times\Gr_{d-k}(V))$ such that $p_t \in E_t$,  and $(\xi_t,\eta_t)\in N^*_{p_t,E_t}Z_M$. 
	Now for $v\in T_pM$,  $B \in T_E\Gr_{d-k}(V)=\mathrm{Hom}(E,V/E)$, we have by Lemma \ref{lem:pair_of_spaces} (applied with $l=1, L_0=\R p$) that $(v,B)\in T_{p,E}Z_M$ if and only if $v+E=B(p)$. 
	
	Hence 
	\begin{align*}
	N^*_{p,E} Z_M=\{&(\xi,\eta) \in T_p^*M \times T_E^*\Gr_{d-k}(V):\\& \langle \xi,v\rangle+\langle \eta,B\rangle=0\text{ whenever } v+E=B(p)\}.
	\end{align*}
	
	Fix a Euclidean structure on $V$, inducing Euclidean structures on the spaces $\mathrm{Hom}(E_t, V/E_t)$.   
	Let us choose some $E_t$ such that $p_t\in E_t$, and $T_{p_t}M\cap E_t\subset \Ker(\xi_t)$, which evidently can be done. Consider the linear subspace 
	\begin{displaymath}
	W_t=\{ B \in T_{E_t}\Gr_{d-k}(V): B(p_t)\in (T_{p_t}M+E_t)/E_t\},
	\end{displaymath} 	
	and recall the natural isomorphism $q_t:(T_{p_t}M+E_t)/E_t \xrightarrow{\sim}  T_{p_t}M/(T_{p_t}M\cap E_t)$.
	We now may define $\eta_t\in W_t^*$ by $\langle \eta_t,B\rangle=-\langle \xi_t,q_t(B(p_t))\rangle$ for each $B \in W_t$, as  $T_{p_t}M\cap E_t \subset \ker \xi_t$. Extend $\eta_t$ by zero to $W_t^\perp$. It follows that $(\xi_t,\eta_t)\in N^*_{p_t,E_t}Z_M$, completing the proof.
	\endproof
	
	It follows by \cite[Corollary 4.1.7]{alesker_intgeo} that the Radon transforms with respect to the Euler characteristic,  $\mathcal R_M=(\tau_M)_*\pi_M^*:\mathcal V_c^{-\infty}(M)\to \mathcal V^{-\infty}(\Gr_{d-k}(V))$ and $\mathcal R_M^T=(\pi_M)_*\tau_M^*:\mathcal V^\infty(\Gr_{d-k}(V))\to \mathcal V^\infty(M)$, are well-defined and continuous.

	\begin{Definition}	
		For any $\phi\in\mathcal V_c^{-\infty}(M)$, let $\widehat \phi\in C^{-\infty}(\Gr_{d-k}(V))$ be the defining current of $\mathcal R_M\phi$ (on the base manifold). Equivalently, using \cite[Proposition 7.3.6]{alesker_val_man4} we have 	
		\begin{displaymath}	
		\widehat\phi=[\mathcal R_M\phi]\in\mathcal W_0^{-\infty}(\Gr_{d-k}(V))/\mathcal W_1^{-\infty}(\Gr_{d-k}(V))=C^{-\infty}(\Gr_{d-k}(V)).	
		\end{displaymath}	
	\end{Definition}	
	
	\begin{Remark}\label{rem:radon_fail1}	
		It is false in general that $\widehat \phi$ is a smooth function when $\phi$ is a smooth valuation, see Remark \ref{rem:radon_fail2}.	
	\end{Remark}	
	
	\begin{Definition}\label{def:smooth_crofton}	
		The Crofton map $\Cr_M: \mathcal M^{\infty}(\Gr_{d-k}(V))\to \mathcal W_k^\infty(M)$ is the restriction of $\mathcal R_M^T$ to $\mathcal M^\infty(\Gr_{d-k}(V))$. More explicitly,  	
		\begin{displaymath}	
		\Cr_M(\mu)(X)=\int_{\Gr_{d-k}(V)}\widehat \chi_X(E)d\mu(E), \quad X \in \mathcal P(M).	
		\end{displaymath}	
	\end{Definition}	
	
	 We will see in Proposition \ref{prop:yomdin} below that $\widehat \chi_X(E)=\chi(X \cap E)$. 
		
	\subsection{Distributional Crofton measures}
	To allow distributional Crofton measures, it seems essential to require that all intersections $E\cap M$ are transversal, for $E\in\Gr_{d-k}(V)$. This is easily seen to be equivalent, for any $k>0$, to having $\R x\oplus T_xM=V$ for all $x\in M$. In particular $\dim V=\dim M+1=m+1$. We deduce that $M$ is a hypersurface that is locally diffeomorphic to an open subset of $\mathbb P_+(V)$ through the radial projection. In other words, $M$ is locally a strictly star-shaped hypersurface around the origin.
	
	\begin{Proposition}\label{prop:smooth_on_grassmannian}	
	\begin{enumerate}
	 \item For all $0\leq k\leq m$ and $\psi \in\mathcal V^\infty_c(M)$, it holds that $E \mapsto \psi(E\cap M)$ is a smooth function on $\Gr_{m+1-k}(V)$.
	 \item The image in $C^{-\infty}(\Gr_{m+1-k}(V))$ of this function equals $\widehat\psi$.
	\end{enumerate}
	\end{Proposition}
	
	\proof
	\begin{enumerate}
	\item Let us first show $E\mapsto\psi(E\cap M)$ is smooth.   By choosing an open cover of $M$ by star-shaped charts, and using the partition of unity property of smooth valuations \cite{alesker_val_man4}, we may assume $M$ projects diffeomorphically to an open subset of $\mathbb P_+(V)$, which we henceforth identify with $M$. 
	
	By Boman's theorem \cite{boman67}, it suffices to prove that $\psi(E_t \cap M)$ is a smooth function of $t\in (-\epsilon,\epsilon)$ for all smooth curves $E_\bullet:(-\epsilon,\epsilon)\to \Gr_{m+1-k}(V)$. It suffices in fact to show smoothness in some open interval around $0$ for any such given curve. 
	
	Let us lift $E_t$ to a smooth curve $g_t\in \mathrm{GL}(V)$ with $g_0=\id$ and $E_t=g_tE_0$. Then $\psi(E_t\cap M)= g_t^*\psi (E_0 \cap M)$ for sufficiently small $t$ such that $g_t(\Supp(\psi))\subset M$, establishing the first part.
	\item 
	
	Let us check $\psi(\bullet \cap M)=\widehat\psi$ in $C^{-\infty}(\Gr_{m+1-k}(V))$. Take $\mu\in\mathcal M^\infty(\Gr_{m+1-k}(V))$, and write 
	\begin{displaymath}
        \mu=\int_{\Gr_{m+1-k}(V)}\delta_Ed\mu(E)= \int_{\Gr_{m+1-k}(V)}\chi_{\{E\}}d\mu(E)\in\mathcal V^\infty(\Gr_{m+1-k}(V)).
	\end{displaymath}

	\textit{Claim.} $\tau_M^{-1}E\pitchfork \pi_M^*\psi$. 
	
	To see this, write $Z=Z_M$ and identify $W:=Z\times_M\mathbb P_M$ with its image in $\mathbb P_{Z}$ under $d\pi_M^*$. Explicitly,  $W|_{(x,E)}=\mathbb P_+(\mathrm {Ker}(d_{(x,E)}\pi_M)^\perp)$, so $W$ is the union of the conormal bundles to all fibers of $\pi_M$. It follows from \cite[Proposition 3.3.3]{alesker_intgeo} that $\WF(\pi_M^*\psi)\subset(\emptyset, N^*W)$. By Proposition \ref{prop:WF_transversal}, it suffices to check that two intersections in $\mathbb P_{Z}$ are transversal: $\pi^{-1}(\tau_M^{-1}E)\pitchfork W$ and $N^*(\tau_M^{-1}E)\pitchfork W$. 
	
	Denote $z=(x,E)$, let $(z,\zeta)$ be an intersection point. The first intersection is easy to analyze: $T_{z,\zeta}\pi^{-1}(\tau_M^{-1}E)$ contains all vertical directions of $\mathbb P_{Z}$, while $T_{z,\zeta}W$ contains all horizontal directions. 
	
	To analyze the second intersection, we lift all manifolds from $\mathbb P_Z$ to $T^*Z$, and retain all notation for the corresponding objects. As in the previous case, the image of $T_{z,\zeta}W$ under the natural projection $\pi:T_{z,\zeta} T^*Z\to T_zZ$ is all of $T_zZ$, and so it suffices to show $T_{z,\zeta}(T_z^*Z)\subset T_{z,\zeta}W+ T_{z,\zeta}N^*(\tau_M^{-1}E)$.
	
	Since $N_z^*(\pi_M^{-1}x) \subset W$, it suffices to show that
	\begin{displaymath}
	T_{z,\zeta}(T_z^*Z)\subset T_{z,\zeta}N_z^*(\pi_M^{-1}x)+ T_{z,\zeta}N_z^*(\tau_M^{-1}E),
	\end{displaymath}	
	which is the same as 
	\begin{displaymath}
	T_z^*Z\subset N_z^*(\pi_M^{-1}x)+ N_z^*(\tau_M^{-1}E)=N_z^*(T_z\pi_M^{-1}x\cap T_z\tau_M^{-1}E).
	\end{displaymath}
	
	The proof of the claim is completed by noting that the intersection $T_z\pi_M^{-1}x\cap T_z\tau_M^{-1}E$ is trivial.
		
	Consider the set 
	\begin{displaymath}
	X:=\left\{(E,[\xi]): E \in \Gr_{m+1-k}(V), [\xi] \in \WF(\chi_{\tau_M^{-1}E})\right\} \subset \Gr_{m+1-k}(V) \times \p_{\p_Z}.
	\end{displaymath}
	We claim that it is compact. If $X \subset \bigcup_{i \in I} U_i$ is an open cover, then for each $E \in \Gr_{m+1-k}(V)$ we find a finite subcover $X_E \subset \bigcup_{i \in I_E} U_i$ of the compact set
	
\begin{displaymath}
X_E:=X\cap (\{E\} \times \p_{\p_Z})=\{E\} \times \WF(\chi_{\tau_M^{-1}E}).
\end{displaymath}
The map $g \mapsto X_{gE}$ is $\GL(V)$-equivariant, hence there exists some open neighborhood $V_E \subset \Gr_{m+1-k}(V)$  of $E$ such that $X_{E'} \subset \bigcup_{i \in I_E} U_i$ for all $E' \subset V_E$. Now $\Gr_{m+1-k}(V)$ is compact, hence finitely many $V_{E_j}$ cover $\Gr_{m+1-k}(V)$. Then $X \subset \bigcup_j \bigcup_{i \in E_j} U_i$ is a finite subcover, proving the claim. The image of $X$ in $\p_{\p_Z}$ is then a compact set disjoint from $\WF(\pi_M^*\psi)$.

	Thus we can find a closed cone $\Gamma\subset  T^*\mathbb P_{Z}\setminus \underline 0$ such that for all $E\in\Gr_{m+1-k}(V)$, $\chi_{\tau_M^{-1}E}\in\mathcal V^{-\infty}_{(\emptyset, \Gamma)}(Z)$, and $\pi_M^*\psi$ acts as a sequentially continuous functional on the latter space. Thus we can write
	\begin{align*}
	\langle \widehat \psi,\mu\rangle&= \langle \pi_M^*\psi, \tau_M^*\int_{\Gr_{m+1-k}(V)} \chi_{\{E\}}d\mu(E)\rangle\\& = \int_{\Gr_{m+1-k}(V)} \langle \pi_M^*\psi,\chi_{\tau_M^{-1}(E)}\rangle d\mu(E).
	\end{align*}	

It remains to check that 
	\begin{equation}\label{eq:psi_hat}
	\langle \pi_M^*\psi,\chi_{\tau_M^{-1}(E)}\rangle= \psi(E\cap M).
	\end{equation} 	

For a compact submanifold with boundary $A \subset M$ that is transversal to $E\cap M$, we have by \cite[Theorem 5]{alesker_bernig},
\begin{displaymath}
\langle \pi_M^*\chi_A,\chi_{\tau_M^{-1}(E)}\rangle=\chi(\pi_M^{-1}A\cap \tau_M^{-1}(E))=\chi(A\cap E)=\chi_A(E\cap M).
\end{displaymath}
	
It follows by linearity that any smooth valuation of the form $\psi=\int_{\mathcal A} \chi_Ad\nu(A)$, where $\mathcal A$ is a family of submanifolds $A$ as above and $\nu$ a smooth measure, satisfies \eqref{eq:psi_hat}. This family $\mathbf {Cr}_E$ of valuations spans a dense subset in $\mathcal V^\infty(M)$. Indeed, we may approximate $\chi_A$ in $\mathcal V^{-\infty}(M)$ by a sequence in $\mathbf {Cr}_E$ for any $A$ transversal to $E\cap M$. Were $\mathbf {Cr}_E$ not dense, by Alesker-Poincar\'e duality one could find a non-zero smooth valuation $\phi$ annihilating $\mathbf {Cr}_E$, and thus also vanishing on all submanifolds with boundary that are transversal to $E\cap M$. By the genericity of transversality and continuity, $\phi$ would vanish on all submanifolds with boundary. But this is impossible by \cite{bernig_broecker07}. It follows that equality in \eqref{eq:psi_hat} holds for all $\psi$.
\end{enumerate}
\endproof

\begin{Corollary}\label{coro:smooth_on_grassmannian} The map $\Gr_{m+1-k}(V)\to \mathcal V^{-\infty}(M)$, $E\mapsto \chi_{E\cap M}$ is smooth, and for $\mu\in\mathcal M^\infty(\Gr_{m+1-k}(V))$ it holds that $\Cr(\mu)=\int_{\Gr_{m+1-k}(V)}\chi_{E\cap M}d\mu(E)$.	
\end{Corollary}
	
\begin{Proposition} \label{prop_extension_crofton_map}
	The map $\Cr:\mathcal M^\infty(\Gr_{m+1-k}(V))\to\mathcal W_k^\infty(M)$ extends to a continuous map $\Cr: \mathcal M^{-\infty}(\Gr_{m+1-k}(V))\to \mathcal W_k^{-\infty}(M)$, by setting, for all $\psi \in \mathcal V^\infty_c(M)$,
\begin{displaymath}
\left\langle \Cr(\mu), \psi\right\rangle := \int_{\Gr_{m+1-k}(V)} \psi(E\cap M)d\mu(E).  
\end{displaymath}
\end{Proposition}
	
\proof
The right hand side is well-defined for a generalized measure $\mu$, since the function $E \mapsto \psi(E \cap M)$ is smooth by Proposition \ref{prop:smooth_on_grassmannian}. Take $\mu\in\mathcal M^\infty(\Gr_{m+1-k}(V))$, $\psi\in\mathcal V_c^\infty(M)$. To verify this new definition extends the smooth one, we ought to check that
\begin{displaymath} 
\langle (\pi_M)_*\tau_M^*\mu, \psi\rangle =  \int_{\Gr_{m+1-k}(V)} \psi(E\cap M)d\mu(E), 
\end{displaymath}
which is the content of Corollary \ref{coro:smooth_on_grassmannian}. Continuity is equally evident. 
\endproof
	
\begin{Remark} \label{rem:radon_fail2}
It is tempting to define $\Cr(\mu)$ as a Radon transform: $\Cr(\mu)=\mathcal R_M^T\mu$, as defined in \cite{alesker_intgeo}. Unfortunately the conditions of \cite[Corollary 4.1.7]{alesker_intgeo},  which guarantee that the transform is well-defined on generalized valuations, do not hold for general $k$, as can be seen by a simple dimension count.
\end{Remark}

	\subsection{Functorial properties of Crofton measures.} \label{sec:restriction_crofton}
	The following is a partial summary of the results of \cite[Appendix B]{faifman_crofton} (adapted from the affine to the linear Grassmannian), whereto we refer the reader for further details. 
	
	Let $j: U^r\hookrightarrow V^d$ be an inclusion of a linear subspace.  There is then a well-defined operation of restriction 
	\begin{displaymath}
	j^*:\mathcal M^\infty(\Gr_k(V)) \to \mathcal M^\infty(\Gr_{k-(d-r)}(U)),
	\end{displaymath}
	which is the pushforward under the (almost everywhere defined) map $J_U: E\mapsto j^{-1}(E)=E\cap U$.
	
	Let $S_U\subset\Gr_{k}(V)$ be the collection of subspaces intersecting $U$ non-generically, and fix a closed cone $\Gamma \subset T^*\Gr_k(V)\setminus 0$ such that $\Gamma\cap N^*S_U=\emptyset$. Given $k\geq d-r$, let $\mathcal M^{-\infty}_\Gamma (\Gr_k(V))$ denote the set of generalized measures (distributions) $\mu$ whose wave front sets lie in $\Gamma$, equipped with the H\"ormander topology. 
	
	The map $j^*$ extends as a sequentially continuous map 
	\begin{displaymath}
	j^*:\mathcal M_\Gamma^{-\infty}(\Gr_k(V)) \to \mathcal M^{-\infty}(\Gr_{k-(d-r)}(U)).
	\end{displaymath}
	
	Similarly, if $\pi :V\to W$ is a quotient map, there is a natural pushforward operation 
	\begin{displaymath}
	\pi_*: \mathcal M^\infty(\Gr_k(V))\to \mathcal M^\infty(\Gr_{k}(W)),
	\end{displaymath}
	which is the pushforward under the  (almost everywhere defined) map $\Pi_W: E\mapsto \pi(E)$. It extends to distributions whose wave front sets are disjoint from the conormal cycle of the collection of subspaces intersecting $\Ker \pi$ non-generically. 
	
	The following proposition captures the intuitively obvious fact that the pullback of distributions/valuations under embeddings commutes with the Crofton map. We prove a weak version which suffices for our purposes.
	
	Recall that $M$ is a locally star-shaped hypersurface around the origin.
	
	\begin{Proposition}\label{prop:restrictions_commute} Take a submanifold $M^r\subset V^d$, a subspace $j: U\hookrightarrow V$ such that $Z:=M\cap j(U)$ is a submanifold, and a distribution $\mu\in\mathcal M_\Gamma^{-\infty}(\Gr_{d-k}(V))$. Assume $\Cr_M(\mu)$ is transversal to $Z$ in the sense of \cite[Definition 3.5.2]{alesker_intgeo}. Then $\Cr_M(\mu)|_Z=\Cr_Z(j^*\mu)$.
	\end{Proposition}
	
	\proof
	Choose an approximate identity $\rho_i\in\mathcal M^\infty(\textrm{GL}(V))$ as $i\to\infty$, and set $\mu_i=\mu\ast\rho_{i}\in\mathcal M^\infty(\Gr_{d-k}(V))$. For all $A\in \mathcal P(Z)$ we have
	\begin{displaymath}
	\Cr_M(\mu_i)(A)= \int_{\Gr_{d-k}(V)} \chi(A \cap E)d\mu_i(E)=\int_{\Gr_{r-k}(U)}  \chi(A \cap E)d((J_U)_*\mu_i)(E),
	\end{displaymath} 
	and therefore $\Cr_M(\mu_i)|_Z=\Cr_Z(j^*\mu_i)$. The restriction of valuations to a submanifold is continuous in the H\"ormander topology on the space of valuations with wave front set contained in $\textrm{WF}(\Cr_M(\mu))$, see \cite[Claim 3.5.4]{alesker_intgeo}. Thus the left hand side weakly converges to $\Cr_M(\mu)|_Z$. The right hand side weakly converges to $\Cr_Z(\mu)$.
	\endproof
	
	\subsection{Applying generalized Crofton formulas to subsets}\label{sec:crofton_functorial2}
	
	Let $M^m\subset V=\R^{m+1}$ be a strictly star-shaped hypersurface around the origin. Given $A\in\mathcal P(M)$ and a Crofton distribution $\mu\in\mathcal M^{-\infty}(\Gr_{m+1-k}(V))$, we would like to evaluate $\Cr(\mu)$ on $A$ using an explicit Crofton integral, whenever $A\pitchfork\Cr(\mu)$.
	
		The following proposition provides some a-priori regularity for $\widehat\chi_A$.
	
	\begin{Proposition}\label{prop:yomdin}	
		For $A\in\mathcal P(M)$, it holds that $\chi(A\cap \bullet)\in L^1(\Gr_{m+1-k}(V))$, is finite and locally constant on $W_A:=\{E:E\pitchfork A\}$. Furthermore, $\widehat \chi_A=\chi(A \cap \bullet)$.
	\end{Proposition}	
	
	\proof	
	Let us first check that $\chi(A \cap \bullet)\in L^1 ( \Gr_{m+1-k}(V))$. Fix a Euclidean structure on $V$ and identify $M$ with the unit sphere. By \cite[{Lemma A.2}]{bernig_fu_solanes}, for a fixed $E_0\in\Gr_{m+1-k}(V)$ we have $[g \mapsto \chi(A\cap gE_0)] \in L^1(\SO(V))$. Let $dg, dE$ be the Haar measures on $\SO(V)$ and $\Gr_{m+1-k}(V)$ respectively, and $p:\SO(V)\to \Gr_{m+1-k}(V)$ given by $g\mapsto gE_0$. Then $p_*(\chi(A\cap gE_0)dg)=\chi(A\cap E)dE$, and so $\chi(A\cap E)$ is integrable. 
	It is evidently finite and locally constant on $W_A$
	
	It remains to check that $\widehat \chi_A=\chi(A\cap \bullet)$. Take an approximate identity ${ \rho}_j\in\mathcal M^\infty(\SO(V))$, which for convenience we assume invariant under inversion. 
	
	Consider the convolution $\phi_j:=\chi_A\ast { \rho_j}\in\mathcal V^{-\infty}(M)$. As $\SO(V)$ is transitive on $M$ and $\mathbb P_M$, it follows that the defining currents of $\phi_j$ are smooth, and therefore $\phi_j\in\mathcal V^\infty(M)$.
	
	By \cite[Theorem A.1]{bernig_fu_solanes}, $\tilde \phi_j:=\int_{\SO(V)}\chi(gA\cap \bullet)d{ \rho}_j(g)$ is a well-defined smooth valuation. Let us show that $\tilde\phi_j=\phi_j$. Take $\psi\in\mathcal V_c^\infty(M)$ and compute: 
	\begin{displaymath}
	\langle\tilde\phi_j, \psi\rangle= \int_{\SO(V)}\psi(gA\cap M)d{ \rho}_j(g) =\int_{\SO(V)}\psi(gA)d{ \rho}_j(g)
	\end{displaymath}  
	by \cite{bernig_fu_solanes}, while 
	\begin{displaymath}
	\langle \phi_j, \psi\rangle=\langle \chi_A, \psi\ast { \rho}_j\rangle =(\psi\ast\nu_j)(A)=\int_{\SO(V)}\psi(gA)d{ \rho}_j(g).
	\end{displaymath}	
	
	Equality now follows by Alesker-Poincar\'e duality.  We thus have the following equalities of functions on $\Gr_{m+1-k}(V)$:
	\begin{displaymath}
	\phi_j( { \bullet} \cap M)=\tilde\phi_j({ \bullet}\cap M)=\int_{\SO(V)}\chi(gA\cap { \bullet})d{ \rho}_j(g)=\chi(A\cap \bullet)\ast { \rho}_j,
	\end{displaymath}
	where the right hand side is the convolution of $\chi(A\cap \bullet)\in L^1(\Gr_{m+1-k}(V))$ with ${ \rho}_j$. It follows that $\phi_j(E\cap M)\to \chi(A\cap E)$ in $L^1(\Gr_{m+1-k}(V))$. 
	
	Fix $\mu\in\mathcal M^\infty(\Gr_{m+1-k}(V))$. By Proposition \ref{prop:smooth_on_grassmannian} and $\GL(V)$-equivariance, 
	\begin{align*}
	\langle \widehat \chi_A,\mu \rangle &= \lim_{j\to\infty} \langle \widehat \chi_A \ast { \rho}_j, \mu\rangle = \lim_{j\to\infty} \langle \widehat {\chi_A\ast { \rho}_j}, \mu\rangle\\ &=\lim_{j\to\infty} \int_{\Gr_{m+1-k}(V)} \phi_j(E\cap M) d\mu(E)
	=\int_{\Gr_{m+1-k}(V)} \chi(A\cap E) d\mu(E), 
	\end{align*}	
	and so $\widehat \chi_A=\chi(A\cap \bullet)$.
	\endproof	
	
	\begin{Definition}
		The \emph{$k$-Crofton wave front} of $A\in\mathcal P(M)$ is $\Cr\WF^k(A):=\WF(\widehat\chi_A)\subset T^*\Gr_{m+1-k}(V)$.
	\end{Definition}
	As $\widehat \chi_A$ is real-valued, $\Cr\WF^k(A)$ must be symmetric under the fiberwise antipodal map.
	\begin{Proposition}\label{prop:apply_crofton_general}
		Assume $\Cr\WF^k(A)\cap \WF(\mu)=\emptyset$. 
		Then \begin{equation}\label{eq:crofton_applicable}
		\Cr(\mu)(A)=\int_{\Gr_{m+1-k}(V)}\chi(A\cap E)d\mu(E).
		\end{equation}
	\end{Proposition}
	
	\proof
	We identify $M$ with $\mathbb P_+(V)$. Then $\mathcal V^{-\infty}(M)\to C^{-\infty}(\Gr_{m+1-k}(V))$, $\phi\mapsto\widehat \phi$ is $\GL(V)$-equivariant. Consider the sequence of smooth valuations $\psi_j$ given by $\psi_j=\int_{\GL(V)}g^*\chi_A\cdot d \rho_j(g)$, where $ \rho_j$ is a compactly supported approximate identity on $\GL(V)$. Clearly $\psi_j\to\chi_A$ in the H\"ormander topology of $\mathcal V^{-\infty}_{\WF(\chi_A)}(M)$. By $\GL(V)$-equivariance we have that 
	\begin{displaymath}
	\widehat\psi_j= \int_{\GL(V)}g^*\widehat \chi_A \cdot d \rho_j(g) \to\widehat\chi_A
	\end{displaymath}
	in $C^{-\infty}_{\WF(\widehat\chi_A)}(\Gr_{m+1-k}(V))$. 
	
	It holds by Propositions \ref{prop_extension_crofton_map} and \ref{prop:smooth_on_grassmannian} that 
	\begin{displaymath}
	\langle \Cr(\mu), \psi_j\rangle = \int_{\Gr_{m+1-k}(V)} \psi_j(E\cap M)d\mu(E)=\langle \mu,\widehat \psi_j\rangle. 
	\end{displaymath}
	As $j \to \infty$, the left hand side converges to $\langle \Cr(\mu),\chi_A\rangle=\Cr(\mu)(A)$, as $A \pitchfork \Cr(\mu)$. The right hand side converges to $\langle \mu,\widehat \chi_A\rangle$ (since $\WF(\mu) \cap \WF(\widehat \chi_A)=\emptyset$), which is the same as $\int_{\Gr_{m+1-k}(V)}\chi(A\cap E)d\mu(E)$ by Proposition \ref{prop:yomdin}.
	\endproof
	
	Determining $\Cr\WF^k(A)$ precisely appears to be difficult in general. Let us focus on a subset $A\in\mathcal P(M)$ which is either a compact domain with smooth boundary, or a compact hypersurface without boundary.
	
	For the following, we write $H=H(A)$ for $\partial A$ if $A$ is of full dimension, and for $A$ when it is a hypersurface. Write $\widehat E:=E\cap M$, and note that $E$ intersects $H$ transversally in $V$ if and only if $\widehat E$ intersects $H$ transversally in $M$. 
	Denote 
	\begin{align*}
	\widetilde B_H & :=\{(x,E)\in Z_H: T_x\widehat E\subset T_xH\}, \\
	B_H & : =\tau_H(\widetilde B_H)\subset\Gr_{m+1-k}(V).	 
	\end{align*}
It is not hard to see that $\widetilde B_H$ is an embedded submanifold of $Z_H$ of dimension 
\begin{align}\notag
  \dim \widetilde B_H&= \dim H+(m-k)(\dim H-(m-k))\\&=k(m+1-k)-1=\dim \Gr_{m+1-k}(V)-1. \label{eq_dimension_tildebh}
\end{align}

If $(x,E)\in \widetilde B_H$, we say that $x\in H$ is a tangent point for $E$. Observe also that $W_A= B_H^c$. 

Write $\tilde \tau_H$ for the restriction $\tau_H|_{\widetilde B_H}:\widetilde B_H\to \Gr_{m+1-k}(V)$.   We sometimes write $B^{m+1-k}_H$, etc. to specify the dimension.
	
	\begin{Definition}
		We say that $E\in\Gr_{m+1-k}(V)$ is a \emph{regular tangent} to $A$ if $\tilde\tau_H$ is immersive on $\tilde\tau_H^{-1}(E)$.
	\end{Definition}
	Note that if $E\notin B_H$ then it is automatically regular. 
	
	For a subset $A\subset V$ we denote by $\mathbb P(A)$ its image in the projective space $\mathbb P(V)$. The regularity of the tangent is equivalent to the non-vanishing of the Gauss curvature of the corresponding section, as follows. 
	\begin{Lemma}
		Fix $(p,E)\in \widetilde B_H$. Choose any line $N\subset T_pM\setminus T_pH$, and set $F=E\oplus N$. Then $\tilde\tau_H:\widetilde B_H\to \Gr_{m+1-k}(V)$ is an immersion at $(p,E)$ if and only if $\mathbb P(H\cap F)\subset \mathbb P( F)$ has non-degenerate second fundamental form at $p$. 
		
		In particular, all tangents to $A$ are regular if and only if $\mathbb P(H)\subset \mathbb P(V)$ is a strictly convex hypersurface. 
	\end{Lemma}
	
	\proof
	Let us sketch the argument, see \cite[Lemma 1(ii)]{teufel} for details.   
		Clearly $d\tilde \tau_H$ is injective on the subspace of directions where $p$ moves transversally to $E$. Namely, fixing any subspace $\overline E\subset T_pH$ such that $E\oplus\overline E=T_pH\oplus \R p$, $d\tilde\tau_H$ is injective on $\{ (v,A)\in T_pH\times T_E\Gr_{m+1-k}(V):  v\in \overline E\}\cap T\widetilde B_H$. That injectivity is retained as the remaining directions are added, corresponds to the non-degeneracy of the Gauss map of the section $H\cap F$.
	\endproof
		
	We now describe the Crofton wave front near regular tangents. For an immersed manifold $i:X\looparrowright Y$ and $y\in i(X)$, we denote 
	\begin{displaymath}
	  N_y^*i(X)=\bigcup_{x\in i^{-1}y}(d_xi(T_xX))^\perp\subset T_y^*Y,\qquad N^*i(X)=\bigcup_{y\in i(X)}N^*_yi(X).
	 \end{displaymath}

	\begin{Proposition}\label{prop:base_of_induction}
		Assume $E_0\in B^{m+1-k}_H$ is a regular tangent. Then $\Cr\WF^k_{E_0}(A)\subset N^*_{E_0}B_H$.
	\end{Proposition}

		That $\Cr\WF^k_{E_0}(A)$ is contained in the {\em sum} of the conormal spaces of the embedded parts of $B_H$ follows from the fact that $\widehat\chi_A$ is locally constant on the complement of $B_H$. However, to show that it is actually contained in the {\em union} of those conormal spaces, in the following proof we will need a more precise description of $\widehat\chi_A$.

	\proof 
	In the following, by a ball (centered at a point) we mean a compact contractible neighborhood (of the point) with smooth boundary.  
	Since $\widetilde B_H$ is a submanifold of $Z_H$, by assumption $B_H\subset\Gr_{m+1-k}(V)$ is an immersed submanifold in a  neighborhood around $E_0$, which is a hypersurface by \eqref{eq_dimension_tildebh}.
	
	The preimage $\tilde \tau_H^{-1}(E_0)$ must be finite, or else we could find a sequence of distinct points $(q_j, E_0)\in \widetilde B_H$, which then has a limit point $(q_0, E_0)$, and $\tilde \tau_H$ would fail to be injective in a neighborhood of $(q_0, E_0)$, contradicting the assumed immersivity of $\tilde \tau_H$ there. Denote $\tilde\tau_H^{-1}(E_0)=\{(q_j, E_0), 1\leq j\leq N\}$.
	
	We can now find a ball $W \subset \Gr_{m+1-k}(V)$ centered at $E_0$, such that $B_H\cap W$ is the finite union of embedded hypersurfaces $F_j$, each diffeomorphic to a Euclidean ball, with $E_0\in F_j$ and $\partial F_j\subset\partial W$ for all $j$. Note that we have no control on how these hypersurfaces intersect each other. Denote by $C_j^\pm$ the connected components of $W\setminus F_j$. The indices are matched by requiring that a neighborhood of $(q_j, E_0)\in\widetilde B_H$ is mapped to $F_j$ by $\tilde\tau_H$.
	
	Fix small balls $K_j\subset M$ around $q_j$ such that 
    \begin{displaymath}
        \tilde\tau_H:\pi_H^{-1}(K_j)\cap \tilde\tau_H^{-1}(W) \to F_j
    \end{displaymath}	
is an embedding, $\partial K_j\pitchfork H$ and $\partial K_j\pitchfork\widehat E_0$ in $M$.
As $Z_0:=\{q_j:1\leq j\leq N\}$ is the subset of all points in $H$ where $\widehat E_0$ fails to intersect $H$ transversally, it holds that $\widehat E_0\pitchfork (H\setminus Z_0)$, and so $\widehat E_0\cap (H\setminus Z_0)$ is a locally closed submanifold in $H$. We assume the $K_j$ small enough so that they are pairwise disjoint, and in particular $\partial K_j\cap Z_0=\emptyset$. We may moreover assume that $H\cap K_j$ is diffeomorphic to a Euclidean ball. By the transversality theorem, we may perturb $K_j$  if necessary to have $(\widehat E_0\cap H) \pitchfork (\partial K_j \cap H)$ in $H$.

	Denote by $\frac12  K_j$ a smaller ball centered at $q_j$. Taking $W$ sufficiently small, we may assume that 
	\begin{equation}\label{eq:transversal1}
		\widehat E \pitchfork \left(H\setminus \cup_j \frac12 K_j\right),\qquad\forall E\in W. 
	\end{equation} 
This follows by the stability of transversal intersections, because $\widehat E$ is a smooth perturbation of $\widehat E_0$, which intersects $H$ transversally in an open neighborhood of $H\setminus \cup_j \textrm{int}(\frac12 K_j)$. Similarly we have 
	\begin{equation} 
		\widehat E \pitchfork \partial K_j\textrm{ in } M,\qquad \forall E\in W, 1\leq j\leq N
	\end{equation} 
		and 
		\begin{equation}\label{eq:transversal2}
			(\widehat E\cap H) \pitchfork (\partial K_j\cap H) \textrm{ in } H,\qquad \forall E\in W, 1\leq j\leq N. 
		\end{equation}
	
	For $\epsilon\in \{\pm\}^N$, denote $C_\epsilon=\cap_{j=1}^N C_j^{\epsilon_j}$. 
	Recall that $\widehat{\chi}_A$ is locally constant on $W_A=B_H^c$, and so is constant on any connected component of a  non-empty set $C_\epsilon$. Let us show there are integers $e_j=e_j(E_0)$ such that for any $\epsilon,\epsilon' \in \{\pm\}^N$ and any $E\in C_{\epsilon}, E'\in C_{\epsilon'}$ one has
	\begin{equation}\label{eq:indicator_difference} 
		\widehat \chi_A(E')-\widehat\chi_A(E)=\sum_{j:\epsilon_j<\epsilon_j'}e_j-\sum_{j:\epsilon_j'<\epsilon_j}e_j.
	\end{equation}

	For $E\in W$, denote $\Sigma_j(E):=\widehat E\cap \partial K_j\cap H$. As it is the transversal intersection of $\widehat E\cap H$ and $\partial K_j \cap H$ in $H$, it is a closed manifold of dimension $(m-k-2)$, and $\chi(\Sigma_j(E))$ is independent of $E\in W$.

    Let us distinguish the two cases under consideration. Assume first $A=H$ is a hypersurface. Since $K_i \cap  K_j = \emptyset$, we have 
    \begin{displaymath}
     \mathbbm 1_M=\sum_{j=1}^N (\mathbbm 1_{K_j}-\mathbbm 1_{\partial K_j})+\mathbbm 1_{\overline {K^c}},
    \end{displaymath}
    with $K:=\bigcup_{j=1}^N K_j$. Hence for $E\in W\setminus B_H$ we have
	\begin{displaymath}
	  \chi(E\cap A) =\sum_{j=1}^N \chi(E\cap K_j\cap A) -\sum_{j=1}^N\chi(\Sigma_j(E))+\chi(E\cap \overline{K^c} \cap A). 	 
	\end{displaymath}

The last summand is constant on $W$ by properties \eqref{eq:transversal1} and \eqref{eq:transversal2}. Consequently, for $E\in C_\epsilon$, $E'\in C_{\epsilon'}$ we have
\begin{displaymath}
  \widehat\chi_A(E')-\widehat \chi_A(E)=\sum_{j=1}^N \left(\chi(E'\cap A\cap K_j)- \chi(E\cap A\cap K_j)\right).
\end{displaymath}

The function $\chi(\bullet\cap A\cap K_j)$ is locally constant on $W\setminus F_j$, and it remains to define 
\begin{equation}\label{eq:local_difference}
 e_j:=\chi(\bullet\cap A\cap K_j)|_{C_j^+}- \chi(\bullet\cap A\cap K_j)|_{C_j^-}.
\end{equation}

    The case of full-dimensional $A$ is only slightly more involved. 
	If $(m-k)$ is odd, we have $\widehat{\chi_A}=\frac12\widehat{\chi_{\partial A}}$, reducing to the previous case. Thus assume $(m-k)$ is even.
	
	Write as before, whenever $E\in W\setminus B_H$,
	\begin{displaymath}
        \chi(E\cap A)=\sum_{j=1}^N \chi(E\cap K_j\cap A) -\sum_{j=1}^N\chi(E\cap \partial K_j\cap A)+\chi(E\cap \overline{K^c}\cap A).
	\end{displaymath}

	Note that for $E\in W\setminus B_H$, all intersections are manifolds with corners.	We have $\chi(E\cap  \partial K_j\cap A)=\frac12\chi(E\cap \partial K_j\cap \partial A)=\frac12\chi(\Sigma_j(E))$, thus it is constant in $W$.

 Set 
	\begin{displaymath}
        S_j:=\widehat E\cap \partial K_j, \quad S:=\widehat E\cap \partial A.
	\end{displaymath}
$S_j$ is a transversal intersection in $M$ for all $E\in W$ and hence a smooth hypersurface, while $S$ is given by a transversal intersection in $M$ and hence smooth for $E\in W\setminus B_H$. Moreover, $S$ is a smooth hypersurface outside of $\frac12 K_j$ for all $E\in W$. 

We claim that the intersection $S_j \cap S=\Sigma_j(E)$ is transversal in $\widehat E$ for all $E\in W$. For if the intersection is not transversal at $x$, then $T_x( \widehat E\cap \partial K_j)=T_x( \widehat E\cap \partial A)$. But by assumption $\widehat E \cap \partial A$ and $\partial K_j \cap \partial A$ intersect transversally in $\partial A$, in particular 
	\begin{displaymath}
	T_x(\widehat E \cap \partial A)+T_x(\partial K_j \cap \partial A)=T_x\partial A.
	\end{displaymath} 
	In conjunction with the previous equality, we get $T_x\partial A \subset T_x\partial K_j$, which is false.

    Let $X, Y\subset P$ be smooth domains in a manifold $P$, and assume $\partial X\pitchfork \partial Y$ and $X$ is compact. Let $Z\subset X\cap Y$ be the closure of a connected component of $X\cap Y$. Then $\chi(Z)$ is constant as $X, Y$ are perturbed while maintaining transversality.
	
	 Taking $P=\widehat E$, $X=E\cap A$ with $\partial X=S$ (which is a manifold for $E \in W \setminus B_H$), and $Y=E \cap K_j$ with $\partial Y=S_j$ we get that $\chi(E\cap A\cap K_j)$ is locally constant in $W\setminus B_H$. Taking $X=E\cap A$ with $\partial X=S$ (which is a manifold outside $\bigcup_j \frac12 K_j$ for all $E \in W$), $Y=E\cap \overline{K^c}$ with $\partial Y=\bigcup_j S_j$, it follows that $\chi(E\cap A\cap \overline{K^c})$ is constant in $W$. Thus we may define $e_j$ as in the previous case by eq. \eqref{eq:local_difference}.
	
	It follows from eq. \eqref{eq:indicator_difference} that for $E\in W$, $\widehat \chi_A(E)$ is a linear combination of the indicator functions of the connected components of the complements of the hypersurfaces $F_j$  in $W$. Therefore $\WF_{E_0}(\widehat\chi_A)\subset \bigcup_j N^*F_j$, concluding the proof.
	
	\endproof

\section{The Crofton wave front of LC-regular hypersurfaces}\label{sec:LC_crofton_WF}
	
	Let $(W,Q)$ be a vector space equipped with a quadratic form. We denote by $\Lambda_{k}^\nu(W) \subset \Gr_{k}(W)$ the collection of subspaces $E\subset W$ where $Q|_E$ has nullity $\nu$. We will need to describe those sets in several cases. 
	
	\begin{Proposition} \label{prop_dimension_lambda}
		Assume $(V,Q)$ has dimension $d$. 
		\begin{enumerate}
		 \item If $Q$ is non-degenerate, then $\Lambda_k^\nu(V)\subset\Gr_k(V)$ is a submanifold of dimension 
		 \begin{equation} \label{eq_dimension_lambda}
		    \dim \Lambda_{k}^\nu(V)=k(d-k)-\binom{\nu+1}{2}.
		  \end{equation}
		Writing $E_0:=E \cap E^Q$, we have
		\begin{equation} \label{eq_tangent_lambda}
		T_E \Lambda_{k}^\nu(V)=\{A \in \Hom(E,V/E): Q(Au,u)=0, \forall u \in E_0\}.
		\end{equation}
		\item If $Q$ has nullity $1$ and $E \in \Lambda_k^\nu(V)$ is such that $\Ker Q \cap E=\{0\}$, then $\Lambda_k^\nu(V)$ is a manifold near $E$ whose dimension is given by Eq. \eqref{eq_dimension_lambda}.
		\end{enumerate}
	\end{Proposition}

\proof
\begin{enumerate}
 \item See \cite[Proposition 4.2]{bernig_faifman_opq}. 
 \item Write $L_0:=\Ker(Q)$. Consider $W:=V \oplus \R$, and extend $Q$ as a non-degenerate quadratic form $\widetilde Q$ on $W$. Let us verify that the submanifolds $\Lambda_k^\nu(W)$ and $\Gr_k (V)$ intersect transversally in $\Gr_k(W)$ at $E$. 

 As $L_0 \not \subset E$, also $E^{\widetilde Q} \not \subset L_0^{\widetilde Q}=V$. Thus we can find a line $L \subset E^{\widetilde Q} \setminus V$. Now any linear map $A:E \to W/E$ decomposes as a sum $A=A_1+A_2$ with $A_1 \in \Hom(E,V/E)=T_E \Gr_k(V), A_2 \in \Hom(E,(E+L)/E) \subset \Hom(E,W/E)=T_E \Gr_k(W)$. Since $\widetilde Q(A_2u,v)=0$ for all $u,v \in E_0$, we have  $A_2\in T_E \Lambda^\nu_k(W)$ by \eqref{eq_tangent_lambda}.  
	
	This proves the claim.  As $\Lambda_k^\nu (V)=\Lambda_k^\nu(W) \cap \Gr_k (V)$, it is a manifold near $E$. The formula for the dimension then follows from the previous case.
\end{enumerate}

\begin{Corollary} \label{cor:degenerate_linear_algebra_bundle}
 Let $B$ be a smooth manifold, and $W$ a real vector bundle of rank $d$ over $B$. 
		Let $Q\in\Gamma(B, \Sym^2(W^*))$ be a smooth field of quadratic forms, of nullity at most $1$ for all $x$. Let $\Gr_k(W)$ be the corresponding bundle of $k$-subspaces over $B$, and consider $\Lambda_k^\nu (W)=\{(x,E)\in \Gr_k(W): E\in\Lambda^\nu_k(W_x,Q_x)\}$. If $(p,E)\in \Lambda_k^\nu (W)$ and $\Ker(Q_ x)\cap E=\{0\}$, then $\Lambda_k^\nu (W)$ is a manifold near $(p,E)$, of dimension 
		\begin{displaymath}
		\dim \Lambda_k^\nu (W)=\dim B+k(d-k)-{\nu+1 \choose 2}.
		\end{displaymath}
\end{Corollary}

\proof
	Using a local trivialization, this reduces to Proposition \ref{prop_dimension_lambda}. 
\endproof
	
	\begin{Lemma}\label{lem:degenerate_linear_algebra} 
		Let $W$ be a $d$-dimensional vector space equipped with a quadratic form $Q$ of nullity $1$ with kernel $L_0$. Assume $L_0\subset E_0\in\Lambda_k^\nu(W)$, and define the set $C \subset T_{E_0} \Gr_k(W)$ of all velocity vectors $E'(0)$ of smooth curves $E(t) \in\Lambda_k^\nu(W)$ with $E(0)=E_0$. Then $C$ is a cone over a closed manifold, and has dimension at most $k(d-k)-{\nu+1\choose 2}$.
	\end{Lemma}
		
	\proof
	If $Q$ is non-negative or non-positive definite, then $\Lambda_k^\nu(W)$ is empty if $\nu\geq 2$, while $\Lambda_k^1(W)=\{E\in\Gr_k(W):L_0\subset E\}$  is a manifold of dimension $(k-1)(d-k)$, whence the statement is trivial. We henceforth assume that is not the case, that is $Q$ has both positive and negative directions.

	Fix $W_0\subset W$ such that $W=L_0\oplus W_0$. Denote $$\Lambda_k^\nu(W, W_0)=\{E\in \Lambda_k^\nu(W): \Ker(Q|_E)\subset W_0\}.$$ Consider the map $I:\Gr_k(W)\setminus\Gr_k(W_0)\to \Gr_{k-1}(W_0)$ given by $I(E)=E\cap W_0$.
	
	We claim that the restriction 
	\begin{displaymath}
\tilde I:=I:\Lambda_k^\nu(W)\setminus\Lambda_k^\nu(W, W_0)\to \Lambda_{k-1}^{\nu-1}(W_0)	
	\end{displaymath}
	is well-defined. That is, the nullity of $I(E)$ is $(\nu-1)$ when $E\notin \Lambda_k^\nu(W, W_0)$.
	 Indeed, for such $E$ of nullity $\nu$, the nullity of $I(E)$ is clearly at least $(\nu-1)$. Since $\Ker(Q|_E) \not\subset W_0$, one can find $w_E\in \Ker(Q|_E)\setminus W_0$ so that $E=I(E)\oplus \Span(w_E)$. If $I(E)$ contains a $\nu$-dimensional subspace  $U$ that is $Q$-orthogonal to $I(E)$, then $U$ is also $Q$-orthogonal to $E$ as $Q(w_E, I(E))=0$. Hence $U \oplus \Span(w_E) \subset \Ker(Q|_E)$, and consequently the nullity of $E$ is at least $(\nu+1)$, a contradiction.
	
	We will describe the fiber $\tilde I^{-1}(F)$ of $F\in\Lambda_{k-1}^{\nu-1}(W_0)$. Denote 
	\begin{displaymath}
	 \Lambda_k^\nu(W,L_0)=\{E\in \Lambda_k^\nu(W): L_0\subset E\},
	\end{displaymath}
and $\pi_0:W \twoheadrightarrow W_0$ is the projection along $L_0$. Clearly $E\in \tilde I^{-1}(F)\cap \Lambda^\nu_k(W, L_0)$ if and only if $E=L_0 \oplus F$. 	
	
	Since $Q|_{W_0}$ is non-degenerate, $F_K:=F^Q\cap F\subset W_0$ is the kernel of $Q$ on $F^Q\cap W_0$. Thus the quotient space $$V_F:=F^Q\cap W_0/F_K$$ inherits a non-degenerate quadratic form, also denoted $Q$. Let $\pi_K:F^Q\cap W_0 \twoheadrightarrow V_F$ be the projection, and observe that $Q(\pi_Kx)=Q(x)$, so that 
	\begin{displaymath}
	 \pi_K(\Lambda^1_1(F^Q\cap W_0)\setminus \mathbb P(F_K))=\Lambda^1_1(V_F).
	\end{displaymath}

	Denote by $\tilde \pi_0:L_0\oplus V_F \twoheadrightarrow V_F$ the projection to the second summand, and observe that $L_0\oplus V_F$ is naturally equipped with a quadratic form $Q$ with nullity $1$ and kernel $L_0$.

	The manifold $\Lambda^1_1(V_F)$ embeds naturally into $\Lambda^1_1(L_0\oplus V_F)$ as the image of all lines of the form $0\oplus L$,  $L\in \Lambda^1_1(V_F)$. Put 
	\begin{displaymath}
		U_0(F)=\Lambda^1_1(L_0\oplus V_F)\setminus\Lambda^1_1(V_F),
	\end{displaymath} 
	which is a neighborhood of $L_0$. 
	
	 Define a smooth map
	\begin{displaymath}
	\Phi_F:\mathbb P(L_0\oplus V_F)\setminus \mathbb P(V_F) \to\Gr_k(W)
		\end{displaymath}
 as follows. For $N \in \mathbb P(L_0\oplus V_F)\setminus \mathbb P(V_F)$, choose any $0 \neq w \in N \subset L_0 \oplus V_F$. Let $\tilde w \in L_0 \oplus (F^Q \cap W_0)$ be a lift of $w$, and set $\tilde N:=\Span(\tilde w), \Phi_F(N):=\tilde N+F$. If $\tilde w'$ is another lift, then $\tilde w'-\tilde w \in F_K \subset F$, and hence $\Phi_F(N)$ is well-defined. In particular, $\Phi_F(L_0)=L_0+F$.

 \textit{Claim.} $\Phi_F(U_0(F))=\tilde I ^{-1}(F)$, and the restriction $\Phi_F:U_0(F)\to \tilde I^{-1}(F)$ is bijective.
	
	\textit{Proof.}	For the first statement, we consider two cases. If $N=L_0$ then $L_0\oplus F\in \Lambda^\nu_k(W)$, and $I(L_0\oplus F)=F$. If $N\neq L_0, N\notin \mathbb P(V_F)$ and $E=\widetilde N+F$, then one easily verifies that $Q|_N=0$ implies $Q|_{\widetilde N}=0$ and consequently $\Ker(Q|_E)=\Ker(Q|_F)\oplus \widetilde N$, so that again $E\in\Lambda^\nu_k(W)$, and clearly $I(E)=F$.
	
	For injectivity, first note that $\Phi_F(N)=L_0 \oplus F \iff N=L_0$. All other points $N\in U_0(F)$ lie inside a unique projective line $\mathbb P(L_0\oplus L)$ with $L\in \Lambda^1_1(V_F)$, and $N\neq L, L_0$. Put $E=\Phi_F(N)\in\tilde I^{-1}(F)\setminus \Lambda^\nu_k(W, L_0)$.
	
	Note that if $\tilde L, \tilde L'\in\Lambda_1^1(F^Q\cap W_0)$ are two lines such that $\pi_0(E)=\tilde L+F=\tilde L'+F$, the projections $\pi_K\tilde L,\pi_K\tilde L'\in\Lambda_1^1(V_F)$ must coincide: choosing $\tilde v'\in \tilde L'$ we can find $\tilde v \in \tilde L$ such that $\tilde v'=\tilde v+f$ for some $f\in F$. But $\tilde v,\tilde v'\in F^Q$, and so $\tilde v-\tilde v'\in F^Q\cap F=F_K$. 
	
	Since $\pi_0(E)=\pi_0(\widetilde N)+F$, it follows that $L_E:= \pi_K\pi_0(\widetilde N)\in\Lambda^1_1(V_F)$ is uniquely defined by $E$, and it holds that $N\in\mathbb P(L_0\oplus L_E)$. It remains to observe that $\Phi_F$ is obviously injective when restricted to $\mathbb P(L_0\oplus L_E)$.
		
	We verify that $\Phi_F$ is onto in the non-trivial case $E\in\tilde I^{-1}(F)\setminus \Lambda_1^1(W, L_0)$. Choose $w_E\in \Ker(Q|_E)\setminus W_0$, and decompose $w_E=w_1+w_0, w_1 \in W_0, w_0 \in L_0$. In particular, $Q(w_1, w_1)=Q(w_E, w_E)-2Q(w_E, w_0)+Q(w_0, w_0)=0$.  As $L_0\not\subset E$, we have $w_1 \neq 0$ and $w_1 = w_E-w_0 \in E^Q\subset F^Q$.  Since $w_0=w_E- w_1\notin E$ while $F=E\cap W_0\subset E$, it follows that $w_1\in \pi_0(E)\setminus F$ and so we can write $$\pi_0(E)=F+\tilde L, \qquad \tilde L=\Span(w_1)\in\Lambda_1^1(F^Q\cap W_0).$$
	Put $L=\pi_K\tilde L\in \Lambda^1_1(V_F)$. Now take $N\in \mathbb P(L_0\oplus L)$ to be the projection of the line $\Span(w_E)\in \mathbb P(L_0\oplus \tilde L\oplus F_K)$. Note also that $\mathbb P(L_0\oplus L)\subset \Lambda^1_1(L_0\oplus V_F)$. Clearly $\Phi_F(N)=E$, concluding the proof of the claim.

	Denote by $\Psi_F:\tilde I^{-1}(F)\to U_0(F)$ the inverse map of $\Phi_F$. We may fix an auxiliary Euclidean structure on $W$, and choose $w_E\in \Ker(Q|_E)$ as the unit vector forming the least angle with $L$. Now the map $E\mapsto\Ker(Q|_E)$, $\Lambda^\nu_k(W)\to\Gr_\nu(W)$ is smooth, in the sense that it restricts to a smooth map on every smooth curve in $\Lambda^\nu_k(W)\subset\Gr_\nu(W)$. This is because $\Ker(Q|_E)$ is the eigenspace of $0$, which has fixed dimension along the smooth curve. It follows that $E_t\in\Lambda^\nu_k(W)$ is a smooth curve in $\Gr_k(W)$ through $F\oplus L_0$, if and only if $\Psi_F(E_t)\in \Lambda^1_1(L_0\oplus V_F)$ is a smooth curve in $\mathbb P(L_0\oplus V_F)$ through $L_0$.
	
	Define the cone $C_0\subset T_{L_0}\mathbb P(L_0\oplus V_F)$ of tangent vectors to all curves $L_t$ through $L_0$ belonging to $\Lambda_1^1(L_0\oplus V_F)$, as well as the cone $C_F\subset T_{L_0 \oplus F}\Gr_k(W)$ that consists of all tangent vectors to smooth curves through $L_0 \oplus F$ inside $\Lambda^\nu_k(W)$.
	
	We conclude that the differential $D_{L_0}\Phi_F:T_{L_0}\mathbb P(L_0 \oplus V_F)\to T_{L_0 \oplus F}\Gr_k(W)$ restricts to a bijective map $A_F:C_0\to C_F$.

Considering all subspaces $F\in\Lambda_{k-1}^{\nu-1}(W_0)$ simultaneously, we have the fibration  
\begin{displaymath}
    \xymatrix{ \tilde I^{-1}(F)\ar@{^{(}->}[r]& \Lambda_k^\nu(W)\setminus\Lambda_k^\nu(W,W_0)\ar@{->>}[d]^{I}\\
		& \Lambda_{k-1}^{\nu-1}(W_0)}
\end{displaymath}
    The image of the section $F\mapsto L_0 \oplus F$  coincides with $\Lambda_k^\nu(W,L_0)$. 
    
	Therefore, the cone $C\subset T_{E_0}\Gr_k(W)$ has a linear factor that can be identified with $T_{E_0\cap W_0}\Lambda_{k-1}^{\nu-1}(W_0)$. Putting $F=E_0\cap W_0$, the cone $C/T_{F}\Lambda_{k-1}^{\nu-1}(W_0)$ is then identified with $C_F$.
	
	The dimension of $C_0$ can be readily computed. It can be identified with the abstract cone with base $ \Lambda_{1, +}^1(V_F)$, the manifold of oriented null lines in $V_F$. Since the form $Q$ on $V_F$ is non-degenerate and indefinite, we have $\dim  \Lambda_{1,+}^1(V_F)=\dim V_0-2=d-k-(\nu-1)-2$. Hence 
 \begin{displaymath}
 \dim C_0=\dim \Lambda_{1,+}^1(V_F)+1=d-k-\nu.
 \end{displaymath}
We have $\dim \Lambda_{k-1}^{\nu-1}(W_0)=(k-1)(d-k)-\binom{\nu}{2}$ by Proposition \ref{prop_dimension_lambda}, and so
 $$\dim C=\dim C_F+\dim\Lambda_{k-1}^{\nu-1}(W_0)= \dim C_0+ \dim\Lambda_{k-1}^{\nu-1}(W_0)=k(d-k)-\binom{\nu+1}{2}.$$
	\endproof
	
	We now turn to LC-regular submanifolds. First, we will need a simple fact on LC-regular metrics.
	
	\begin{Lemma}\label{lem:LC_gram_determinant}
		Let $(M,g)$ be LC-regular, and assume $g$ is degenerate on $T_pM$. Let $v_1(x),\dots,v_m(x)$ be any local frame near $p$, with Gram matrix $A(x)=(g(v_i, v_j))_{i,j=1}^m\in\Sym_m(\R)$. Then the condition $d_p(\det A)\neq 0$ is independent of the choice of the frame $(v_j)$. Moreover, if the nullity of $g_p$ is $\nu=1$, then $d_p(\det A)\neq 0$, and the degenerate subset of the metric near $p$ is a smooth hypersurface.
	\end{Lemma}
	\proof
	Let $\tilde v_1(x),\dots,\tilde v_m(x)$ be a different local frame with corresponding Gram matrix $\widetilde A(x)$. 
	Then the change of basis matrix $U(x)\in \GL(m)$ satisfies $\widetilde A(x)=U(x)^TA(x) U(x)$. By assumption, $\det A(p)=\det \widetilde A(p)=0$. Thus $d_p(\det \widetilde A)=\det U(p)^2d_p(\det A)$, which implies the first statement.
	
	For the second statement, choose coordinates $x_1,\dots,x_m$ on $M$ near $p$, and take $v_j=\frac{\partial}{\partial x_j}$. We may assume that $\Ker(g_p)=\Ker A(p)=\Span(v_m)$, and by assumption $A(p)$ has non degenerate principal $(m-1)$-minor. By LC-regularity, we can choose a curve $p(t)\in M$ with $p(0)=p$ and a smooth vector field $v(t)$ along it with $v(0)=v_m$ such that $\left.\frac{d}{dt}\right|_{t=0}g(v(t),v(t))=\left.\frac{d}{dt}\right|_{t=0}\langle A(p(t))e_m, e_m\rangle\neq 0$. It follows that $$\left.\frac{d}{dt}\right|_{t=0}\det A(p(t))=\left.\frac{d}{dt}\right|_{t=0}\langle A(p(t))e_{m},e_{m}\rangle\cdot \det (A(p))_{i,j=1}^{m-1}\neq 0.$$  
	As the degenerate subset of $g$ near $p$ is $\{x:\det A(x)=0\}$, the last assertion follows.
	\endproof
	
	The following is the main result of the section. We use the notation and terminology of Sections \ref{sec:crofton_functorial} and \ref{sec:crofton_functorial2}.
	\begin{Proposition}\label{prop:LC_regularity_implies_WF}
		Let $(V,Q)$ be a pseudo-Euclidean vector space of dimension $(n+1)$, and $M=Q^{-1}(r)$ with $r \in \{\pm1\}$ a pseudo-Riemannian space form. Let $H \subset M$ be an LC-regular hypersurface, and $E \in \Gr_{n-k+1}(V)$. Assume $E \in B_H$ is a regular tangent to $H$ of nullity $\nu$. Then each embedded part of $B_H$ through $E$ intersects $\Lambda^\nu_{n-k+1}(V)$ transversally at $E$. 
	\end{Proposition}
	
	\proof
	Denote the signature of $M$ by $(p_M,q_M)$. Write $g=Q|_H$, $\widehat E=E\cap M$. Since $H$ is a hypersurface, the nullity of $g_x$ is at most one for all $x\in H$. Define $\widetilde\Lambda^\nu_{n+1-k}:=\tau_M^{-1}\Lambda^\nu_{n+1-k}(V)$, which is a submanifold of $Z_M$ as $\tau_M:Z_M\to\Gr_{n+1-k}(V)$ is a submersion. The relevant maps are given by the following diagram. 
	
	\begin{center}
	\begin{tikzpicture}
		\matrix (m) [matrix of math nodes,row sep=3em,column sep=4em,minimum width=2em]
	   {
	   	  &  Z_M &  \widetilde \Lambda_{n+1-k}^\nu(V) \\
   	     & Z_H  & \widetilde B_H  \\
   	     \Lambda_{n+1-k}^\nu(V)  & \Gr_{n+1-k}(V) & B_H=\tau_H(\widetilde B_H)\\  };
        \path[left hook-stealth]
        (m-1-3) edge node [right] {}  (m-1-2)
        (m-2-3) edge node [right] {}  (m-2-2)
        (m-3-3) edge node [right] {}  (m-3-2);
        \path[right hook-stealth]
         (m-2-2) edge node [right] {} (m-1-2)
        (m-3-1) edge node [right] {} (m-3-2); 
        \path[draw,->>]
         (m-2-2) edge node [right] {$\tau_H$} (m-3-2)
        (m-1-2) edge[bend right] node [left] {$\tau_M$} (m-3-2)
        (m-2-3) edge node [right] {$\tilde \tau_H$} (m-3-3);
       	\end{tikzpicture}
\end{center}

	As $B_H$ is a hypersurface, one should show for every embedded part $F$ of $B_H$ through $E$ that $T_E\Lambda^\nu_{n+1-k}(V) \not \subset T_E F$. Assuming the contrary, there is $p\in H\cap E$ such that $T_E \Lambda^\nu_{n+1-k}(V) \subset d\tau_H(T_{p,E}\widetilde B_H)$. Observe that $\Ker(Q|_E)\subset T_p\widehat E$.
	
	Before proceeding with the more complicated general case, we consider the case $\nu=1$. Since $\dim \Lambda^1_{n+1-k}(V) =\dim B_H=k(n+1-k)-1$ by Proposition \ref{prop_dimension_lambda}, we have $T_E \Lambda^1_{n+1-k}(V) =d\tau_H(T_{p,E}\widetilde B_H)$. Let $v_0 \in E \cap T_pH$ be in the kernel of ${g|_{T_p\widehat E}}$. For any smooth curve $v_t\in TH$ through $v_0$, we may find a smooth curve $E_t\subset B_H$ through $E$ such that $v_t \in E_t$. Then $A:=\left.\frac{d}{dt}\right|_{t=0} E_t\in\Hom(E, V/E)$ satisfies $Q(Av_0, v_0)=0$, and by \eqref{eq_tangent_lambda} we have $A\in T_E \Lambda^1_{n+1-k}(V)$.

	It follows that $\left.\frac{d}{dt}\right|_{t=0} g(v_t)=2Q(v_0, Av_0)=0$, contradicting the LC-regularity of $H$.

	Let us now consider the general case. Fix an auxiliary Riemannian metric $h$ on $M$, and let $\rho_H$ be the least distance projection to $H$, defined and smooth in a neighborhood of $p$.
		
	By our assumption $\tilde \tau_H$ is an immersion at $(p,E)$.  Using Proposition \ref{prop_dimension_lambda} we see that
	\begin{equation}\label{eq:dimensions}
	\dim (d\tilde\tau_H)^{-1}T_{E}\Lambda^\nu_{n+1-k}(V)=\dim \Lambda^\nu_{n+1-k}(V)=k(n-k+1)-{\nu+1\choose 2}.
	\end{equation}

		\textit{Claim 1.} 
           Let $X,Y$ be manifolds, let $F:X \to Y$ be a submersion at $p\in Z$. Let $Z \subset X$ and $W \subset Y$ be embedded submanifolds, and $p \in Z$. Denote $f:=F|_Z$. Then 
           \begin{displaymath}
           	(d_pf)^{-1} T_{f(p)}W=T_pZ \cap T_p(F^{-1}W).
           \end{displaymath}
	
As $F^{-1}W$ is a submanifold near $p$ and $T_p(F^{-1}W)=(d_pF)^{-1}T_{F(p)}W$, the statement is clear. Applying the claim to $X=Z_M, Y=\Gr_{n+1-k}(V), Z=\tilde B_H, W=\Lambda_{n+1-k}^\nu(V)$ and $F=\tau_M$ yields

 \begin{displaymath}
	(d_{(p,E)}\tilde \tau_H)^{-1} T_E \Lambda_{n+1-k}^\nu(V)=T_{(p,E)}\tilde B_H \cap T_{(p,E)} \widetilde\Lambda_{n+1-k}^\nu(V).
\end{displaymath}

	\textbf{Case 1:} $\Ker(g_p)\cap E=\{0\}$.  
	
	Define 
\begin{displaymath}	
	\widetilde B^{\nu}_H:=\tilde \tau_H^{-1}\Lambda^\nu_{n-k+1}(V)=\{(q,F)\in \widetilde B_H:T_q\widehat F\in\Lambda^\nu_{n-k}(T_qH)\}.
\end{displaymath}

	By Corollary \ref{cor:degenerate_linear_algebra_bundle} we see that $\widetilde B^\nu_H$ is a smooth manifold near $(p,E)$, of dimension 
	\begin{align} \label{eq:smaller_dimension}
        \dim T_{p,E}\widetilde B_H^\nu&= \dim H + (n-k)(n-1-(n-k))-{\nu+1\choose 2}\\
        &=k(n-k+1)-1 -{\nu+1\choose 2}. \nonumber 
    \end{align}

	\textit{Claim 2.} $(d_{p,E}\tilde\tau_H)^{-1}T_{E}\Lambda^\nu_{n+1-k}(V)\subset T_{p,E}\widetilde B_H^\nu$.
	
\noindent	We postpone the proof. Combined with eqs. \eqref{eq:dimensions} and \eqref{eq:smaller_dimension} we get a contradiction. \\

	\textbf{Case 2:}  $T_pH$ has nullity one, and $\Ker(g_p)\subset E$. 
	
	Let $S\subset H$ be the degenerate subset of $g$. It follows from Lemma \ref{lem:LC_gram_determinant} that $S$ is a smooth hypersurface. Define $\widetilde B^\nu_H(S)=\{(q,F)\in \widetilde B_H: q\in S,F\in \Lambda^\nu_{n-k+1}(V)\}$. It is a fiber bundle over $S$ with fiber $\Lambda^\nu_{n-k}(\R^{p_M-1,q_M-1,1})$.
	
	\textit{Claim 3.} For any $(w,\xi)\in d_{p,E}\tilde\tau_H^{-1}T_E\Lambda^\nu_{n-k+1}(V)$ with $w\in T_pS$ there is a curve $(q(t), F(t))\in \widetilde B^\nu_H(S)$ with $(q'(0), F'(0))=(w, \xi)$. 
	
\noindent	Again we postpone the proof of the claim. The set of all vectors $(q'(0), F'(0))$ as in the claim defines, by Lemma \ref{lem:degenerate_linear_algebra}, a cone in $T_{p,E}\widetilde B_H$ of dimension 
	\begin{align*}
        N&=\dim S+ (n-k)((n-2)-(n-k-1))-{\nu +1\choose 2}  \\&= (n-k)(k-1)-{\nu+1\choose 2}+n-2<\dim d_{p,E}\tilde\tau_H^{-1}T_E\Lambda^\nu_{n-k+1}(V).
    \end{align*}
	It follows by the claim that we can find a curve $(p(t), E(t))\in \widetilde B_H$ through $(p,E)$ with $E'(0)\in T_{E}\Lambda_{n-k+1}^\nu(V)$ and $p'(0)\notin T_pS$.
	
	Let $v_{n-1}\in T_pH$ span $\Ker(g_p)$, and recall that $v_{n-1}\in T_p\widehat E$, in particular $v_{n-1}\in\Ker(Q|_{E})$. Choosing any smooth vector field $v_{n-1}(t)\in T_{p(t)}\widehat E(t)$, we find $\left.\frac{d}{dt}\right|_{t=0}Q(v_{n-1}(t))=2Q(v_{n-1}'(0), v_{n-1}(0))=0$.
	
	Choose a frame $(v_j(t))_{j=1}^{n-1}$ for $H$ along $p(t)$ with $v_{n-1}(0)=v_{n-1}$, such that $T_{p(t)}\widehat E(t)=\Span(v_k(t),\dots,v_{n-1}(t))$. Then 
	\begin{displaymath}
        \left.\frac{d}{dt}\right|_{t=0}\!\!\det (Q(v_i(t),v_j(t)))_{i,j=1}^{n-1}=\det (Q(v_i(t),v_j(t)))_{i,j=1}^{n-2}\left.\frac{d}{dt}\right|_{t=0}\!\!Q(v_{n-1}(t))=0,
    \end{displaymath} 
    which means that $d_p(\det g)(p'(0))=0$ by Lemma \ref{lem:LC_gram_determinant}. Since $\Ker d_p(\det g)=T_pS$ and $p'(0) \notin T_pS$, we get a contradiction. This completes the proof of the proposition, modulo the two claims we now proceed to prove.\endproof

\proof[Proof of Claim 2] 
	Consider a curve $(p(t), E(t))\in \widetilde \Lambda^\nu_{n-k+1}$ with $(p'(0), E'(0))\in T_{p, E}\widetilde B_H$.
	We ought to find a curve $(q(t), F(t))\in \widetilde B^\nu_H$ through $(p,E)$ with $(q'(0), F'(0))=(p'(0), E'(0))$. 
	
	Set $q(t)=\rho_H(p(t))$, evidently $q'(0)=p'(0)$. Fix a subspace $W_0\subset T_pH$ which is non-degenerate and contains $T_p \widehat E$. If $T_pH$ is non-degenerate, we can just take $W_0=T_pH$. Otherwise, $\dim \ker g_p=1$ and we may take any hyperplane $W_0 \subset T_pH$ which contains $T_p\widehat E$ and satisfies $W_0 \cap \ker g_p=\{0\}$. Now fix any linear map $A:W_0\to V/W_0$ which makes the following diagram commutative.
	
	\begin{displaymath}\xymatrixcolsep{5pc}
	\xymatrix{E \ar[r]^{E'(0)} \ar[d]^{i} & V/E \ar[d]_{\pi}\\
		W_0\ar[r]^{A}\ar[d]^{i} & V/W_0\ar[d]_{\pi}\\
		T_pH\oplus \R p \ar[r]^-{\left.\frac{d}{dt}\right|_{ t=0}(T_{q(t)}H\oplus \R q(t))}  & 	V/(T_pH\oplus \R p)
	}
	\end{displaymath}
	and use Lemma \ref{lem:pair_of_spaces} to find smooth paths $W(t) \supset E(t)$, $\widetilde W(t)\subset T_{q(t)}H\oplus \R q(t)$ with $W(0)=\widetilde W(0)=W_0$ and $W'(0)=\widetilde W'(0)=A$. For small $t$, $W(t), \widetilde W(t)$ are non-degenerate of fixed signature $(\alpha, \beta)$. 
	
	Consider the manifold $Z=\{(x, W)\in M\times \Gr_{\alpha+\beta}(V), x\in W, \mathrm{sign}(Q|_W)=(\alpha, \beta)\}$.
	Clearly $Z$ is a homogeneous space for $\OO(V, Q)$, with the equivariant projection $\pi_Z:\OO(V,Q)\to Z$ normalized by $\pi_Z(\id)=(p,W_0)$. We can fix a smooth section $X_Z:Z\to \OO(V,Q)$ near $(p, W_0)$ with $X_Z(p, W_0)=\id$ such that $\pi_Z\circ X_Z=\id$. Now define the smooth path $R_t\in \OO(V,Q)$ by 
	\begin{displaymath}
	R_t=X_Z(q(t),\widetilde W(t))\circ X_{ Z}(p(t), W(t))^{-1}.
	\end{displaymath} 
	Then $R_tp(t)=q(t)$, and $\left.\frac{d}{dt}\right|_{t=0}R_t=0$ since $\left.\frac{d}{dt}\right|_{t=0}W(t)=\left.\frac{d}{dt}\right|_{t=0}\widetilde W(t)$.
	
	Setting $F(t)=R_tE(t)$, we have $(q'(0), F'(0))=(p'(0), E'(0))$, and $(q(t), F(t)) \in\widetilde B^\nu_H$. This proves the claim. 
\endproof
	
\proof[Proof of Claim 3]
Consider a curve $(p(t), E(t))\in \widetilde\Lambda^\nu_{n-k+1}(V)$ through $(p, E)$, with $p'(0)=w\in T_pS$ and $(p'(0), E'(0))=(w,\xi)\in T_{p,E}\widetilde B_H$. Let $\rho_S:M\to S$ be the least distance projection with respect to $h$, well-defined and smooth in some neighborhood of $p$. Set $q(t)=\rho_S( p(t))$, clearly $q'(0)=p'(0)$. Denote $L_0=\Ker(g_p)\subset T_pH\cap E$, and extend to a smooth path of lines $L_t\subset E(t)\cap E(t)^{Q}\in\Lambda_\nu^\nu(T_{p(t)}M)$. Consider the manifold of pairs 
	\begin{displaymath}
	Z=\{(x, L): x\in M,L\in\Lambda^1_1(T_xM)\}. 
	\end{displaymath}
Clearly $Z$ is a homogeneous space for $\OO(V, Q)$, with the equivariant projection $\pi_Z:\OO(V,Q)\to Z$ normalized by $\pi_Z(\id)=(p,L_0)$. We can fix a smooth section $X_Z:Z\to \OO(V,Q)$ near $(p, L_0)$ with $X_Z(p, L_0)=\id$ such that $\pi_Z\circ X_Z=\id$. Now define the smooth path $R_t\in \OO(V,Q)$ by 
\begin{displaymath}
R_t=X_Z(q(t),\Ker(g_{q(t)}))\circ X_Z(p(t),L_t)^{-1}.
\end{displaymath}
Then $R_tp(t)=q(t)$, and $\left.\frac{d}{dt}\right|_{t=0}R_t=0$, provided that $\left.\frac{d}{dt}\right|_{t=0}L_t=\left.\frac{d}{dt}\right|_{t=0}\Ker(g_{q(t)})$.  
	
	Let us verify that $L_t$ can be chosen in this fashion. 
	In the following, we fix some Riemannian metric on various manifolds, and write $|x-y|_X$ for the corresponding distance between $x,y\in X$. We will also write, for two subspaces $E, F\subset V$, $\measuredangle(E,F)$ for the angle between them with respect to some Euclidean metric. This should not create ambiguity, as we will be concerned only with rough small scale asymptotics.
	
	As $(p'(0), E'(0))\in T_{p,E}\widetilde B_H$, we may find a curve $(\tilde p(t), \widetilde E(t))\in \widetilde B_H$ through $(p,E)$ with $(p'(0), E'(0))= (\tilde p'(0), \widetilde E'(0))$. Define $\widetilde H(t):=T_{\tilde p(t)}H\oplus \R \tilde p(t)$. It follows that  $\measuredangle(E(t),\widetilde H(t))=O(t^2)$, and by Lemma \ref{lem:pair_of_spaces} we have the commutative diagram 	
	\begin{displaymath}\xymatrixcolsep{3pc}
	\xymatrix{E \ar[r]^{E'(0)} \ar[d]^{i} & V/E \ar[d]_{\pi}\\
		\widetilde H(0)\ar[r]^{\widetilde H'(0)} & V/\widetilde H(0)
	}
	\end{displaymath}
	
	Taking the dual diagram and identifying $V=V^*$ using $Q$, we get 
	
	\begin{equation}\label{eq:diag1}
	\xymatrixcolsep{4pc}\xymatrix{L_0 \ar[r]^{f_0} \ar[d]^{i} & V/L_0 \ar[d]_{\pi}\\
		E^{Q}\ar[r]^{(E^{Q})'(0)} & V/E^{Q}\\}
	\end{equation}
	where $f_0= (\widetilde H^{Q})'(0)$. As $\tilde p'(0)=p'(0)=q'(0)$, it is clear that 
	\begin{displaymath}
	f_0= \left. \frac{d}{dt}\right|_{t=0}(T_{\tilde p(t)}H\oplus \R \tilde p(t))^{Q}=\left. \frac{d}{dt}\right|_{t=0}(T_{q(t)}H\oplus \R q(t))^{Q}=\left. \frac{d}{dt}\right|_{t=0}\Ker g_{q(t)}.
	\end{displaymath}
	
	By Lemma \ref{lem:pair_of_spaces}, we can find $L_t\subset E(t)^{Q}$ with $\left.\frac{d}{dt}\right|_{t=0}L_t=f_0$. Note that $L_t$ is not in general a null line of $Q$. We now proceed to modify the definition of $L_t$ to force it to be a null line.
	
	Observe that if $q\in S$, $\widetilde E\in\Gr_{n+1-k}(V)$ and $\Ker(g_{q})\subset \widetilde E$, then $\widetilde E^Q\subset T_qH\oplus \R q$.
	We have 
	\begin{align*}
	|\tilde p(t)-q(t)|_M & =O(t^2), \\
	|T_{\tilde p(t)}H-T_{q(t)}H|_{\Gr_{n-1}(V)} & =O(t^2),\\
	|L_t-\Ker(g_{q(t)})|_{\mathbb P(V)} & =O(t^2).
\end{align*}
	It follows that $\measuredangle(E(t)^{Q}, \widetilde H(t))=O(t^2)$, and so we may apply Lemma \ref{lem:pair_of_spaces} to get the commutative square 	\begin{displaymath}
	\xymatrixcolsep{3pc}\xymatrix{E^{Q} \ar[r]^{(E^{Q})'(0)} \ar[d]^{i} & V/E^Q \ar[d]_{\pi}\\
		\widetilde H(0)\ar[r]^{\widetilde H'(0)} & V/\widetilde H(0),
	}
	\end{displaymath}
	and by duality also 
	\begin{equation}\label{eq:diag2}
	\xymatrix{L_0 \ar[r]^{f_0} \ar[d]^{i} & V/L_0 \ar[d]_{\pi}\\
		E\ar[r]^-{E'(0)} & V/E.}
	\end{equation}
	
	Denote $K(t)=E(t)\cap E(t)^Q$, $K_0=K(0)$. Observe there is a natural inclusion $\alpha_K: V/K_0\hookrightarrow V/E\oplus V/E^Q$. It follows from Lemma \ref{lem:pair_of_spaces} applied to the inclusions $K(t)\subset E(t)$, $K(t)\subset E(t)^Q$ that $\alpha_K\circ K'(0):K_0\to  V/E\oplus V/E^Q$ coincides with $E'(0)\oplus (E^Q)'(0):K_0 \to V/E \oplus V/E^Q$.
	 
	Combining diagrams \eqref{eq:diag1} and \eqref{eq:diag2} then yields the commutative diagram
	
	\begin{equation*}
	\xymatrixcolsep{6pc}\xymatrix{L_0 \ar[r]^{f_0} \ar[d]^{i} & V/L_0 \ar[d]_{\pi}\\
		K_0\ar[r]^-{\alpha_K\circ K'(0)} & V/E\oplus V/E^{Q},}
	\end{equation*}
	and so also 
	\begin{equation*}
	\xymatrix{L_0 \ar[r]^{f_0} \ar[d]^{i} & V/L_0 \ar[d]_{\pi}\\
		K_0\ar[r]^-{K'(0)} & V/K_0.}
	\end{equation*}

	By Lemma \ref{lem:pair_of_spaces}, we may redefine $L_t$ such that $\left.\frac{d}{dt}\right|_{t=0}L_t=f_0$ and $L_t \subset K(t)=E(t)\cap E(t)^{Q}$. In particular, $L_t$ is a null line of $Q$.
	
	Setting $F(t):=R_tE(t)$ we have $q(t) \in F(t)$, $T_{q(t)}\widehat F(t)\subset T_{q(t)}H$ since $E(t)\subset L_t^{Q}$, and $F'(0)=E'(0)$ since $\left.\frac{d}{dt}\right|_{t=0}R_t=0$. This proves the claim. 
\endproof
	
	\begin{Remark}
		It is easy to see that the conclusion of the proposition with $k=n-1$ is equivalent to the LC-regularity of $H$. 
	\end{Remark}
	
	\begin{Corollary}\label{cor:CrWF_of_LC}
		Let $V,M, H, E$ be as in Proposition \ref{prop:LC_regularity_implies_WF} with $H$ compact without boundary, and $A\subset M$ is either $H$ itself or a domain with $\partial A=H$. Then $\Cr\WF_{E}(A)\cap N_E^*\Lambda^\nu_{n-k+1}(V)=\emptyset$.
	\end{Corollary}
	\proof
	Follows from Propositions \ref{prop:LC_regularity_implies_WF} and \ref{prop:base_of_induction}.
	\endproof
	
	\section{Construction of an invariant measure on the Grassmannian}
	\label{sec:opq_distributions}
	
	For $X \in \Sym_r(\R)$ and $\lambda \in \C$ we set, as in \cite{muro99},
	\begin{displaymath}
	|\det X|_p^\lambda:=\begin{cases} |\det X|^\lambda & \text{if } \mathrm{sign}(X)=(p,r-p)\\ 0 & \text{otherwise.} \end{cases}
	\end{displaymath}
	It is well-known, essentially due to Cayley and G\"arding \cite{garding48}, that $|\det X|_q^\lambda$ extends as a meromorphic in $\lambda$ family of generalized functions, which are in fact tempered distributions.
	
	We will use the set 
	\begin{displaymath}
	U_{\mathbb C}:= \{\Re\zeta>\frac12  \} \cup \left\{\mathrm{Im}\zeta>0\right\} \subset \C,
	\end{displaymath}
	and write $\sqrt{z}$ for the unique branch of the square root function on $U_\C$ such that $\sqrt z>0$ for $z>\frac12$.

	\subsection{A holomorphic family of Crofton measures}
	
	For the following, let $\Sym_r^+(\R)\subset\Sym_r(\R)$ be the cone of positive-definite matrices, and $\mathfrak h_r=\Sym_r(\R)\oplus \bi\Sym_r^+(\R)\subset\Sym_r(\C)$ the Siegel upper half space. The following is well-known.
	
	\begin{Lemma}\label{lem:siegel_determinant}
		For $Z\in\mathfrak h_r$, $\det Z\neq 0$. In particular, we can define for every $\lambda \in \C$ the holomorphic function $Z \mapsto (\det Z)^\lambda$, normalized by $\lim_{\epsilon \to 0^+} \det (I_r+\bi\epsilon I_r)^\lambda=1$. Moreover, all eigenvalues of $Z\in\mathfrak h_r$ lie in the upper half plane of $\C$.
	\end{Lemma}
	
	\proof
	Write $Z=X+\bi Y$, $Y>0$. Let $Q_X(v)=\langle Xv, v\rangle, Q_Y=\langle Yv,v\rangle$ be the corresponding quadratic forms. Choose a basis $u_j$ such that the Gram matrix of $Y$ is $I_r$, and of $X$ is diagonal: $D=\mathrm{diag}(d_j)$. Since $D+\bi I_r=U^TZU$ with $U$ invertible, and $\det(D+\bi I_r)=\prod (d_j+\bi)\neq 0$, it follows that $\det Z\neq 0$. Since $\mathfrak h_r$ is simply connected, the second statement follows.

	For the last statement, we first note there can be no real eigenvalues. Indeed by the first statement, $\det(X+\bi Y-\lambda I_r)=\det ((X-\lambda I_r) + \bi Y)\neq 0$ for $\lambda\in\R$. Next we argue as before and select a diagonalizing basis, given by $U\in \mathrm{GL}(r)$. We furthermore may assume that $\det U>0$, by interchanging two basis elements. Choose a smooth path $U_t\in \mathrm{GL}(r)$ with $U_0=\mathrm{Id}$ and $U_1=U$. Then $U_t^TZU_t\in\mathfrak h_r$ is a smooth path. For $t=1$, the endpoint is $D+\bi I_r$, which has all eigenvalues in the upper half plane. If $Z$ has eigenvalues in the lower half-plane, then by continuity for some $t$ there will be a real eigenvalue, a contradiction.
	\endproof
	
	 Recall for the following that given a non-degenerate quadratic form $Q$ on $V$, a \textit{compatible Euclidean form} is any positive-definite form $P$ such that $V$ admits a decomposition $V=V_+\oplus V_-$ which is both $P$- and $Q$-orthogonal, and $Q|_{V_\pm}=\pm P|_{V_{\pm}}$.
	
	 From here on, let $V=\R^{p} \oplus \R^{q}=\R^{n+1}$ with the standard quadratic form $Q$ of signature $(p,q)$ and the corresponding compatible Euclidean form $P_0$. Define a family of complex-valued quadratic forms $Q_\zeta$ on $V$ with $\zeta\in\mathbb C$, by 
	\begin{displaymath}
	 Q_\zeta:=Q+2\zeta P_0. 
	\end{displaymath}
We then have 
	\begin{displaymath}
	 Q_\zeta(x,y):= \begin{cases}
	                               (2\zeta+1) P_0(x,y) & x,y \in \R^p\\
	                               (2\zeta-1) P_0(x,y) & x,y \in \R^q\\
	                               0 & x \in \R^p, y \in \R^q.
	                              \end{cases}
	\end{displaymath}
	
	Observe that $Q_\zeta$ is real and positive-definite for $\zeta>\frac12$, and $Q_0=Q$. Furthermore by Lemma \ref{lem:siegel_determinant}, $\det Q_\zeta\neq 0$ for $\zeta\in U_{\mathbb C}$, as either $Q_\zeta$ or $\bi Q_\zeta$ lies in $\mathfrak h_{n+1}$. Note that a complex-valued non-degenerate quadratic form $Q$ on a real vector space $E$ defines an element $\vol_Q^2\in \Dens_\C(E)^2$, and given a branch of square root we also get a complex-valued density $\vol_Q\in \Dens_\C(E)$. 
	
By a ($P$-)frame on an open subset $U\subset \Gr_{n+1-k}(V)$ we understand a smooth section of the Stiefel manifold of $P$-orthonormal $(n+1-k)$-frames in $V$ defined over $U$. 
		For a subspace $E\in\Gr_{n+1-k}(V)$ and $\zeta\in U_{\mathbb C}$, choose a frame $u_i(E)$ on $U$ and define $X^P_\zeta:U\to\Sym_{n+1-k}(\C)$ to be the corresponding Gram matrix of $Q_\zeta$, namely $X^P_\zeta(E)=(Q_\zeta(u_i(E), u_j(E)))_{i,j=1}^{n+1-k}$.
	
	Note that if $\widetilde X_{\zeta}$ is the corresponding matrix for a different frame $\widetilde u_i(E)$ on $U$ then 
	\begin{equation}\label{eq:frame_change}
		\widetilde X_\zeta(E)=B(E)^T X_\zeta(E) B(E)
	\end{equation} 
	for some smooth map $B:U\to\OO(n+1-k)$.
	
Observe that by eq. \eqref{eq:frame_change}, $\det(X^P_\zeta)$ is independent of the choice of $P$-orthonormal bases of $E$. Moreover, either the real or imaginary part of $Q_\zeta|_E$ is positive-definite, and consequently by Lemma \ref{lem:siegel_determinant}, $\det X^P_\zeta(E)\neq 0$. 
	
	The function $\det(X^P_\zeta)^\lambda\in C^\infty(\Gr_{n+1-k}(V),\C)$ is thus well-defined for all $P$ and $\lambda\in\C$, and analytic in $\zeta\in U_{\mathbb C}$, once the normalization $\det(X^P_1)^\lambda>0$ is fixed, as $U_{\mathbb C}$ is simply-connected.

	Define the smooth measure $\widetilde m^{\zeta, P}_k$ on the Grassmannian $\Gr_{n+1-k}(V)$ by 
	\begin{displaymath}
	d\widetilde m^{\zeta, P}_k:=\det(X^P_\zeta)^{-\frac{n+1}{2}}(E)d\sigma_P(E),
	\end{displaymath}
	where $d\sigma_P(E)$ is the $\OO(P)$-invariant probability measure on the Grassmannian.
	
	\begin{Proposition}
		The complex-valued smooth measure 
		\begin{displaymath}
		m^\zeta_k:=(2\zeta+1)^{\frac{p(n+1-k)}{2}}(2\zeta-1)^{\frac{q(n+1-k)}{2}}\widetilde m_k^{\zeta,P_0}, \quad \zeta\in U_{\mathbb C}.
		\end{displaymath}
		depends analytically on $\zeta$ and is normalized, i.e. 
		\begin{displaymath}
		\int_{\Gr_{n+1-k}(V)} dm^\zeta_k=1.
		\end{displaymath}
	\end{Proposition}
	
	\proof
	The first statement is clear. For the second, we first see how $\widetilde m^{\zeta, P}_k$ depends on $P$. Let $P_1,P_2$ be two Euclidean structures on $V$. 
	From the natural identification $T_E\Gr_{n+1-k}(V)=E^*\otimes V/E$ we obtain that 
	\begin{displaymath}
	\Dens(T_E\Gr_{n+1-k}(V))=\Dens^*(E)^{n+1}\otimes \Dens(V)^{n+1-k}.
	\end{displaymath}
	
	Spelling this out gives 
	\begin{displaymath}
	\frac{d\sigma_{P_1}(E)}{d\sigma_{P_2}(E)}=\left(\frac{\vol_{P_1|_E}}{\vol_{P_2|_E}}\right)^{-(n+1)} \left(\frac{\vol_{P_1}}{\vol_{P_2}}\right)^{n+1-k},
	\end{displaymath}
	
	Since 
	\begin{displaymath}
	\det X_\zeta^{P_i}(E)=\frac{\vol^2_{Q_\zeta|_E}}{\vol^2_{P_i|_E}},
	\end{displaymath}
	we find that
	\begin{displaymath}
	\widetilde m^{\zeta, P_1}_k=\left(\frac{\vol_{P_1}}{\vol_{P_2}}\right)^{n+1-k}\widetilde m^{\zeta, P_2}_k.
	\end{displaymath}
	
	For $\zeta>\frac{1}{2}$, $Q_\zeta$ is a Euclidean structure. Then 
	\begin{align*}
	1 & =\int \widetilde m^{\zeta, Q_\zeta}_k= \left(\frac{\vol_{Q_\zeta}}{\vol_{P_0}}\right)^{n+1-k} \int \widetilde m^{\zeta, P_0}_k\\
	& = \sqrt{2\zeta+1}^{p(n+1-k)} \sqrt{2\zeta-1}^{q(n+1-k)} \int \widetilde m^{\zeta, P_0}_k=\int m^\zeta_k.
	\end{align*} 
	By uniqueness of the analytic continuation, this formula also holds for general $\zeta\in U_{\mathbb C}$.
	\endproof
	 
	\subsection{Homogeneous distributions on the space of symmetric matrices}
	
	\begin{Lemma}\label{lem_f_lambda}
		The meromorphic family of generalized functions
		\begin{displaymath}
		f_\lambda(X):=\sum_{h=0}^r e^{\bi \pi h\lambda}|\det X|^\lambda_{r-h} \in C^{-\infty}(\Sym_r(\R))
		\end{displaymath}
		is analytic in $\lambda\in\mathbb C$ and satisfies 
		\begin{equation} \label{eq_reflection_f_lambda}
		 f_\lambda(-X)=e^{\bi \pi r \lambda} \overline{f_\lambda(X)}.
		\end{equation}
	\end{Lemma}
	
	\proof
	We recall some results from \cite{muro99}. Consider a linear combination 
		\begin{displaymath}
		g_\lambda(X):=\sum_{h=0}^r a_h |\det X|^\lambda_{r-h}
		\end{displaymath} 
		with constant coefficients $a_h \in \C$ and set $\vec a:=(a_0,\ldots,a_r) \in \C^{r+1}$. Then $g_\lambda \in C^{-\infty}(\Sym_r(\R))$ is meromorphic with possible poles in the set $\left\{-m, -\frac{2m+1}{2}: m\geq 1\right\}$.
		
		The order of the pole at $s$ in this set can be obtained as follows. Set $\epsilon=-1$ if $s$ is an even integer and $\epsilon=1$ otherwise. Define inductively linear maps $d^{(m)}=(d^{(m)}_0,\ldots,d^{(m)}_{r+1-m}):\C^{r+1} \to \C^{r+1-m}$ by setting 
		\begin{align*}
		d^{(0)}_h(\vec a) & := a_h\\
		d^{(1)}_h(\vec a) & :=a_h+\epsilon a_{h+1} \\
		d_h^{(2l+1)}(\vec a) & := d^{(2l-1)}_h-d^{(2l-1)}_{h+2}, \quad l=1,2,\ldots\\
		d_h^{(2l)}(\vec a) & := d^{(2l-2)}_h+d^{(2l-2)}_{h+2}, \quad l=1,2,\ldots
		\end{align*}
		
		Then $g_\lambda$ has a pole of order $p$ at $s=-\frac{2m+1}{2}$ if and only if $d^{2p}(\vec a) \neq 0, d^{2p+2}(\vec a)=0$. Similarly, $g_\lambda$ has a pole of order $p$ at $s=-m$ if and only if  $d^{2p-1}(\vec a) \neq 0, d^{2p+1}(\vec a)=0$. Here we use the convention that $d^{(m)}=0$ if $m >r+1$ and that a pole of order $0$ is a point of analyticity.
		
		In our situation, the coefficients $a_h=a_h(\lambda)=e^{\bi \pi h \lambda}$ depend on $\lambda$ and we cannot apply Muro's result directly. However, writing 
		\begin{displaymath}
		f_\lambda(X)=\sum_{h=0}^r a_h(\lambda) |\det X|^\lambda_{r-h}=\sum_{j=0}^\infty  \frac{(\lambda-s)^j}{j!}  \sum_{h=0}^r a_h^{(j)}(s) |\det X|^\lambda_{r-h},
		\end{displaymath}
		we see that it is enough to prove that the order of the pole of $\sum_{h=0}^r a_h^{(j)}(s) |\det X|^\lambda_{r-h}$ at $\lambda=s$ is at most $j$ for all $j$. 
		
		By induction we find that for all $l=0,1,\ldots$
		\begin{align*}
		d^{2l}_h(\vec a(\lambda)) & = e^{\bi \pi h \lambda} (1+e^{2\pi \bi \lambda})^l, \\
		d^{2l+1}_h(\vec a(\lambda)) & = e^{\bi \pi h \lambda} (1+\epsilon e^{\bi \pi \lambda}) (1-e^{2\pi \bi \lambda})^l, 
		\end{align*}
		and hence 
		\begin{align*}
		d^{2j+2}_h\left(\vec a^{(j)}\left(-\frac{2m+1}{2}\right)\right) & = \left.\frac{d^j}{d\lambda^j}\right|_{\lambda=-\frac{2m+1}{2}} e^{\bi \pi h \lambda} (1+e^{2\pi \bi \lambda})^{j+1}=0, \\
		d^{2j+1}_h\left(\vec a^{(j)}(-m)\right) & = \left.\frac{d^j}{d\lambda^j}\right|_{\lambda=-m} e^{\bi \pi h \lambda} (1+\epsilon e^{\bi \pi \lambda}) (1-e^{2\pi \bi \lambda})^{j}=0, 
		\end{align*}
		which finishes the proof.
	
	\endproof

	\begin{Proposition}\label{prop:upper_complex_limit}
	Let $C' \subset \Sym_r^+(\R)$ be a closed convex cone. Then 
	\begin{displaymath}
	 \lim_{Y\to 0, Y \in C'} \det(X+\bi Y)^\lambda=f_\lambda(X)
	\end{displaymath}
 in the strong topology on tempered distributions on $\Sym_r(\R)$.
	\end{Proposition}
	
	\proof
	For $\Re\lambda\geq 0$ the statement is easy, so in the following we assume $\Re\lambda<0$. 
	
	First we claim that the limit exists.  We will show that there are constants $\alpha=\alpha(\lambda)\geq 0$ and $b'=b(\lambda,C')$ such that
	\begin{equation}\label{eq:det_bound}
	|\det(X+\bi Y)^\lambda|  \leq b' \|Y\|^{-\alpha},\qquad  \forall X\in\Sym_r(\R), \forall Y\in C'\setminus\{0\}
	\end{equation}
	
	First take $Y=I_r$. Letting $(\mu_j)_{j=1}^r\subset\R$ be the eigenvalues of $X$, we get 
	\begin{align*}
	|\det(X+\bi I_r)^\lambda| &= \prod|\mu_j+\bi |^{{\mathrm{Re}}\lambda} e^{-\mathrm{Im}\lambda\cdot\mathrm{Arg}(\mu_j+\bi)}\\
	& \leq e^{\pi r|\mathrm{Im}\lambda|}\prod(\mu_j^2+1)^{\frac{\Re\lambda}{2}}\\
	& \leq e^{\pi r|\mathrm{Im}\lambda|}.
	\end{align*}
	
	Now for general $Y$, we have 
	\begin{displaymath}
	|\det (X+\bi Y)^\lambda|=|\det Y|^{\Re\lambda}|\det(\sqrt Y^{-1} X \sqrt Y^{-1} +\bi I)^\lambda| \leq e^{\pi r {|\mathrm {Im}\lambda|}}|\det Y|^{\Re\lambda},
	\end{displaymath}
	and letting $c:=\sup\left\{ \frac{\|Y\|^r}{|\det Y|}: Y\in C'\right\}$, we conclude that \eqref{eq:det_bound} holds with  $b'=c^{ -\Re\lambda} e^{\pi r |\mathrm {Im}\lambda|}$, and $\alpha=-r\Re\lambda\geq 0$.
	
	It now follows from \cite[Section 26.3]{vladimirov66} that the limit 
	\begin{displaymath}
	 \det(X+\bi 0)^\lambda:=\lim_{Y\to 0, Y \in C'}\det(X+\bi Y)^\lambda\in \mathcal S'
	\end{displaymath}
 exists in the strong topology on the space of tempered distributions of order $\lceil r|\Re\lambda|\rceil+{r+1\choose 2}+3$. 
	
	It remains to verify that $\det(X+\bi 0)^\lambda=f_\lambda(X)$ for $\Re\lambda<0$. Denote 
	\begin{displaymath}
	H_\epsilon=\{X+\mathbf i\epsilon I_r: X\in \Sym_r(\R)\}\subset \Sym_r(\C).
	\end{displaymath}
	Let $\psi(X)$ be a Schwartz function on $\Sym_r(\R)$, which is the Fourier transform of a compactly supported smooth function, in particular it has an analytic extension to $\Sym_r(\C)$. Writing $dZ=\largewedge_{i=1}^r \largewedge_{j=i}^r dz_{ij}$, the integral $\int_{H_\epsilon} \psi(Z)\det(Z)^\lambda dZ$ is convergent, since $\psi$ is rapidly decaying at infinity and $\det(Z)^\lambda$ of polynomial growth. It is clearly analytic in $\lambda\in \C$. Furthermore, its value is independent of $\epsilon$ as the integrand is a closed form, rapidly decaying at infinity. For $\lambda>0$ we have 
	\begin{displaymath}
	(\mu_j+\bi \epsilon)^\lambda=|\mu_j+\bi \epsilon|^\lambda e^{\bi \lambda\mathrm{Arg}(\mu_j+\bi \epsilon)}\to |\mu_j|^\lambda e^{\bi\lambda \frac{\pi}{2}(1-\sign(\mu_j))}. 
	\end{displaymath}

	Hence  
	\begin{displaymath}
	\det(X+\bi \epsilon I_r)^\lambda=\prod_{j=1}^r (\mu_j+\bi \epsilon)^\lambda \to \prod |\mu_j|^\lambda e^{\bi\pi \#\{\mu_j<0\}\lambda}=f_\lambda(X),
	\end{displaymath}
	and so 
	\begin{displaymath}
	\int_{H_\epsilon} \psi(Z)\det(Z)^\lambda dZ\to \int_{\Sym_r(\R)} \psi(X) f_\lambda(X)dX
	\end{displaymath}
	for $\lambda>0$. By analytic extension we conclude that for all $\lambda\in\C$ and $\epsilon>0$, 
	\begin{displaymath}
	\int_{H_\epsilon} \psi(Z)\det(Z)^\lambda dZ =\int_{\Sym_r(\R)} \psi(X) f_\lambda(X)dX,
	\end{displaymath}
	that is 
	\begin{displaymath}
	\int_{\Sym_r(\R)} \psi(X+\bi \epsilon I_r)\det(X+\bi \epsilon I_r)^\lambda dX =\int_{\Sym_r(\R
		)} \psi(X) f_\lambda(X)dX. 
	\end{displaymath}
	
	As $\epsilon\to 0$, we have  $\psi(X+\bi \epsilon I_r)\to \psi(X)$ in $\mathcal S$, while $\det(X+\bi \epsilon I_r)^\lambda\to\det(X+\bi 0)^\lambda$ in $\mathcal S'$. It follows by continuity that 
	\begin{displaymath}
	\int_{\Sym_r(\R)} \psi(X)\det(X+\bi 0)^\lambda dX = \int_{\Sym_r(\R)} \psi(X) f_\lambda(X)dX.
	\end{displaymath}
	Finally, noting that the set of Schwartz functions such as $\psi$ is dense, we conclude that $\det(X+\bi 0)^\lambda=f_\lambda$ for all $\lambda\in\C$, as claimed.
	\endproof
	Henceforth we use $f_\lambda$ and $\det(X+\bi 0)^\lambda$ interchangeably. 
	
	The following statement shows that the convergence along $Y\in \R_+ I_r$ holds in a finer topology, namely the \emph{normal H\"ormander topology}. We refer to \cite{brouder_dang_helein} for its definition (where it is called \emph{normal topology}). The main point for us is that the operation of pull-back of generalized sections is continuous in this topology, provided some condition on wave fronts is satisfied.  
	
	We do not know if an analogue of the following proposition holds for an arbitrary distributional boundary value; the proof below is tailored to our particular case, and in essence leverages strong convergence by induction on dimension.
	
	\begin{Proposition}\label{prop:hormander_convergence}
		Denote  $N^*\Gamma^r=\cup_{\nu=0}^r N^* \Gamma^r_\nu\subset T^*\Sym_r(\R)$, where $\Gamma^r_\nu$ consists of all matrices of nullity $\nu$.
		It then holds for all $\lambda\in \mathbb C$ that $\det(X+\mathbf i\epsilon I_r)^\lambda\to f_\lambda(X)$ in $C^{-\infty}_{N^*\Gamma^r}(\Sym_r(\R))$ in the normal H\"ormander topology .
	\end{Proposition}
	
	\proof
	First note that $g^*f_\lambda=\det(g)^{2\lambda}f_\lambda$ for all $g\in\GL(r)$. Thus we have the differential equations $(\underline A-2\lambda\tr(A))f_\lambda=0$, where $\underline A$ is the vector field defined by the infinitesimal action of $A\in\mathfrak{gl}(r)$. It follows from \cite[Theorem 8.3.1]{hoermander_pde1} that $\WF(f_\lambda)\subset N^*\Gamma^r$. 
	
	We proceed by induction on $r$, the case $r=1$ being trivial. Since $\det(X+\mathbf i\epsilon I_r)^\lambda\to \det (X+\mathbf i0)^\lambda$ in the strong topology by Proposition \ref{prop:upper_complex_limit} and $N^*_0\Gamma^r_r=T_0^*\Sym_r(\R)$, it remains to consider convergence in $\Sym_r(\R)\setminus\{0\}$.
	Consider a matrix $Y\in\Sym_r(\R)$ of nullity $\nu<r$. Let $E_0$ be its kernel, and $F_0=E_0^\perp$. There is then a unique  map $E:U\to \Gr_\nu(\R^r)$ in a neighborhood $U$ of $Y$ such that $E(Y)=E_0$, and $E(X)$ is an invariant subspace of $X$. Here and in the following, $U$ is assumed sufficiently small for various purposes.
	
	We claim $E=E(X)$ is smooth. Indeed, consider $Z=\{(X,F): X(F)= F\}\subset\Sym_r(\R)\times\Gr_{r-\nu}(\R^r)$. Clearly $Z$ is the graph of a unique function $F=F(X)$ near $(Y, F_0)$. Let us check that $Z$ is a manifold near $(Y, F_0)$. 
	Define $\alpha:U\times\Gr_{r-\nu}(\R^r)\to\Gr_{r-\nu}(\R^r)\times\Gr_{r-\nu}(\R^r)$ by $\alpha(X, F)=(F, X(F))$. Then $Z=\alpha^{-1}( \Delta)$, where $\Delta$ is the diagonal. Let us verify that $\alpha$ is a submersion at $(Y, F_0)$. 
	
	For $M\in\Sym_r(\R)$ and $H\in T_{F_0}\Gr_{r-\nu}(\R^r)=\mathrm{Hom}(F_0,\R^r/F_0)$, one computes $d_{Y,F_0}\alpha(M, H)=(H, Y\circ H+M|_{F_0\to \R^r/F_0})=(H, M|_{F_0\to \R^r/F_0})$, since by construction $Y: \R^r/F_0\to \R^r/F_0$ is the zero map. Noting that any linear map $F_0\to  \R^r/F_0$ is induced by a symmetric matrix mapping $M:\R^r \to \R^r$, it follows that $\alpha$ is submersive and $Z$ is a manifold. Further, 
	\begin{displaymath}
        T_{Y,F_0}Z=\{(M,H):M\in\Sym_r(\R), H=M|_{F_0\to \R^r/F_0} \}. 
	\end{displaymath}

	In particular if $(0, H)\in T_{Y,F_0}Z$, then we must have $H=0$. It follows that $F(X)$ is smooth in $U$, and therefore so is $E(X)=F(X)^\perp$.
	
	Choose arbitrary orthonormal frames $e_i(X)$ for $E(X)$ and $f_i(X)$ for $F(X)=E(X)^\perp$ depending smoothly on $X$.
	Define 
	\begin{displaymath}
            A:U\to\Sym_\nu(\R),\quad B:U\to \Sym_{r-\nu}(\R)
	\end{displaymath}
    by $$A(X)=(\langle Xe_i(X), e_j(X)\rangle ),\quad B(X)=(\langle X f_i(X), f_j(X)\rangle ).$$ Then $A$ is a submersion in $U$. Indeed one has
	\begin{displaymath}
        d_YA(M)_{i,j}=\langle M e_i(Y), e_j(Y)\rangle+\langle Ye_i(Y), d_Ye_j(M)\rangle +\langle Ye_j(Y), d_Ye_i(M)\rangle, 
	\end{displaymath}
	and the last two summands vanish as $e_i(Y), e_j(Y)\in E_0$. It follows that $d_YA:\Sym_r(\R)\to \Sym_\nu(\R)$ is surjective, and so $A$ is submersive near $Y$.
	
	It holds that
	\begin{displaymath}
        \det(X+\mathbf i\epsilon I_r)^\lambda=A^*\det(X_1+\mathbf i\epsilon I_\nu) ^\lambda \det(B(X)+\mathbf i\epsilon I_{r-\nu})^\lambda, \quad  X_1 \in \Sym_{\nu}(\R).
	\end{displaymath}
	As $ B(X)$ is non-degenerate, the second factor is a smooth function in $(X,\epsilon)\in U\times\R$.
	
	For the first factor, we have by the induction assumption that $\det(X_1+\mathbf i\epsilon I_\nu) ^\lambda\to \det(X_1+\mathbf i0)^\lambda$ in the normal topology on $C^{-\infty}_{N^*\Gamma^\nu}(\Sym_\nu(\R))$. It then holds that \begin{displaymath}\WF(A^*\det(X_1+\mathbf i0)^\lambda)\subset A^*(N^*\Gamma^\nu) =N^*(A^{-1}\Gamma^\nu)= N^*(\Gamma^r\cap U),\end{displaymath}
	and by \cite{brouder_dang_helein}, $A^*\det(X_1+\mathbf i\epsilon I_\nu)^\lambda\to A^*\det (X_1+\mathbf i0)^\lambda$ in the normal H\"ormander topology on $C^{-\infty}_{N^*\Gamma^r}(\Sym_r(\R))$. We conclude that
	\begin{displaymath}\det(X+\mathbf i\epsilon I_r)^\lambda\to \det(X+\mathbf i0)^\lambda \end{displaymath}
	in the normal H\"ormander topology on $C^{-\infty}_{N^*\Gamma^\nu}(\Sym_r(\R))$.
	\endproof
	
	\begin{Remark}
		Using the Hilbert-Schmidt inner product to identify $T_0^*\Sym_r(\R)=\Sym_r(\R)$, the statement of the proposition in fact holds with all conormal cones intersected with $\Sym_r^+(\R)$, which follows from \cite[Theorem 8.1.6]{hoermander_pde1}. 
	\end{Remark}
	
	\subsection{Construction of an $\OO(p,q)$-invariant Crofton distribution}
	In \cite[Proposition 4.9]{faifman_crofton}, an $\OO(p,q)$-invariant distribution was constructed on $\Gr_{n+1-k}(V)$. To avoid singularities, it made use of several auxiliary Euclidean structures that gave rise to several locally defined distributions that were then patched together. For the present paper, we will need an alternative construction making use of a single Euclidean structure. To handle the singularities, we must carefully monitor the wave front set.  

Write $P=P_0$ for the Euclidean structure on $V$. Let $dE=d\sigma_{P}$ denote the $\OO(P)$-invariant probability measure on $\Gr_{n+1-k}(V)$.  Decompose $V=V_P^+\oplus V_P^-$ such that $Q|_{V_P^\pm}=\pm P|_{V_P^\pm}$. 

We denote $\kappa=n+1-k$. An orthonormal basis  $u_1,\dots, u_\kappa$ spanning $E\in \Gr_{\kappa}(V)$ will be called \emph{adapted} if, denoting $s=\dim E\cap V_P^+$, $t=\dim E\cap V_P^-$, the vectors $u_{\kappa-s-t+1},\dots, u_{\kappa-t}$ form a basis of $E\cap V_P^+$,	while $u_{\kappa-t+1},\dots, u_{\kappa}$ form a basis of $E\cap V_P^-$. 
A frame $u_i(E')$, $i=1,\dots, \kappa$ given near $E$ and adapted at $E$ is \emph{well-adapted} in a neighborhood $W$ if, whenever $E'\in W$ is such that $E'\cap E$ is spanned by the subset $(u_{i}(E))_{i\neq j}$ for some $1\leq j\leq \kappa$, then $u_{i}(E')=u_{i}(E)$ for all $i\neq j$. 

It is easy to see that a well-adapted frame can always be chosen to extend a given adapted orthonormal basis $u_i(E)$ of $E$ to a small neighborhood - define the frame $u_i(E')$ by orthogonally projecting $u_i(E)$ to $E'$, and then applying the Gram-Schmidt process.

For $X\in\Sym_\kappa(\R)$ and $\mu\in\R$, denote 
\begin{displaymath}
	E_\mu(X)=\{v\in\R^\kappa: Xv=\mu v\},\qquad \mathrm{mult}(\mu, X)=\dim E_\mu(X).
\end{displaymath}
Define for $a\geq 0$ 
\begin{align*}
	B^a_\mu=\{X: \mathrm{mult}(\mu, X)=a\}\subset\Sym_\kappa(\R).
\end{align*}
We will make use of the Hilbert-Schmidt Euclidean structure $\langle X,Y\rangle=\tr(XY)$ to identify $T_X\Sym_\kappa(\R)=T_X^*\Sym_\kappa(\R)=\Sym_\kappa(\R)$.
\begin{Lemma}\label{lem:symmetric_submanifolds}
	 $B^a_\mu\subset \Sym_\kappa(\R)$ is a locally closed submanifold. It holds that $N^*_XB_\mu^a=\{\Xi\in\Sym_\kappa(\R): X\Xi =\mu \Xi\}=\{\Xi\in\Sym_\kappa(\R): \Xi X=\mu \Xi\}=\Span\{vv^T: v\in E_\mu(X)\}$, and $\mathrm{codim} B_\mu^a=\binom{a+1}{2}$.
	\end{Lemma}

\proof
$B^a_0$ locally coincides with an orbit of the action of $\GL(\kappa)$ on $\Sym_\kappa(\R)$ by $(g, X)\mapsto g^TXg$, and $B^a_\mu=\mu I+B^a_0$.
Now $B_\mu^a$ fibers over $\Gr_a(\R^\kappa)$ with fiber $\Sym_{\kappa-a}(\R)$. Consequently, 
$$ \dim B_\mu^a=a(\kappa-a)+{\kappa-a+1\choose 2}$$ and one computes that $\mathrm{codim} B_\mu^a=\binom{a+1}{2}$. 

Let us describe the set $N^*_XB_\mu^a$. 
As $T_XB_0^a=\{A^TX+XA: A\in\mathfrak{gl}_\kappa(\R)\}$, we have $\Xi\in N^*_XB_0^a\iff \tr(\Xi A^T X+\Xi X A)=0$ for all $A$, or equivalently $N^*_XB_0^a=\{\Xi\in\Sym_\kappa(\R): \Xi X=0\}$. It follows that $$N^*_XB_\mu^a= N^*_{X-\mu I}B_0^a=\{\Xi: \Xi X=\mu \Xi\},$$ and the second form follows by transposition. Finally, $\Xi=uv^T+vu^T$ is easily checked to satisfy $\Xi X=\mu \Xi$ when $u,v\in E_\mu(X)$. By a simple dimension count we conclude that $N^*_XB_\mu^a=\Span\{uv^T+vu^T: u,v\in E_\mu(X)\}$, which  coincides with $\Span\{vv^T: v\in E_\mu(X)\}$ as $uv^T+vu^T=(u+v)(u+v)^T-uu^T-vv^T$.
\endproof

\begin{Lemma}\label{lem:nearby_orbits}
	For any $Y\in\Sym_\kappa(\R)$ with $Y\in B_0^r$ and for every $\epsilon>0$, there is a neighborhood $W_Y$ of $Y$ such that for all $X\in W_Y$, if $X\in B^{r'}_0$ then $T_{X}B^{r'}_0$ contains a subspace that is $\epsilon$-close to $T_Y B^r_0$.
\end{Lemma}

\proof Recall that by Lemma \ref{lem:symmetric_submanifolds}, $T_XB^{r'}_0=\{\Xi\in\Sym_\kappa(\R): X\Xi= 0\}^\perp$. The statement now follows from the following general fact.

\textit{Claim.} Let $M_0\in\mathrm{Mat}_{n\times n}(\R)$ be a matrix. Then for any $\epsilon>0$ there is a neighborhood $W_\epsilon$ of $M_0$ such that for any $M\in W_\epsilon$, $\Ker(M)^\perp$ contains a subspace that is $\epsilon$-close to $\Ker(M_0)^\perp$.

\textit{Proof.} Assume that $\mathrm{rank}M_0=r$, and the first $r$ rows $u_1(M_0)^T,\dots, u_r(M_0)^T$ are linearly independent. Therefore, $\Ker(M_0)^\perp=\Span(u_1(M_0),\dots, u_r(M_0))$. By choosing $W_\epsilon$ small enough, we may ensure that $\Span(u_1(M),\dots, u_r(M))$ is $r$-dimensional, and $\epsilon$-close to $\Ker(M_0)^\perp$. Since $\Span(u_1(M),\dots, u_r(M))\subset \Ker(M)^\perp$, this concludes the proof.
\endproof

\begin{Lemma}\label{lem:eigenspace_tangent}
Let $M_\tau$ be a smooth curve in $\Sym_\kappa(\R)$ such that the spectrum of $M_\tau$ lies in $[-1,1]$ for all $\tau$.
If $\epsilon \in \{-1,1\}$ and $\mathrm{mult}(\epsilon, M_0)=s$, then $\left.\frac{d}{d\tau}\right|_0M_\tau\in T_{M_0}B^{s}_\epsilon$.
\end{Lemma}

\proof 
Write $\dot M_0=\left.\frac{d}{d\tau}\right|_0M_\tau$. By Lemma \ref{lem:symmetric_submanifolds}, we ought to show that for all unit vectors $v\in E_\epsilon(M_0)$, $\langle \dot M_0,  vv^T\rangle=0$. Now for any such $v$, $\langle M_0 v, v\rangle =\epsilon$. We know by assumption $|\langle M_\tau v, v\rangle|\leq 1$ for all $\tau$, and so $$\langle \dot M_0, vv^T\rangle=\langle \dot M_0 v, v\rangle =\left.\frac{d}{d\tau}\right|_0\langle M_\tau v, v\rangle=0.$$
\endproof

\begin{Proposition}\label{prop:m0_explicit}
	Let $X_0$ be locally defined by a well-adapted frame at $E\in\Gr_{n-k+1}(V)$. Then there is a neighborhood  $W_{E}$ of $E$ where the distribution $$\widetilde m^0_k(E'):= X_0^*f_{-\frac{n+1}{2}}(E')\cdot dE'$$ is well-defined.
	Furthermore, the corresponding distributions agree on non-empty intersections $W_{E_1}\cap W_{E_2}$ for all $E_1, E_2$, giving rise to a globally defined distribution $\widetilde m^0_k\in\mathcal M^{-\infty}(\Gr_{n+1-k}(V))$. Moreover,  $\widetilde m^0_k$ is $\OO(Q)$-invariant. 
\end{Proposition}
\proof

Fix $E$, and a well-adapted to $E$ frame $u_j(E')$ defined in a neighborhood $W'_E$. Set $Y=X_0(E)$, $s=\dim E\cap V_P^+=\mathrm{mult}(1, Y)$, $t=\dim E\cap V_P^-=\mathrm{mult}(-1, Y)$, $r=\mathrm{mult}(0, Y)$.

\textit{Claim.} It holds that $\Image(d_E X_0)+T_YB_{0}^{r}=T_Y\Sym_\kappa(\R)$. 
 
 Let us prove the claim. We may assume that $Y$ lies in the singular support of $f_{\lambda}$, that is $r\geq 1$. Thus $Y\in B^r_0\cap B^s_1\cap B^t_{-1}$.

	The intersection $B^{s,t}_{1,-1}:=B^s_1\cap B^t_{-1}$ is transversal. To this end simply observe that by Lemma \ref{lem:symmetric_submanifolds}, $N_Y^*B_1^s\cap N_Y^*B_{-1}^t=\{\Xi: \Xi Y=\Xi =-\Xi \}=\{0\}$, so $B^s_1\pitchfork B^t_{-1}$, and $B^{s,t}_{1,-1}$ is a submanifold. It holds that
	\begin{displaymath}
		\dim B^{s,t}_{1,-1}={\kappa+1\choose 2}-{s+1\choose 2}-{t+1\choose 2}.
	\end{displaymath}

	Similarly, the intersection $B^r_0\cap B^{s,t}_{1,-1}$ is transversal. Indeed, $N^*_YB^{s,t}_{1,-1}=\{ \Xi_1+\Xi_2: \Xi_1Y=\Xi_1, \Xi_2Y=-\Xi_2\}$ and 
	\begin{displaymath}
		\mathrm{codim}T_YB^{s,t}_{1,-1}=\mathrm{codim}T_Y B_1^s+\mathrm{codim}T_Y B_{-1}^t={s+1\choose 2}+{t+1\choose 2}.
	\end{displaymath}   
If $\Xi= \Xi_1+\Xi_2\in N^*_YB^{s,t}_{1,-1}\cap N_Y^*B_0^r$, then $\Xi_1-\Xi_2=\Xi_1Y+ \Xi_2Y=\Xi Y=0$, so that $ \Xi_1= \Xi_2$, which can only happen if $\Xi_1=\Xi_2=0$ since $\Xi_1Y=\Xi_1$, $\Xi_2Y=-\Xi_2$, therefore $\Xi=0$. Thus $B_0^r \pitchfork B_{1,-1}^{s,t}$ as claimed.
	
	Set $E_Y:=E_1(Y)\oplus E_{-1}(Y)$, and define 
	\begin{displaymath}
		W_Y=\{X\in B^{s,t}_{1,-1}:E_\mu(X)=E_\mu(Y)\text{, } \forall\mu\neq \pm 1\}.
	\end{displaymath}

As $X\in W_Y$ is uniquely determined by its eigenspace $E_X(1)$, $W_Y$ is evidently a manifold that can be identified with $\Gr_{s}(E_Y)=\Gr_s(\R^{s+t})$, in particular $\dim W_Y=st$.
By definition, $W_Y\subset B^r_0$. 

Now since $-P\leq Q\leq P$, the spectrum of $ X_0(E)$ lies in $[-1,1]$. By Lemma \ref{lem:eigenspace_tangent} we have $\mathrm{Image}(d_E X_0)\subset T_YB^{s,t}_{1,-1}$.

For $1\leq j\leq \kappa$, choose a smooth curve $\gamma_j(\tau)$ through $E$ given  by 
\begin{displaymath}
	\gamma_j(\tau)=\Span(u_1(E),u_2(E),\dots,\cos\tau u_j(E)+\sin\tau \xi,\ldots,u_\kappa(E)),
\end{displaymath}
where $\xi\in E^P$ is arbitrary. Observe that $T_E\Gr_{\kappa}(V)=\Span\{\gamma_j'(0):1\leq j\leq \kappa\}$. Since the frame is well-adapted to $E$, $u_i(\gamma_j(\tau))=u_i(E)$ for $i\neq j$, and so $u_j(\gamma_j(\tau))=\cos\tau u_j(E)+\sin\tau \xi$. One computes
\begin{displaymath}
	d_E X_0 (\gamma_j'(0))=
	\begin{pmatrix}
		0&\cdots &0 &Q(\xi, u_1)&0&\cdots&0 \\
		\vdots &\ddots&\vdots &\vdots&\vdots&\ddots &\vdots \\
		0&\cdots & 0& Q(\xi, u_{j-1}) &0&\cdots&0\\
		Q(\xi, u_1)&\cdots &Q(\xi, u_{j-1})& Q(\xi, 2u_j) & Q(\xi, u_{j+1}) & \cdots & Q(\xi,u_\kappa) \\
		0&\cdots & 0& Q(\xi, u_{j+1})&0 & \cdots&0\\
		\vdots &\ddots&\vdots &\vdots&\vdots&\ddots&\vdots \\
		0&\cdots &0 &Q(\xi, u_\kappa)&0&\cdots&0 
	\end{pmatrix}
\end{displaymath}

Note that $E\cap(E^P)^Q=(E\cap V_P^+) \oplus (E\cap V_P^-)$. Hence $Q(\xi, u_1), Q(\xi,u_2),\ldots, Q(\xi,u_{\kappa-s-t})\in (E^P)^*$ are linearly independent functionals, while the bottom right $(s+t)\times (s+t)$ minor of $d_E  X_0 (\gamma_j'(0))$ vanishes. 

Therefore for $1\leq j\leq \kappa-s-t$, we may choose $\xi_j\in E^P$ such that $Q(\xi_j, u_i)=0$ for $1\leq i\leq j-1$, while $Q(\xi_j, u_i)$ is arbitrary for $j\leq i\leq \kappa-s-t$. For $\kappa-s-t+1\leq j\leq \kappa$, we may choose $\xi_j\in E^P$ to get arbitrary $\kappa-s-t$ first entries in the $j$-th row and column. Thus the entries of a matrix in $\Image(d_E X_0)$ can be made arbitrary outside of the bottom right $(s+t)\times (s+t)$ minor. Consequently, $\mathrm{codim}(\mathrm{Image}(d_E X_0))= {s+t+1\choose 2}$.

We claim that $\mathrm{Image}(d_E X_0)\cap T_YW_Y=\{0\}$. This is because $T_YW_Y$ consists of all matrices that vanish outside of the bottom right $(s+t)\times (s+t)$-minor $M$, which has zeros in its top left $s\times s$ minor and bottom right $t\times t$ minor.

One easily verifies that ${s+1\choose 2}+{t+1\choose 2}+st={s+t+1\choose 2}$, so that $$\dim \mathrm{Image}(d_E X_0)+\dim T_{Y}W_Y\geq\dim T_Y B^{s,t}_{1,-1}.$$ Hence  
\begin{displaymath}
	\mathrm{Image}(d_E X_0)\oplus T_YW_Y=T_Y B^{s,t}_{1,-1}. 
\end{displaymath}
Since $T_YW_Y\subset T_YB^r_0$ and $T_YB^r_0 + T_Y B^{s,t}_{1,-1}=T_Y\Sym_\kappa(\R)$, we conclude that $ \mathrm{Image}(d_E X_0)+T_YB^r_0=T_Y\Sym_\kappa(\R)$ as claimed.

Fix $\epsilon>0$. For $E'$ in a sufficiently small neighborhood $W_E$ of $E$, $\Image d_{E'} X_0$ must contain a subspace that is $\epsilon$-close to $\mathrm{Image}(d_E X_0)$.  We may moreover  by Lemma \ref{lem:nearby_orbits} assume $W_E$ is such that for all $E'\in W_E$, $Y'= X_0(E')$ has nullity $r'\leq r$, and $T_{Y'}B_0^{r'}$ contains a subspace that is $\epsilon$-close to $T_YB_0^r$. Thus for sufficiently small $\epsilon$ we find a neighborhood $W_E$ such that for all $E'\in W_E$ with $Y'= X_0(E')$ of nullity $r'$, 
\begin{displaymath}
	\Image(d_{E'} X_0)+T_{ X_0(E')}B^{r'}_0=\Sym_\kappa(\R). 
\end{displaymath}

Define 
\begin{displaymath}
	L_{E',Y'}=\mathrm{Ker}\left(d X_0^*:T^*_{Y'}\Sym_\kappa(\R)\to T_{E'}^*\Gr_{\kappa}(V)\right).
\end{displaymath}
Thus $N_Y^*B_0^{r'}\cap  L_{E',Y'}=\{0\}$ for $E'\in W_E$.

By Proposition \ref{prop:hormander_convergence}, $\WF_Y(f_\lambda)\subset N_{Y'}^*B_0^{r'}$.  We conclude that $\WF(f_{\lambda})\cap L_{E',Y'}=\emptyset$. 

By \cite[Proposition 1.3.3]{duistermaat_book96} $ X_0^*$ defines a sequentially continuous linear operator on $C^{-\infty}_{N^*\Gamma^\kappa}(\Sym_\kappa(\R))$, where $f_\lambda$ lies. Moreover, $  X_0^*f_\lambda$ must itself be an analytic family: for a smooth compactly supported test measure $\psi$ on $W_{E}$ we have $\langle  X_0^*f_\lambda, \psi\rangle =\langle f_\lambda, ( X_0)_*\psi\rangle$, and by \cite[Chapter VI, Proposition 3.9]{guillemin_sternberg} we have $\WF(( X_0)_*\psi)\subset \cup_{E'\in W_E} L_{E',  X_0(E')}$. Now analyticity of a vector-valued function coincides with weak analyticity in quasi-complete locally convex vector spaces, and $C^{-\infty}_{N^*\Gamma^\kappa}(\Sym_\kappa(\R))$ is quasi-complete \cite[Proposition 29]{brouder_dabrowski}. Furthermore $(X_0)_*\psi$ defines a continuous linear functional on $C^{-\infty}_{N^*\Gamma^\kappa}(\Sym_\kappa(\R))$ by \cite[Lemma 3]{brouder_dabrowski}, confirming that the family of generalized functions $ X_0^*f_\lambda\in C^{-\infty}(W_E)$ is analytic.

Now observe that the continuous functions $ X_0^*f_{\lambda}$ defined separately for $W_{E_1}$ and $W_{E_2}$ coincide on non-empty intersections $W_{E_1}\cap W_{E_2}$ for $\mathrm{Re}\lambda>0$ by eq. \eqref{eq:frame_change}. If follows by uniqueness of the  analytic extension that this holds for all $\lambda\in\C$. In particular, $\widetilde m^0_k$ is a globally well-defined distribution.

For invariance, we note that for $g\in\OO(Q)$, $g^*X_0^*f_\lambda=\psi_g(E)^\lambda X_0^*f_\lambda$ for  $\lambda>0$, and consequently by uniqueness of analytic extension for all $\lambda\in\C$. Here $\psi_g(E)=\Jac(g:E\to gE)^{-2}$, which for $g\in\OO(Q)$ satisfies $\psi_g(E)=\frac{\det X_0(gE)}{\det X_0(E)}$, see \cite[Proposition 4.7]{bernig_faifman_opq} for details. Using the identification $$\Dens(T_E\Gr_{n+1-k}(V))=\Dens(E^*\otimes V/E)=\Dens^*(E)^{\otimes (n+1)}\otimes\Dens(V)^{\otimes n+1-k},$$
it follows that $X_0^*f_\lambda dE$ is an $\OO(Q)$-invariant distribution on $\Gr_{n+1-k}(V)$ when $\lambda=-\frac{n+1}{2}$.
\endproof

Henceforth whenever $ X_0$ appears, a local well-adapted frame should be chosen arbitrarily unless an explicit choice is provided.

\subsection{Some properties of the invariant distributions}

We will use the following rescaling of the invariant distribution constructed above, which brings the total integral to $1$ as will be later seen.

	\begin{Definition}\label{def_tilde_m} Set
	\begin{displaymath}
		m_k:= e^{\frac{\bi\pi}{2} (n+1-k)q} \widetilde m_k^0 \in\mathcal M^{-\infty}(\Gr_{n+1-k}(V))^{\OO(Q)}.
	\end{displaymath}
\end{Definition}
	
	\begin{Lemma}\label{lem:crofton_sign}
		Let $j\colon \R^{p,q}\to\R^{q,p}$ be given by $j(x,y)=(y,x)$ where $x\in \R^p,y\in\R^q$. Let us also denote by $j$ the induced map $\Gr_{p+q-k}(\R^{p,q})\to \Gr_{p+q-k}(\R^{q,p})$. Then
		\begin{displaymath}
		j^*m_k=\overline{m_k}.
		\end{displaymath}
	\end{Lemma}
	
	\begin{proof}
		We have $m_k=\mathbf i^{(n+1-k)p}{X_0^*f_{-\frac{n+1}2}}d\sigma_{P_0}$  on  $\R^{q,p}$.  Since $ X_{ 0} \circ j=- X_{ 0}$ (where $X_0$ is defined using $j$-corresponding frames), \eqref{eq_reflection_f_lambda} implies that $j^*X_{ 0}^*f_{-\frac{n+1}2}=\mathbf i^{-(n+1-k)(n+1)}\overline{ X_{ 0}^*f_{-\frac{n+1}2}}$. It follows that
		\begin{align*}
		j^*m_k&=\mathbf i^{(n+1-k)p-(n+1-k)(n+1)} \overline{ X_{ 0}^*f_{-\frac{n+1}2}}d\sigma_{P_0}\\
		&=\mathbf i^{-(n+1-k) q}\overline{ X_{ 0}^*f_{-\frac{n+1}2}}d\sigma_{P_0},
		\end{align*}
		which is the conjugate of $m_k$ in $\R^{p,q}$.
		\end{proof}
		
	\begin{Proposition}\label{prop:boundary_value_grassmannian} Define \begin{displaymath}
		N^*\Lambda:=\cup_{\nu\geq1}N^*\Lambda^\nu_{n+1-k}(V).
		\end{displaymath} 
		\begin{enumerate}
			\item The wave front set of $m_k$ is contained in $N^*\Lambda$.
			
			\item $m^{\mathbf i\epsilon}_k\to m_k$ in $\mathcal M_{N^*\Lambda}^{-\infty}(\Gr_{n+1-k}(V))$ as $\epsilon\to 0^+$ in the normal H\"ormander topology.
		\end{enumerate}	 
	\end{Proposition}
	\proof
	
	Write $\lambda=-\frac{n+1}{2}$. For $\zeta\in U_{\mathbb C}$ we have $$(2\zeta+1)^{-\frac{n+1-k}{2}p}(2\zeta-1)^{-\frac{n+1-k}{2}q}m^\zeta_k(E)=  \det (X^{P_0}_\zeta)^{\lambda}dE.$$
	We compute, using a well-adapted frame, 
	\begin{displaymath}
	\det(X^{P_0}_{\mathbf i\epsilon}(E))^{\lambda} = \det(X_{0}(E)+2\mathbf i\epsilon I_{n+1-k} )^{\lambda}=X_0^*\det(X+2\mathbf i\epsilon I_{n+1-k})^{\lambda}.
	\end{displaymath}
	
	By Proposition \ref{prop:hormander_convergence}, we have $\det(X+2\mathbf i\epsilon  I_{n+1-k})^{\lambda}\to f_{\lambda}(X)$ in the normal H\"ormander topology on $C^{-\infty}_{\Gamma^{n+1-k}}(\Sym_{n+1-k}(\R))$. By the proof of Proposition \ref{prop:m0_explicit}, we may use the continuity of the pull-back $X_0^*$ in the normal H\"ormander topology \cite{brouder_dang_helein}. Noting that $X_0^{-1}(\Gamma^{n+1-k}\nu)\subset \Lambda_{n+1-k}^\nu(V)$ so that $ X_0^*N^*\Gamma^{n+1-k}\subset N^*\Lambda$, we find that $$  X_0^*\det(X+2\mathbf i\epsilon I_{n+1-k})^{\lambda} dE\to \widetilde m_k^0$$
	in the normal H\"ormander topology as stated. 
	\endproof

	\begin{Corollary}\label{cor:can_compute_on_LC}
		Let $M \subset V^{n+1}$ be a pseudosphere or a pseudohyperbolic space, and $A\subset M$ either a smooth domain with LC-regular boundary, or a smooth LC-regular hypersurface without boundary. Assume all $Q$-degenerate tangents to $A$ of codimension $k$ are regular. Then 
		\begin{equation}\label{eq:crofton_general}\Cr(m_k)(A)=\int_{\Gr_{n+1-k}(V)}\chi(A\cap E)dm_k(E). \end{equation}
	\end{Corollary}
	\proof
	First note that $\Cr(m_k)\in\mathcal V^{-\infty}(M)$ is isometry invariant and by \cite[Theorem C]{bernig_faifman_solanes_part2} is given by a linear combination of the intrinsic volumes. Thus $A$ is WF-transversal to $\Cr(m_k)$.
	The assertion now follows from Corollary \ref{cor:CrWF_of_LC}, and Propositions \ref{prop:boundary_value_grassmannian} part i) and \ref{prop:apply_crofton_general}.
	\endproof

	\begin{Corollary}\label{cor:can_compute_mu1}
		Let $M^{n}\subset V^{n+1}$ be a pseudosphere or a pseudohyperbolic space, and $A\subset M$ either a smooth domain or a hypersurface without boundary. Denote $H=H(A)$.
		\begin{enumerate}
			\item 
			Assume that for each $x\in H$, $H$ is either pseudo-Riemannian at $x$ or tangentially regular at $x$.  Then \eqref{eq:crofton_general} holds for $k=1$. 
			\item If $\mathbb P(H)\subset\mathbb P(V)$ is strictly convex, then \eqref{eq:crofton_general} holds for all $k$.
		\end{enumerate}
	\end{Corollary}
	\proof
	In both cases, it follows from \cite[Lemma 4.7]{bernig_faifman_solanes} that $H$ is LC-regular, and we can apply Corollary \ref{cor:can_compute_on_LC}.
	\endproof
		
	\textbf{Example.} The complex-valued distribution $m_{n}\in\mathcal M^{-\infty}(\mathbb P(V))$ is invariant under the group of projective transformations preserving the quadric $[Q]=\{Q=0\}$. Its singular support is $[Q]$, and $\WF(m_{n})\subset N^*[Q]$. In particular, $m_{n}(A)$ is well defined for any domain $A\subset \mathbb P(V)$ that is smooth near $[Q]$ and transversal to it. When $Q$ is definite, the quadric $[Q]$ has no real points and $m_{n}$ is the Haar measure on the round projective space.
	
	\subsection{The flat case}
	Next we construct a translation- and $\OO(p,q)$-invariant distribution on the affine Grassmannian $\overline\Gr_{p+q-k}(\R^{p,q})$.
	
	\begin{Proposition}\label{prop:derivative_crofton}
		Let $P$ be a $Q$-compatible Euclidean structure in $W=\R^{p+1,q}=W_+\oplus W_-$. Let $x\in W_+\cap S^{p,q}$, $T=T_xS^{p,q}$, and define
		\begin{displaymath}
		s\colon\overline{\Gr}_{p+q-k}(T) \longrightarrow\Gr_{p+q+1-k}(W) ,\quad s(v+F)=F\oplus \R(x+v),\ {F\in\Gr_{p+q-k}(T)},
		\end{displaymath}
		which is a diffeomorphism onto its open image.
		Given $t>0$, consider the homothety $v\mapsto tv$ on $T$ and the induced map $h_t$ on $\overline{\Gr}_{p+q-k}(T)$.

		\begin{enumerate}
			\item[i)] Let $dE$ be an $\OO(P)$-invariant measure on $\Gr_{p+q+1-k}(V)$, thus given by a smooth density. Then 
			\begin{displaymath}
			d\overline F=\frac1{k!}\left.\frac{d^k}{dt^k}\right|_{t=0}h_{t}^* s^*dE  
			\end{displaymath}
			is an $\overline{\OO(P|_T)}$-invariant measure on $\overline\Gr_{p+q-k}(T)$.
			\item[ii)] Let $X_0\colon\Gr_{p+q+1-k}(W)\to \Sym_{p+q+1-k}(\R)$ be as in Proposition \ref{prop:m0_explicit}, and let $X_0'$ be the corresponding map on $\Gr_{p+q-k}(T)$. Then
			\begin{displaymath}
			\left.\frac1{k!}\frac{d^k}{dt^k}\right|_{t=0} (h_{1/t})_* (s^{-1})_* ( X_0^*f_{-\frac{p+q+1}2}(E)dE)= (X_0')^* f_{-\frac{p+q+1}2}(F) d\overline F ,
			\end{displaymath}
			and this generalized measure is $\overline{\OO(Q|_T)}$-invariant. Here $(s^{-1})_*$ denotes the push-forward by $s^{-1}$  of the restriction to the open set $\mathrm{Image}(s)$.
		\end{enumerate}
	\end{Proposition}
	
	\begin{proof}
		\begin{enumerate} 
		 \item For $g \in \OO(P|_T)=\Stab_{\OO(P)}(x)\subset \OO(P)$, since $g\circ s\circ  h_t=s\circ h_t\circ g$ and $g^* dE=dE$, we have $g^*h_{t}^* s^*dE=h_t^* s^*g^*dE=h_t^* s^*dE$, which yields $\OO(P|_T)$-invariance. As for translation invariance, let $\rho_U\colon U\times \R^k\to \overline{\Gr}_{p+q-k}(T)$ be a local trivialization of the bundle $\pi\colon\overline{\Gr}_{p+q-k}(T)\to{\Gr_{p+q-k}}(T)$, and put $\eta=\rho_U^*s^*dE$. Then $h_t\circ \rho_U(F,w)=\rho_U(F,tw)$ and thus
		\begin{displaymath}
		\rho_U^*h_t^*s^*(dE)_{(F,w)}=t^k\eta_{(F,tw)}=t^k\eta_{(F,0)} +O(t^{k+1}).
		\end{displaymath}
		Since the induced action of a translation of $T$ on $U\times \R^k$ has the form $(F,w)\mapsto (F, w+\varphi(F))$, the translation invariance of $\overline{dF}$ follows.
		\item 
			Let $f_1,\ldots, f_{p+q-k}$ be a $P$-orthonormal basis of $F\in \Gr_{p+q-k}(T)$, and let $w\in T$ be $P$-orthogonal to $F$. A $P$-orthonormal basis of $s(tw+F)$ is $(1+t^2P(w))^{-\frac12}(x+tw), f_1,\ldots, f_{p+q-k}$. Hence, the Gram matrices $X_0',X_0$ of $Q$ restricted to $F$ resp. $s(tw+F)$ satisfy 
			\begin{displaymath}
                    \det X_0=(1+t^2P(w))^{-1}Q(x+tw)\det X_0'.
			\end{displaymath}
		
		Therefore, for $\lambda>0$, $F\in\Gr_{p+q-k}(T)$ and $w\in F^P\cap T$ we have
		\begin{displaymath}
		f_\lambda(X_0(s(tw+F)))=(1+t^2P(w))^{-\lambda}Q({x+tw})^\lambda f_\lambda(X_0'(F))
		\end{displaymath} 
		whenever $Q(x+tw)>0$, where $f_\lambda$ is defined by Lemma \ref{lem_f_lambda} on $\Sym_{p+q+1-k}(\R)$ or $\Sym_{p+q-k}(\R)$ depending on the argument.

		By analytic continuation we get 
		\begin{displaymath}
            s^* X_0^*f_{\lambda}(tw+F)=(1+t^2P(w))^{-\lambda}Q({x+tw})^\lambda (X_0')^*f_\lambda(tw+F)
		\end{displaymath}
        for all $\lambda$. Hence,  
		\begin{displaymath}
		\lim_{t\to 0} h_t^*  s^* X_0^* f_\lambda= ( X_0')^* f_\lambda.
		\end{displaymath}

		In the proof of $i)$ we have seen $(h_{1/t})_*(s^{-1})_*dE=O(t^k)$. Hence, by continuity
		\begin{align*}
		&\left.\frac{d^k}{dt^k}\right|_{t=0} (h_{1/t})_* (s^{-1})_* ( X_0^*f_{-\frac{p+q+1}2}(E)dE)=\\&=\lim_{t\to 0}h_t^*s^* X_0^*f_{-\frac{p+q+1}2}\left.\frac{d^k}{dt^k}\right|_{t=0} (h_{1/t})_* (s^{-1})_* dE\\
		&=k!( X_0')^* f_{-\frac{p+q+1}2} d\overline F.
		\end{align*}
		Translation invariance is clear. Further, if $g\in \OO(Q|_T)\subset \OO(Q)$, then $g\circ s\circ  h_t=s\circ h_t\circ g$. Since $ X_0^*f_{-\frac{p+q+1}2}(E)dE$ is $\OO(Q)$-invariant, this yields $\OO(Q|_T)$-invariance.
		\end{enumerate}
	\end{proof}
	
	Translation-invariance and ${\OO(P|_T)}$-invariance characterize $d\overline F$ uniquely up to normalization. The normalization can be deduced from Theorem \ref{thm_crofton_formula} in the case $q=0$. As for the translation-invariant and $\OO(Q|_T)$-invariant generalized measure obtained on $T\cong \R^{p,q}$, we take the normalization of Definition \ref{def_tilde_m} as follows.
	
	\begin{Definition}
	 On $\overline \Gr_{p+q-k}(\R^{p,q})$ we fix the following translation-invariant and ${\OO(p,q)}$-invariant generalized measure
		\begin{displaymath}
		\check{m}_k:=e^{\frac{\bi\pi}2 (p+q+1-k)q}  (X_0')^*f_{-\frac{p+q+1}{2}}(F) d\overline F.
		\end{displaymath}
	\end{Definition}
	
		The Crofton map in the flat case is
	\begin{displaymath} 
	\Cr: \mathcal M^{-\infty}(\overline{\Gr}_{p+q-k}(V))\to \mathcal V^{-\infty}(V),
		\end{displaymath}
	given by $\langle \Cr(\mu), \psi\rangle= \int_{\overline{\Gr}_{p+q-k}(V)}\psi(\overline E)d\mu(\overline E)$ for all $\psi\in\mathcal V_c^\infty(V)$.
	
	The results of the present section and sections \ref{sec:crofto}, \ref{sec:LC_crofton_WF} can be easily adapted to the flat pseudo-Euclidean setting. Let us state explicitly Corollary \ref{cor:can_compute_mu1} in the flat case.
	\begin{Corollary}
		Let $A\subset \R^{p,q}$ be either a smooth domain or a hypersurface without boundary. Denote by $H=H(A)$ the corresponding closed hypersurface.
		\begin{enumerate}
			\item 
			Assume that for each $x\in H$, $H$ is either pseudo-Riemannian near $x$ or has non-zero Gauss curvature at $x$.  Then
			\begin{equation}\label{eq:crofton_affine}\Cr(\check m_k)(A)=\int_{\overline{\Gr}_{p+q-k}(V)}\chi(A\cap E)d\check m_k(E). \end{equation}
			holds for $k=1$. 
			\item If $H$ is strictly convex, then \eqref{eq:crofton_affine} holds for all $k$.
		\end{enumerate}
	\end{Corollary}


	\section{Crofton formulas for generalized pseudospheres}
	\label{sec:template}
	
	For the de Sitter space embedded in Lorentz space, one can compute the Crofton formulas through a direct computation of the restriction of the measures to subspaces, combined with the Hadwiger theorem and the template method. However for general signatures, an explicit computation appears to be hard. Instead, we carry out an analytic extension argument, which recovers the Crofton formulas for all signatures in a unified fashion.
	
	For $\zeta>\frac12$, we denote by $S_\zeta={Q_\zeta^{-1}(1)}\subset \R^{p+q}=\R^{n+1}$ the unit sphere in the Euclidean space $(\R^{n+1},Q_\zeta)$. For $\zeta=0$ we have $S^{p-1,q}=Q_0^{-1}(1)$ with the induced pseudo-Riemannian metric $Q_{0}$.  We will also denote by $S^{n}$ the unit sphere in $\R^{n+1}$ with respect to some fixed Euclidean structure (which is independent of $\zeta$).
	
	In the following we make use of the operation of restriction of Crofton distributions, as described in Section \ref{sec:crofton_functorial}.
	
	\begin{Proposition} \label{prop:restrictions}
		For the standard inclusion $e\colon \R^{p,q}\hookrightarrow \R^{p+l,q+r}$, we have $e^*m_k=m_k$.
	\end{Proposition}
	
	\begin{proof}
	Note first that the restriction $e^*m_k$ is well-defined by \cite[Remark 2.13]{faifman_crofton}. For $\zeta>\frac12$, $Q_\zeta$ is positive definite, and so $e^*m_k^\zeta=m_k^\zeta$ by the uniqueness of probability measure on the Grassmannian invariant under the positive definite orthogonal group, as $$e^*:\mathcal M^\infty(\Gr_{p+q+l+r-k}(\R^{p+l,q+r}))\to \mathcal M^\infty(\Gr_{p+q-k}(\R^{p,q}))$$ is essentially the pushforward operation under intersection with $\R^{p,q}$. By analytic extension in $\zeta$, we get $e^*m_k^{\bi \epsilon} = m_k^{\bi \epsilon}$. The statement then follows from Proposition \ref{prop:boundary_value_grassmannian} ii).
	\end{proof}

	\begin{Proposition}\label{prop:weak_continuity} Given $A\in\mathcal P( S^{p-1,q})$, let $\overline A\in \mathcal P(S^{n})$ be its radial projection. Assume $A$ is either an LC-regular hypersurface, or a smooth domain with LC-regular boundary. Assume further that either all $Q_0$-degenerate tangents of codimension $k$ are regular, or $\chi(A\cap E)$ is constant for a.e. plane $E$ of codimension $k$. Then
		\begin{displaymath}
		\lim_{\epsilon\to 0^+} \Cr_{S^{n}}(m_k^{\mathbf i\epsilon})(\overline A)=\Cr_{S^{p-1,q}}(m_k)(A).
		\end{displaymath}
	\end{Proposition}
	
	\begin{proof}
		We have by Proposition \ref{prop:apply_crofton_general} 
		\begin{displaymath}
            \Cr_{S^{n}}(m_k^{\mathbf i\epsilon})(\overline A)=\langle m_k^{\mathbf i\epsilon}, \chi(\overline A\cap \bullet)\rangle.
		\end{displaymath}
By Corollary \ref{cor:can_compute_on_LC}, it holds that
		\begin{align*}
		\Cr_{S^{p-1,q}}(m_k)(A)=\langle m_k, \chi(A\cap \bullet)\rangle.
		\end{align*}
		By part ii) of Proposition \ref{prop:boundary_value_grassmannian}, $m_k^{\mathbf i \epsilon}$ tends to $m_k$ as $\epsilon \to 0^+$ in the normal H\"ormander topology on $\mathcal M^{-\infty}_{N^*\Lambda}(\Gr_{n+1-k}(V))$. Combining Corollary \ref{cor:CrWF_of_LC} and Proposition \ref{prop:boundary_value_grassmannian} part i), we see that evaluating at $\chi(\overline A\cap \bullet)=\chi( A\cap \bullet)$ is continuous in this topology, and the statement follows. 
	\end{proof}
	
	We consider for a moment the case $q=1$. We will use two types of templates in the de Sitter sphere $S^{p-1,1}$. The first one is the Riemannian $(p-1)$-unit sphere \begin{displaymath}
	R^{p-1,0}=S^{p-1,1}\cap \{x_{p+1}=0\}.
	\end{displaymath}
	
	Fix $\theta\in (0,\pi/ 4)$. Our second template is
	\begin{displaymath}
	R^{p-1,1}=R^{p-1,1}(\theta)=\{x\in S^{p-1,1}\colon x_{p+1}^2 \leq \tan^2\theta (x_1^2+\cdots+x_{p}^2)\}. 
	\end{displaymath}
	The points of $\partial R^{p-1,1}$ lie at (time-like) distance of $\rho=\mathrm{arctanh}(\tan\theta)$ from $R^{p-1,0}$.

	For each $\zeta>\frac 12$ and $s=0,1$,  we denote by $T_\zeta^{p-1,s}$ the radial projection of $R^{p-1,s}$ on $S_\zeta$. Thus $T_\zeta^{p-1,0}$ is a totally geodesic $(p-1)$-sphere in $S_\zeta$, and the points of $\partial T_\zeta^{p-1,1}$ lie at distance $\varepsilon=\arctan(\sqrt\xi \tan\theta)
	$ from $T_\zeta^{p-1,0}$ where $\xi=\frac{2\zeta-1}{2\zeta+1}$. We then have 
	\begin{displaymath}
	 \frac{d\rho}{d\theta}=\frac{1+\tan^2 \theta}{1-\tan^2 \theta}, \quad  \frac{d\epsilon}{d\theta}=\sqrt{\xi} \frac{1+\tan^2 \theta}{1+\xi\tan^2 \theta}.
	\end{displaymath}

	We will denote by $\mu_k^\zeta\in\mathcal V^\infty(S_\zeta)$ the Riemannian intrinsic volumes in $S_\zeta$, and by $\mu_k\in\mathcal V^{-\infty}(S^{p-1,q})\otimes \C$ the (complex-valued)  intrinsic volumes on $S^{p-1,1}$. Note that $\mu_k^\zeta(T_\zeta^{p-1,0})=\mu_k(R^{p-1,0})$ for all $\zeta>\frac12$.

	\begin{Proposition}\label{prop:continuation_mu}
	 For $s=0,1$, the function $\zeta \mapsto \mu_k^\zeta(T_\zeta^{p-1,s})$ extends to a holomorphic function $f_{k,s}(\zeta)$ on $U_\C$ such that $\lim_{\zeta\to 0}f_{k,s}(\zeta)=\mu_k(R^{p-1,s})$.
	\end{Proposition}
	
	\begin{proof} For $s=0$, the statement is trivial as $\mu_k^\zeta(T_\zeta^{p-1,0})$  does not depend on $\zeta$. Let  us consider $s=1$.
		The radial projections $\pi_\zeta:S^{n}\to S_\zeta$ and $\pi_{0}:S^{n}\to S^{p-1,1}$ have Jacobians 
		\begin{align*}
		\mathrm{Jac}\pi_\zeta &= \left(\frac{\cos\epsilon}{\cos\theta}\right)^{p-1}\frac{d\epsilon}{d\theta}=\xi^{\frac12}\left(\frac{1+\tan^2\theta}{1+\xi\tan^2\theta}\right)^{\frac{p-1}2+1}, \quad \xi=\xi(\zeta)=\frac{2\zeta-1}{2\zeta+1}\\
		\mathrm{Jac}\pi_{0} &= \left(\frac{\cosh\rho}{\cos\theta}\right)^{p-1}\frac{d\rho}{d\theta}=\left(\frac{1+\tan^2\theta}{1-\tan^2\theta}\right)^{\frac{p-1}2+1}.
		\end{align*}
		
		 Since $\xi(\zeta)=\frac{2\zeta-1}{2\zeta+1}$ is continuous on $\C\setminus\{-\frac12\}$, it maps $U_\C$ to a simply connected region in $\C\setminus \{0\}$. Moreover $\xi(\zeta)\in\R$ if and only if $\zeta\in\R$,  and $\xi(\zeta)>0$ for $\zeta>\frac12$, so that $\xi(\zeta),1+\xi(\zeta)\tan^2\theta\neq 0$ for $\zeta\in U_\C$. It follows that the right hand side of the first equation extends to a holomorphic function in $U_\C$ whose limit as $\zeta\to 0$ equals the right hand side of the second equation multiplied by $\mathbf i$. The statement follows for $k=p$ since $\mu_{p}^\zeta=\vol_{p}$, and $\mu_{p}=\mathbf i \vol_{p}$  on $S^{p-1,1}$.
		
		Consider now $k=p-1$. Since $\mu_{p-1}(R^{p-1,1})=\frac12\vol_{p-1}(\partial R^{p-1,1}),\mu_{p-1}(T_\zeta^{p-1,1})=\frac12\vol_{p-1}(\partial T_\zeta^{p-1,1})$, and
		\begin{align*}
		\frac{d}{d\theta}\vol (R^{p-1,1}(\theta))&=\frac{d\rho}{d\theta}\frac{d}{d\rho}\vol (R^{p-1,1}(\theta))=\frac{1+\tan^2\theta}{1-\tan^2\theta}\vol_{p-1}(\partial R^{p-1,1}) \\
		\frac{d}{d\theta}\vol (T_\zeta^{p-1,1}(\theta)) &=\frac{d\epsilon}{d\theta}\frac{d}{d\epsilon}\vol (T_\zeta^{p-1,1}(\theta))    =\sqrt\xi\frac{1+\tan^2\theta}{1+\xi\tan^2\theta}\vol_{p-1}(\partial T_\zeta^{p-1,1}),
		\end{align*}this case follows from the previous one.
		
		For $(k-p-1)$ positive and odd, since $N^*R^{p-1,1}$ is contained in the time-like orbit of the cosphere bundle of $S^{p-1,1}$, we have
		\begin{align} \label{eq:mu_k_expansion}
		\mu_k(R^{p-1,1}) & = \sum_\nu \mathbf i^{p-1-k-2\nu}c_{p,k,\nu} [[0,\phi_{k+2\nu,\nu}^-]](R^{p-1,1}) +\mathbf{i} d_{p,k} \vol(R^{p-1,1})\\
		\mu_k^\zeta(T_\zeta^{p-1,1}) & = \sum_\nu c_{p,k,\nu}  [[0,\phi_{k+2\nu,\nu}^\zeta]] (T_\zeta^{p-1,1}) +d_{p,k} \vol(T_\zeta^{p-1,1}),\label{eq:mu_k_expansion2}
		\end{align}
		for certain constants $c_{p,k,\nu },d_{p,k}$, where $\phi_{k,r}^-$ is the smooth form given in Lemma 5.1 of \cite{bernig_faifman_solanes} when $M=S^{p-1,1}$, and $\phi_{k,r}^\zeta$ is the form $\phi_{k,r}^+$  in the same lemma when $M=S_\zeta$. For $k-p-1\geq 0$ and even, equations \eqref{eq:mu_k_expansion}, \eqref{eq:mu_k_expansion2} hold with the volume term removed.   
		
		Since $S_\zeta$ and $S^{p-1,q}$ have constant curvature 1, we have $\phi_{k,r}^-=\phi_{k,0}^-$ and $\phi_{k,r}^\zeta=\phi_{k,0}^\zeta$. By the structure equations (see \cite[eqs. (32),(33)]{bernig_faifman_solanes}) we have 
		\begin{align*}
		d\phi_{k,0}^-&=\theta_0 \wedge (- k\phi_{k-1,0}^- -(p-1-k) \phi_{k+1,0}^-)\\
		d\phi_{k,0}^\zeta&=\theta_0 \wedge (k\phi_{k-1,0}^\zeta -(p-1-k) \phi_{k+1,0}^\zeta),
		\end{align*}
		where $\theta_0$ is the contact 1-form defined by the pseudo-Riemannian metric.
		
		Now take $M=S^{p-1,1}$ and assume $\omega\in\Omega^{\dim M-1}(\mathbb P_M)$, and $d\omega=\theta_0\wedge \omega'$. Let $\nu:\partial R^{p-1,1}(\theta)\to\mathbb P_M$ be the outer normal map, and extend it smoothly to $M$. We then have
		\begin{align*} 
		\frac{d}{d\theta}[[{ 0,}\omega]](R^{p-1,1}(\theta)) & =\frac{d}{d\theta} \left\langle \omega, \llbracket N^*R^{p-1,1}(\theta) \rrbracket \right\rangle \\
		& = \frac{d}{d\theta} \left\langle \nu^*\omega, \llbracket \partial R^{p-1,1}(\theta) \rrbracket \right\rangle\\
		&=\frac{d}{d\theta} \left\langle \nu^*\theta_0\wedge \nu^*\omega', \llbracket R^{p-1,1}(\theta) \rrbracket \right\rangle \\
		& =\left\langle \nu^*\theta_0\wedge \nu^*\omega', \frac{\partial}{\partial \theta}\cdot \llbracket \partial R^{p-1,1}(\theta) \rrbracket \right\rangle\\
		& = \left\langle  \nu^*\theta_0, \frac{\partial}{\partial \theta} \right\rangle \cdot \left\langle \nu^*\omega',  \llbracket \partial R^{p-1,1}(\theta) \rrbracket \right\rangle \\
		& =\frac{d\rho}{d\theta} \cdot [[0,\omega']] (R^{p-1,1}(\theta)).
		\end{align*}
		
		Hence $\frac{d}{d\theta}[[0,\phi^-_{k,0}]] (R^{p-1,1}(\theta))$ equals
		\begin{align*}
		\frac{1+\tan^2\theta}{1-\tan^2\theta}\left(- k[[0,\phi_{k-1,0}^-]] (R^{p-1,1}(\theta)) -(p-1-k) [[0,\phi_{k+1,0}^-]] (R^{p-1,1}(\theta))\right),\end{align*}
		and similarly $\frac{d}{d\theta}[[0,\phi_{k,0}^\zeta]] (T_\zeta^{p-1,1}(\theta))$ is
		\begin{align*}\sqrt\xi \frac{1+\tan^2\theta}{1+\xi\tan^2\theta} \left(k[[0,\phi_{k-1,0}^\zeta]](T_\zeta^{p-1,1}(\theta)) -(p-1-k) [[0,\phi_{k+1,0}^\zeta]] (T_\zeta^{p-1,1}(\theta))\right).
		\end{align*}

		It follows by induction on $k=p,\dots,0$ that $[[0,\phi_{k,0}^\zeta]](T_\zeta^{p-1,1}(\theta))$ is holomorphic in $\zeta\in U_\C$ and 
		\begin{equation}\label{eq:induction_analytic}
		\lim_{\zeta\to 0} [[0,\phi_{k,0}^\zeta]](T_{\zeta}^{p-1,1}(\theta))=\mathbf i^{p-1-k} [[0,\phi_{k,0}^-]](R^{p-1,1}(\theta)).
		\end{equation}
		By \eqref{eq:mu_k_expansion} and \eqref{eq:mu_k_expansion2} this completes the proof.
	\end{proof}

	In order to normalize the leading coefficient in the Crofton formulas we rescale the measures $m_k, \check m_k$ as follows.
	\begin{Definition} \label{def_normailzation}
		Let $M\subset \R^{p,q}$ be the pseudosphere of curvature $\sigma>0$, or the pseudohyperbolic space of curvature $\sigma<0$. We define 
 		\begin{displaymath}
 		 \Cr_k^M={\pi\omega_{k-1}}\sqrt{\sigma^{-1}}^k\Cr_M(m_k).
 		\end{displaymath}
		In the flat pseudo-Euclidean space $M=\R^{p,q}$ we take
		\begin{displaymath}
		 {\Cr_k^M={\pi\omega_{k-1}}\Cr_M(\check m_k).}
		\end{displaymath}
 	\end{Definition}
	
	\begin{Theorem}[Crofton formula]\label{thm_crofton_formula} Let $M$ be a pseudosphere, a pseudohyperbolic space or a pseudo-Euclidean space. Then, independently of the signature of $M$,
 		\begin{equation} \label{eq:crofton}
		{\Cr_k^M= \sum_{j=0}^{\lfloor\frac{n-k}{2}\rfloor}\frac {\omega_{k-1}}{\omega_{k+2j-1}} {-\frac k 2 \choose j}\sigma^{j} \mu_{k+2j}}
		\end{equation}
		where $\sigma$ is the sectional curvature of $M$ and $n$ its dimension.
	\end{Theorem}
	
	\begin{proof} Take first the pseudosphere $M=S^{p-1,q}$ of curvature $\sigma=1$. We can assume $q>0$ as the formula is known in $S^{n}$ (cf. e.g. \cite{fu_wannerer}). We know that
		\begin{displaymath}
		\Cr_k^{M}=\sum_{j=0}^{\lfloor\frac{n-k}2\rfloor} (a_{j,p,q}\mu_{k+2j}+b_{j,p,q}\overline{\mu_{k+2j}})
		\end{displaymath}
		for certain coefficients $a_{j,p,q},b_{j,p,q}\in\C$. Indeed, by \cite[Theorem C]{bernig_faifman_solanes_part2} we may express $\Cr_k^{M}$ as a linear combination of the intrinsic volumes and their complex conjugates. Since both $\mu_r$ and $\Cr_r^M$ are the restrictions of elements in $\Val_r^{-\infty,+}$ and thus belong to the $(-1)^r$-eigenspace of the Euler-Verdier involution, only the displayed terms appear.
		
		Let $e\colon S^{p-1,q}{\hookrightarrow}S^{p-1+l,q}$ and $\tilde e\colon S^{p-1+l,1}{\hookrightarrow}S^{p-1+l,q}$ be standard inclusions.  By Proposition \ref{prop:restrictions}, we have
		\begin{align*}
		\sum_{j=0}^{\lfloor\frac{p-1+q-k}2\rfloor} a_{j,p,q}\mu_{k+2j}+b_{j,p,q}\overline{\mu_{k+2j}} &= \Cr_k^{S^{p-1,q}}=e^*(\Cr_k^{S^{p-1+l,q}})\\
		&=\sum_{j=0}^{\lfloor\frac{p-1+q-k}2\rfloor} a_{j,p+l,q}\mu_{k+2j}+b_{j,p+l,q}\overline{\mu_{k+2j}}\\
		\sum_{j=0}^{\lfloor\frac{p+l-k}2\rfloor} a_{j,p+l,1}\mu_{k+2j}+b_{j,p+l,1}\overline{\mu_{k+2j}}&= \Cr_k^{S^{p-1+l,1}}=\tilde e^*(\Cr_k^{S^{p-1+l,q}})\\
		&=\sum_{j=0}^{\lfloor\frac{p+l-k}2\rfloor} a_{j,p+l,q}\mu_{k+2j}+b_{j,p+l,q}\overline{\mu_{k+2j}}.
		\end{align*}
		
		By the linear independence of $\{\mu_{i}\}_i\cup\{\overline\mu_i\}_i$ \cite[Corollary 7.4]{bernig_faifman_solanes}, and taking $l\geq q-1$, this yields 
		\begin{align*}a_{j,p,q}&=a_{j,p+l,q}=a_{j,p+l,1},\\
		b_{j,p,q}&=b_{j,p+l,q}=b_{j,p+l,1} 
		\end{align*}
		for all $j\leq \frac{p-1+q-k}2$.  
		
		It suffices then to determine $a_j:=a_{j,p,1},b_j:=b_{j,p,1}$; i.e. to prove the statement in the de Sitter sphere $M=S^{p-1,1}$. To this end we evaluate both sides on the templates $R^{p-1,s}\subset M$ with $s=0,1$. In order to compute $\Cr_k^M(R^{p-1,s})$ we use the spherical Crofton formula:
		\begin{equation}\label{eq_crofton_zeta}
		\Cr_{S_\zeta}(\pi\omega_{k-1}m_k^\zeta)= \sum_{j\geq 0}\frac{\omega_{k-1}} {\omega_{k+2j-1}} {-\frac k 2 \choose j}\mu^{\zeta}_{k+2j}=:\sum_{j\geq 0}c_{j}\mu^{\zeta}_{k+2j},
		\end{equation}for $\zeta>\frac12$.
		Given $p\geq k$ and $s=0,1$, let $S^{p}$ be the unit sphere of an arbitrary Euclidean structure in $\R^{p,1}=\R^{p+1}$ and let $T^{p-1,s}$ be the radial projection on $S^{p}$ of $R^{p-1,s}$.  By Definition \ref{def_normailzation}, Proposition   \ref{prop:weak_continuity}, and applying analytic continuation to \eqref{eq_crofton_zeta} via Proposition \ref{prop:continuation_mu},  we have
		\begin{align*}
		\Cr_k^M(R^{p-1,s}) &={\pi \omega_{k-1}}\lim_{\epsilon\to 0^+} \Cr_{S^p}(m_k^{\mathbf{i}\epsilon})(T^{p-1,s})\\
		&= \lim_{\epsilon\to 0^+} \sum_{j\geq 0} c_j  f_{k+2j,s}(\mathbf{i}\epsilon)\\
		&= \sum_{j\geq 0} c_j \mu_{k+2j}(R^{p-1,s}).
		\end{align*}
		
		Now, for  $s=0,1$, taking $p=k+2l-s+1$ we get 
		\begin{align*}
		 \sum_{j\geq 0} c_j \mu_{k+2j}(R^{k+2l-s,s})&=\sum_{j\geq 0} a_j \mu_{k+2j}(R^{k+2l-s,s})+b_j \overline{\mu_{k+2j}(R^{k+2l-s,s})}\\
		&=\sum_{j\geq 0} (a_j+(-1)^{s}b_j) \mu_{k+2j}(R^{k+2l-s,s}),
		\end{align*}
		since $\mu_{k+2j}(R^{k+2l-s,s})\in \mathbf i^s\R$ by \eqref{eq:mu_k_expansion}.
		
		For $l=0$ we have $\mu_{k+2j}(R^{k -s,s})=0$ for all $j \geq 1$ and hence $ c_0=a_0+(-1)^{s}b_0$ for  $s=0,1$, and thus $a_0= c_0,b_0=0$. 
		Suppose that $a_j= c_j, b_j=0$ for all $j< j_0$. Taking $l=j_0$ we deduce
		$ c_{j_0}=a_{j_0}+(-1)^{s}b_{j_0}$  for  $s=0,1$. By induction we deduce $a_j= c_j$ and $b_j=0$ for all $j$, which completes the proof for $\sigma=1$.
		
		For $\sigma>0$ the theorem follows by the homogeneity of the $\mu_k$ (cf. \cite[Proposition 1.2. iii)]{bernig_faifman_solanes}) .

		Let us now turn to $\sigma=-1$, i.e. to $H^{p,q-1}\subset \R^{p,q}$. Note that the anti-isometry $j\colon \R^{q,p}\to \R^{p,q}$  of Lemma \ref{lem:crofton_sign} maps $S^{q-1,p}$ to $H^{p,q-1}$.  Therefore, by Lemma \ref{lem:crofton_sign} and the homogeneity of the $\mu_k$,
		\begin{align*}
		\Cr_k^{H^{p,q-1}}(j(A)) &=\pi\omega_{k-1}\mathbf i^k\int_{\Gr_{n+1-k}} \chi(E\cap j(A)) {dm_k(E)} \\
		&=\pi\omega_{k-1}\mathbf i^k\int_{\Gr_{n+1-k}} \chi(E\cap A) dj^*m_k(E)\\
		&=\mathbf i^k\overline{\Cr_k^{S^{{q-1,p} }}(A)}\\
		&=\mathbf i^k\sum_\nu c_{\nu} \overline{\mu_{k+2\nu}(A)}\\
		&=\mathbf i^k\sum_\nu c_{\nu} \mathbf i^{-k-2\nu}{\mu_{k+2\nu}(j(A))}.
		\end{align*}
		This proves the statement for $\sigma=-1$. The case $\sigma<0$ follows as before from the homogeneity of the $\mu_k$.
				
		Finally we consider the case $\sigma=0$. Let us identify $M=\R^{p-1,q}$ with the tangent space $T_x S^{p-1,q}$ at some $x\in S^{p-1,q}$.
		Let $\Lambda_k^x\colon \mathcal V^{-\infty}(S^{p-1,q})^{\OO(p,q)}\to \Val^{-\infty}(T_x S^{p-1,q})^{O(p-1,q)}$ be given by (cf. \cite[Proposition 3.1.5]{alesker_val_man1})
		\begin{displaymath}
		\Lambda_k^x(\varphi)=\frac{1}{k!}\left. \frac{d^k}{dt^k}\right|_{t=0} h_t^*\phi^*\varphi,
		\end{displaymath}
		where $\phi\colon U\subset T_x S^{p-1,q}\to S^{p-1,q}$ is defined on a neighborhood of $x$ by $$\phi(w)=Q(x+w)^{-\frac12}(x+w)$$ and $h_t(w)=tw$. By Proposition \ref{prop:derivative_crofton}  we have
		\begin{displaymath}
		\Lambda_k^x \Cr_k^{S^{p-1,q}}=\Cr_k^{\R^{p-1,q}}. 
		\end{displaymath}
		
		On the other hand, denoting by $g$  the metric on $S^{p-1,q}$, since $\mu_k\in\mathcal W_k^{-\infty}$, behaves naturally with respect to isometries and is $k$-homogeneous, we have
		\begin{align*}
		\Lambda_k^x \mu_k^g&= \lim_{t\to 0} t^{-k}(\phi\circ h_t)^*\mu_k^g=\lim_{t\to 0} \mu_k^{(\phi\circ h_t)^*g/t^2}.
		\end{align*}
		Since $(\phi\circ h_t)^*g/t^2$ converges, $C^\infty$-uniformly on compact sets, to the flat metric $g_0$, we conclude by \cite[Proposition 1.2 ii)]{bernig_faifman_solanes} that
		\begin{displaymath}
		\Lambda_k^x \mu_k^g =\mu_k^{g_0}.
		\end{displaymath}
		
		Applying $\Lambda_k^x$ to both sides of \eqref{eq:crofton} the case $\sigma=0$ follows.
	\end{proof}
	
	Recall from \cite{bernig_faifman_solanes} that the intrinsic volumes $\mu_k$  were defined  in terms of certain generalized curvature measures $C_{k,p}^0,C_{k,p}^1$. On a manifold of constant curvature $\sigma$, these fulfill $C_{k,p}^i=\sigma^p C_{k,0}^i$. Using this and \cite[Eq. (61)]{bernig_faifman_solanes}, the Crofton formula \eqref{eq:crofton} becomes
	\begin{displaymath}
	 \Cr_k^M =\mathbf i^q \sum_j d_{k,j}\sigma^j \glob(C_{k+2j,0}^0+\mathbf i C_{k+2j,0}^1),
	\end{displaymath}where $\glob:\mathcal C^{-\infty}(M)\to \mathcal V^{-\infty}(M)$ is the globalization map (cf. \cite[Section 2]{bernig_faifman_solanes}). The constants $d_{k,j}$ are independent of the signature and the curvature and can thus be deduced from the case of Euclidean spheres. Therefore, by \cite[\S3.2]{fu_wannerer} we obtain
	\begin{equation}\label{eq:Cr_to_C}
	 \Cr_k^M =\frac{\pi^k}{k!\omega_k}\mathbf i^q \sum_j \left(\frac{\sigma}4\right)^j \glob(C_{k+2j,0}^0+\mathbf i C_{k+2j,0}^1).
	\end{equation}
	
	\begin{Remark}
	It is interesting to note that \eqref{eq:Cr_to_C} yields
	\begin{displaymath}
	  \chi -\frac{\sigma}{2\pi} \Cr_2^M = \mathbf i^q \glob(C_{0,0}^0 + \mathbf i C_{0,0}^1),
	 \end{displaymath}
	 which can be seen as a generalization of the fact that the angular excess of a spherical triangle is proportional to its area.
	\end{Remark}

	\def\cprime{$'$}


\begin{thebibliography}{10}

\bibitem{akhiezer_kazarnovskii}
Dmitri Akhiezer and Boris Kazarnovskii.
\newblock Average number of zeros and mixed symplectic volume of {F}insler
  sets.
\newblock {\em Geom. Funct. Anal.}, 28(6):1517--1547, 2018.

\bibitem{alesker_mcullenconj01}
Semyon Alesker.
\newblock {Description of translation invariant valuations on convex sets with
  solution of P. McMullen's conjecture.}
\newblock {\em Geom. Funct. Anal.}, 11(2):244--272, 2001.

\bibitem{alesker03_un}
Semyon Alesker.
\newblock Hard {L}efschetz theorem for valuations, complex integral geometry,
  and unitarily invariant valuations.
\newblock {\em J. Differential Geom.}, 63(1):63--95, 2003.

\bibitem{alesker_val_man1}
Semyon Alesker.
\newblock Theory of valuations on manifolds. {I}. {L}inear spaces.
\newblock {\em Israel J. Math.}, 156:311--339, 2006.

\bibitem{alesker_val_man2}
Semyon Alesker.
\newblock Theory of valuations on manifolds. {II}.
\newblock {\em Adv. Math.}, 207(1):420--454, 2006.

\bibitem{alesker_val_man4}
Semyon Alesker.
\newblock Theory of valuations on manifolds. {IV}. {N}ew properties of the
  multiplicative structure.
\newblock In {\em Geometric aspects of functional analysis}, volume 1910 of
  {\em Lecture Notes in Math.}, pages 1--44. Springer, Berlin, 2007.

\bibitem{alesker_intgeo}
Semyon Alesker.
\newblock Valuations on manifolds and integral geometry.
\newblock {\em Geom. Funct. Anal.}, 20(5):1073--1143, 2010.

\bibitem{alesker_bernig}
Semyon Alesker and Andreas Bernig.
\newblock {The product on smooth and generalized valuations}.
\newblock {\em American J. Math.}, 134:507--560, 2012.

\bibitem{alesker_bernstein04}
Semyon Alesker and Joseph Bernstein.
\newblock Range characterization of the cosine transform on higher
  {G}rassmannians.
\newblock {\em Adv. Math.}, 184(2):367--379, 2004.

\bibitem{alvarez_fernandes}
J.~C. {\'A}lvarez~Paiva and E.~Fernandes.
\newblock Gelfand transforms and {C}rofton formulas.
\newblock {\em Selecta Math. (N.S.)}, 13(3):369--390, 2007.

\bibitem{bernig07}
Andreas Bernig.
\newblock Valuations with {C}rofton formula and {F}insler geometry.
\newblock {\em Adv. Math.}, 210(2):733--753, 2007.

\bibitem{bernig_broecker07}
Andreas Bernig and Ludwig Br{\"o}cker.
\newblock {Valuations on manifolds and Rumin cohomology.}
\newblock {\em J. Differ. Geom.}, 75(3):433--457, 2007.

\bibitem{bernig_faifman_opq}
Andreas Bernig and Dmitry Faifman.
\newblock Valuation theory of indefinite orthogonal groups.
\newblock {\em J. Funct. Anal.}, 273(6):2167--2247, 2017.

\bibitem{bernig_faifman_solanes}
Andreas Bernig, Dmitry Faifman, and Gil Solanes.
\newblock {Curvature measures on pseudo-Riemannian manifolds}.
\newblock To appear in J. Reine Angewandte Mathematik.

\bibitem{bernig_faifman_solanes_part2}
Andreas Bernig, Dmitry Faifman, and Gil Solanes.
\newblock Uniqueness of curvature measures in pseudo-{R}iemannian geometry.
\newblock {\em J. Geom. Anal.}, 31(12):11819--11848, 2021.

\bibitem{bernig_fu_solanes}
Andreas Bernig, Joseph H.~G. Fu, and Gil Solanes.
\newblock Integral geometry of complex space forms.
\newblock {\em Geom. Funct. Anal.}, 24(2):403--492, 2014.

\bibitem{birman84}
Graciela~S. Birman.
\newblock Crofton's and {P}oincar\'{e}'s formulas in the {L}orentzian plane.
\newblock {\em Geom. Dedicata}, 15(4):399--411, 1984.

\bibitem{boman67}
Jan Boman.
\newblock Differentiability of a function and of its compositions with
  functions of one variable.
\newblock {\em Math. Scand.}, 20:249--268, 1967.

\bibitem{brouder_dang_helein}
Christian Brouder, Nguyen~Viet Dang, and Fr\'{e}d\'{e}ric H\'{e}lein.
\newblock Continuity of the fundamental operations on distributions having a
  specified wave front set (with a counterexample by {S}emyon {A}lesker).
\newblock {\em Studia Math.}, 232(3):201--226, 2016.

\bibitem{croke84}
Christopher~B. Croke.
\newblock A sharp four-dimensional isoperimetric inequality.
\newblock {\em Comment. Math. Helv.}, 59(2):187--192, 1984.

\bibitem{brouder_dabrowski}
Yoann Dabrowski and Christian Brouder.
\newblock Functional properties of {H}\"{o}rmander's space of distributions
  having a specified wavefront set.
\newblock {\em Comm. Math. Phys.}, 332(3):1345--1380, 2014.

\bibitem{duistermaat_book96}
J.~J. Duistermaat.
\newblock {\em Fourier integral operators}.
\newblock Modern Birkh\"{a}user Classics. Birkh\"{a}user/Springer, New York,
  2011.
\newblock Reprint of the 1996 edition [MR1362544], based on the original
  lecture notes published in 1973 [MR0451313].

\bibitem{faifman_crofton}
Dmitry Faifman.
\newblock Crofton formulas and indefinite signature.
\newblock {\em Geom. Funct. Anal.}, 27(3):489--540, 2017.

\bibitem{fu_alesker_product}
Joseph H.~G. Fu.
\newblock Intersection theory and the {A}lesker product.
\newblock {\em Indiana Univ. Math. J.}, 65(4):1347--1371, 2016.

\bibitem{fu_wannerer}
Joseph H.~G. Fu and Thomas Wannerer.
\newblock Riemannian curvature measures.
\newblock {\em Geom. Funct. Anal.}, 29(2):343--381, 2019.

\bibitem{fu_barcelona}
Joseph~H.G. Fu.
\newblock Algebraic integral geometry.
\newblock In Eduardo Gallego and Gil Solanes, editors, {\em Integral Geometry
  and Valuations}, Advanced Courses in Mathematics - CRM Barcelona, pages
  47--112. Springer Basel, 2014.

\bibitem{garding48}
Lars G\aa{}rding.
\newblock Extension of a formula by {C}ayley to symmetric determinants.
\newblock {\em Proc. Edinburgh Math. Soc. (2)}, 8:73--75, 1948.

\bibitem{guillemin_sternberg}
Victor Guillemin and Shlomo Sternberg.
\newblock {\em Geometric asymptotics}.
\newblock Mathematical Surveys, No. 14. American Mathematical Society,
  Providence, R.I., 1977.

\bibitem{hoermander_pde1}
Lars H{\"o}rmander.
\newblock {\em The analysis of linear partial differential operators. {I}}.
\newblock Classics in Mathematics. Springer-Verlag, Berlin, 2003.
\newblock Distribution theory and Fourier analysis, Reprint of the second
  (1990) edition [Springer, Berlin; MR1065993 (91m:35001a)].

\bibitem{Kiderlen_Vedel_Jensen}
Markus Kiderlen and Eva B. Vedel~Jensen (eds.).
\newblock {\em Tensor Valuations and Their Applications in Stochastic Geometry
  and Imaging}.
\newblock Springer Lecture notes in mathematics, 2017.

\bibitem{klain_rota}
Daniel~A. Klain and Gian-Carlo Rota.
\newblock {\em Introduction to geometric probability}.
\newblock Lezioni Lincee. [Lincei Lectures]. Cambridge University Press,
  Cambridge, 1997.

\bibitem{langevin_chaves_bianconi03}
R\'{e}mi Langevin, Rosa Maria~Barreiro Chaves, and Ricardo Bianconi.
\newblock Formulas of {C}auchy and {C}rofton in {L}orentz-{M}inkowski and de
  {S}itter spaces.
\newblock volume~41, pages 99--113 (2003). 2001/02.
\newblock Homage to Luis Santal\'{o}. Vol. 1 (Spanish).

\bibitem{muro99}
Masakazu Muro.
\newblock Singular invariant hyperfunctions on the space of real symmetric
  matrices.
\newblock {\em Tohoku Math. J. (2)}, 51(3):329--364, 1999.

\bibitem{oh90}
Yong-Geun Oh.
\newblock Second variation and stabilities of minimal {L}agrangian submanifolds
  in {K}\"{a}hler manifolds.
\newblock {\em Invent. Math.}, 101(2):501--519, 1990.

\bibitem{oneill}
Barrett O'Neill.
\newblock {\em Semi-{R}iemannian geometry}, volume 103 of {\em Pure and Applied
  Mathematics}.
\newblock Academic Press, Inc. [Harcourt Brace Jovanovich, Publishers], New
  York, 1983.
\newblock With applications to relativity.

\bibitem{schneider_book14}
Rolf Schneider.
\newblock {\em Convex bodies: the {B}runn-{M}inkowski theory}, volume 151 of
  {\em Encyclopedia of Mathematics and its Applications}.
\newblock Cambridge University Press, Cambridge, expanded edition, 2014.

\bibitem{solanes_teufel}
Gil Solanes and Eberhard Teufel.
\newblock Integral geometry in constant curvature {L}orentz spaces.
\newblock {\em Manuscripta Math.}, 118(4):411--423, 2005.

\bibitem{steiner}
Jakob Steiner.
\newblock {\"U}ber parallele {F}l{\"a}chen.
\newblock {\em Monatsber. Preuß. Akad. Wiss.}, pages 114--118, 1840.
\newblock Ges. Werke, vol. 2, pp. 171--176, Reimer, Berlin, 1882.

\bibitem{teufel}
Eberhard Teufel.
\newblock Kinematische {B}er\"{u}hrung im \"{A}quiaffinen.
\newblock {\em Geom. Dedicata}, 33(3):317--323, 1990.

\bibitem{treibergs85}
Andrejs Treibergs.
\newblock Estimates of volume by the length of shortest closed geodesics on a
  convex hypersurface.
\newblock {\em Invent. Math.}, 80(3):481--488, 1985.

\bibitem{vladimirov66}
Vasili\u{\i}~Sergeevi\v{c} Vladimirov.
\newblock {\em Methods of the theory of functions of many complex variables}.
\newblock Translated from the Russian by Scripta Technica, Inc. Translation
  edited by Leon Ehrenpreis. The M.I.T. Press, Cambridge, Mass.-London, 1966.

\bibitem{weyl_tubes}
Hermann Weyl.
\newblock On the {V}olume of {T}ubes.
\newblock {\em Amer. J. Math.}, 61(2):461--472, 1939.

\bibitem{wolf61}
Joseph Wolf.
\newblock Homogeneous manifolds of constant curvature.
\newblock {\em Comment. Math. Helv.}, 36:112--147, 1961.

\bibitem{ye_ma_wang16}
Nan Ye, Xiang Ma, and Donghao Wang.
\newblock The {F}enchel-type inequality in the 3-dimensional {L}orentz space
  and a {C}rofton formula.
\newblock {\em Ann. Global Anal. Geom.}, 50(3):249--259, 2016.

\end{thebibliography}

\end{document}